\documentclass{amsart}
\usepackage[utf8]{inputenc}
\usepackage[T1]{fontenc}
\usepackage{textcomp}

\usepackage[a4paper]{geometry}
\usepackage[utf8]{inputenc}
\usepackage[OT2, T1]{fontenc}
\usepackage{babel}
\usepackage{stmaryrd}
\SetSymbolFont{stmry}{bold}{U}{stmry}{m}{n}

\usepackage{amsmath}
\usepackage{amssymb}
\usepackage{amsthm}
\usepackage{stmaryrd}
\usepackage{mathtools}
\usepackage[bb=boondox]{mathalfa}
\usepackage{dsfont}
\usepackage{tikz-cd}
\usepackage{mathrsfs}
\usepackage{tikz}
\usepackage{tikz}

\usepackage{enumitem}

\usepackage{array}
\usepackage{graphicx}

\usepackage{natbib}
    
\numberwithin{equation}{section}

\theoremstyle{plain}
\newtheorem{thm}{Theorem}[section]

\newtheorem{lem}[thm]{Lemma}
\newtheorem{prop}[thm]{Proposition}
\newtheorem{coro}[thm]{Corollary}

\theoremstyle{definition}

\newtheorem{defi}[thm]{Definition}

\newtheorem{rem}[thm]{Remark}

\usepackage{hyperref}
\usepackage[capitalize,nameinlink,noabbrev,sort&compress]{cleveref}

\crefname{thm}{Theorem}{Theorems}
\crefname{lem}{Lemma}{Lemmas}
\crefname{prop}{Proposition}{Propositions}
\crefname{coro}{Corollary}{Corollaries}
\crefname{defi}{Definition}{Definitions}
\crefname{rem}{Remark}{Remarks}
\crefname{ex}{Example}{Examples}

\newcommand\restr[2]{%
  \left.\kern-\nulldelimiterspace
  #1
  \vphantom{\big|}
  \right|_{#2}
}

\usepackage{upgreek}
\usepackage{mathrsfs}

\usepackage{indentfirst}

\makeatletter
\renewcommand\paragraph{\@startsection{paragraph}{4}%
  \z@{3pt}{1pt}%
  {\normalfont\itshape\small}}
\makeatother

\title{ Asymptotic Expansion for Expanding Spherical Averages in Real Rank One Spaces}
\author{Zhiyuan Deng}
\address{Institute of Mathematics,
University of Zurich,
Switzerland}

\email{zhiyuan.deng@math.uzh.ch}

\thanks{
Zhiyuan Deng was supported by the Swiss National Science Foundation (SNF)
under Grant No.~200020-212617.
}

\author{Yutian Sun}

\address{School of Mathematics, Shandong University 
Jinan 250100, Shandong, P. R. China }
\email{yutian.sun@mail.sdu.edu.cn}
\begin{document}

\begin{abstract}
This paper establishes an asymptotic expansion for expanding sector average on compact quotients of real rank-one Lie groups, with $SO(n,1)$ as a model case. Using tools from harmonic analysis and representation theory, the problem is reduced, via the action of the Casimir operator, to an ordinary differential equation(ODE) governing the evolution along the expanding $A$-direction. The asymptotic expansion is then derived from the analysis of this ODE.
\end{abstract}

\maketitle
\tableofcontents
\newpage

\section{Introduction}

One of the central themes in homogeneous dynamics is to understand how expanding submanifolds distribute under the action of flows.  
One asks whether such submanifolds become uniformly distributed in the ambient space, and if so, at what rate.  
The study of these equidistribution properties has been one of the major achievements in ergodic theory and dynamics system.

In this paper we work in the real rank-one setting.
Let $G$ be a connected semisimple Lie group of real rank one with finite center, and fix a maximal compact subgroup $K \subset G$.
We use the Iwasawa decomposition $G = KAN$, where $A = \{ a_t : t \in \mathbb{R} \}$ is a one-parameter Cartan subgroup. Here $N$ denotes the nilpotent subgroup corresponding to the sum of the positive restricted root spaces with respect to the fixed choice of the maximal abelian subspace $\mathfrak{a}\subset \mathfrak{g}$ and a choice of positive restricted roots $\Sigma^+ \subset \Sigma(\mathfrak{g},\mathfrak{a})$, namely $N = \exp(\mathfrak{n})$, where $\mathfrak{n} = \sum_{\alpha \in \Sigma^+} \mathfrak{g}_\alpha$.
The associated minimal parabolic subgroup is denoted by $P = MAN$, where the Levi component $M = Z_K(A)$ is the centralizer of $A$ in $K$.
Throughout the paper we write $X := \Gamma \backslash G$ for the compact homogeneous space associated with a cocompact lattice $\Gamma < G$.

For concreteness, we focus on the model case $G = \mathrm{SO}(n,1)^{\circ}$, acting by isometries on real hyperbolic $n$-space. 
In this realization, $K = \mathrm{SO}(n)$ and $G/K \simeq \mathbb{H}^n$. 
The element $a_t \in A$ acts on $\mathbb{H}^n$ as translations by hyperbolic distance $t$ along geodesics.

Let $(\pi,H)$ be a unitary representation of $G$. Given $v \in H$ and a density function $\varphi \in C^{\infty}(K)$, we define the sector average
\begin{equation*}
I(t,v,\varphi)
:=
\int_{K} \varphi(k)\,\pi(k a_{t})\,v \, dk ,
\end{equation*}
which is well defined since $k \mapsto \pi(k a_{t})v$ is continuous on the compact group $K$.

In the particular case of the regular representation on $L^{2}(X)$, defined by
$\pi(g)f(x) = f(xg)$, the averaging operator reduces to the standard pointwise
$K$-average along the $A$-orbit in $X$. 
More precisely, for $f \in L^{2}(X)$, the function
\begin{equation}\label{equation:pointwise_average}
I(t,f,\varphi)(x)
=
\int_{K} \varphi(k)\, f(x k a_{t}) \, dk
\end{equation}
is well defined for almost every $x \in X$.

The main topic of this paper is to discuss the asymptotic expansion of such average $I(t,f,\varphi)(x)$ as $t\to \infty$.  

Those type of averages are closely related to classical problems on the distribution of expanding orbits in homogeneous spaces. Early proofs often relied on Margulis’ thickening trick combined with mixing of the geodesic flow~\cite{Margulis2004}.  Later Eskin and McMullen~\cite{EskinMcMullenmixingcouting} proved equidistribution of certain symmetric-subgroup orbits in hyperbolic spaces by exploiting the decay of matrix coefficients.   The thickening method introduces an unavoidable error, so it cannot yield the sharp error estimate for the original orbit.

Later work showed that spectral and representation-theoretic methods control equidistribution directly through the spectral gap of the $G$-action on $L^{2}(\Gamma\backslash G)$, making effective rates accessible. 

We introduce a simplified version of the estimate from Burger for equidistribution of horocycle orbit~\cite{Burgerhorocyclegeofinite}.
In the model case $G=\mathrm{SL}(2,\mathbb{R})$, fix a cocompact lattice $\Gamma<G$ and write
$X:=\Gamma\backslash G$. Let
\begin{equation*}
a_t:=
\begin{pmatrix}
e^{t/2} & 0\\
0 & e^{-t/2}
\end{pmatrix},
\qquad
u_s:=
\begin{pmatrix}
1 & s\\
0 & 1
\end{pmatrix},
\qquad
U:=\{u_s:s\in\mathbb{R}\}.
\end{equation*}
For $f\in C^{\infty}(X)$, $x\in X$, and a compactly supported smooth function
$\varphi\in C_c^{\infty}(\mathbb{R})$, define the weighted expanding horocycle average
\begin{equation*}
H(t,f,\varphi)(x):=\int_{\mathbb{R}} \varphi(s)\,f(xu_s a_t)\,ds.
\end{equation*}
A representative Burger-type effective estimate~\cite{Burgerhorocyclegeofinite} has the form
\begin{equation*}
H(t,f,\varphi)(x)
=
\left(\int_{\mathbb{R}}\varphi(s)\,ds\right)\left(\int_X f\,d\mu_X\right)
+O_{x,f,\varphi}(e^{-\eta t}),
\qquad t\to\infty,
\end{equation*}
where $\mu_X$ is the $G$-invariant probability measure on $X$, $\eta>0$ is determined by spectral data, and the implied constant in $O_{x,f,\varphi}(\cdot)$ depends on $x$, $f$, and $\varphi$. Thus one obtains effective equidistribution with a single decay-rate error term, but not a full higher-order asymptotic expansion.

Strömbergsson~\cite{Strombergsson2013DeviationHorocycle} refined this framework for cofinite lattices in $\mathrm{SL}(2,\mathbb{R})$ by incorporating the continuous spectrum, leading to sharper effective bounds for expanding horocycle averages. Edwards~\cite{edwardsrateofexpandinghorospheres} generalized this approach to expanding horospheres in semisimple groups with finite-volume quotient, obtaining effective rates controlled by exponential mixing and Sobolev norms. Analogous computations appear to extend to symmetric-subgroup orbits, as indicated in an unfinished sketch available online~\cite{edwards2018symmetricsubgroup}.
Lutsko~\cite{Lutsko2022SpectralHorospherical} also developed an abstract spectral approach to horospherical equidistribution with explicit error terms. 

More recently, Corso and Ravotti~\cite{Davidehorocycle2020,Corso2022LargeHC} (see also Flaminio and Forni~\cite{Flaminioforni})
studied expanding horocycle orbits and expanding hyperbolic circles in
$\mathrm{SL}(2,\mathbb{R})$, obtaining an asymptotic expansion together with quantitative bounds of the remainder term via spectral decomposition. 
In contrast, due to the new technical difficulties occurring in higher dimension, 
 our work requires a complex iterative scheme discussed in Section~\ref{Section:Asymp_Irreducible_repnt_with_pure_imginary}, which also  produces a fully explicit asymptotic expansion for the sector average~\eqref{equation:pointwise_average}.
 
\medskip
\noindent
\textbf{Main Result}
 We present below a simplified version of the result to suppress technical notation. The full statement is given in Theorem~\ref{theorem:main_vector}.

Denote $\|\cdot\|_{W_{s,2}}$ as the Sobolev norm of $f$ of order $s$ with respect to $L^2$-norm on $X$, more details in Section~\ref{sec:sobolev}. We define \begin{equation}\label{equation:L^2orthogonal_constant}
    L^{2}_{0}(X)
:=
\left\{ f \in L^2(X) : \langle f,1\rangle = 0 \right\},
\end{equation}
and $\nu(\Gamma)>0$ is the spectral gap parameter associated with the
right regular representation of $G$ on $L^{2}_{0}(X)$, see more details in Section~\ref{representationtheory}; it
governs the exponential decay rate of matrix coefficients for this
representation.
\begin{thm}[Pointwise asymptotic expansion]
There exists a countable family of complex numbers
$\{\lambda_j\}_{j\ge1}$ satisfying $\Re(\lambda_j)
\le
\frac{1-n}{2}+\nu(\Gamma),\ j\ge1.$

Let $\ell>\frac{n-1}{2}+\nu(\Gamma)$ be an integer and let
$\varphi\in C^\infty(K)$ be a density function.
Then there exists an integer $\beta_\ell>0$ such that, for every
$f\in W_{\beta_\ell,2}(X)$, the sector average
\begin{equation*}
I(t,f,\varphi)(x)
=
\int_K \varphi(k)\,f(xka_t)\,dk
\end{equation*}
admits the following pointwise asymptotic expansion:
\begin{equation*}
I(t,f,\varphi)(x)
=
\sum_{j\ge1}
\sum_{m=0}^{2\ell}
\bigl(C_{j,m}(f,\varphi)(x)+ t\,C^{p}_{j,m}(f,\varphi)(x)\bigr)
e^{(\lambda_j-m)t}
\;+\;
\mathcal{R}_\ell(t,f,\varphi)(x),
\end{equation*}
where the coefficient functions
$C_{j,m}(f,\varphi)$ and $C_{j,m}^{p}(f,\varphi)$
belong to $W_{\beta_\ell,2}(X)$, and the double series over $j$ and $m$ converges uniformly for each
$x\in X$.
Moreover, the remainder term satisfies the uniform bound:
\begin{equation*}
\sup_{x\in X}\bigl|\mathcal{R}_\ell(t,f,\varphi)(x)\bigr|
\;\ll_{\Gamma,n,\ell}\;
(1+t)\,e^{(1-\ell) t}\,
\|f\|_{W_{\beta_\ell,2}}\,
\|\varphi\|_{W_{\beta_\ell}},
\qquad t\ge0.
\end{equation*}
\end{thm}

\medskip
\noindent
\textbf{Method of Proof}
Our main idea is to use the Casimir operator $\Omega_{G}$ of $G$ to derive a second-order ODE satisfied by the sector average $I(t,v,\varphi)$, where $v$ is an eigenvector of $\Omega_{G}$. Then after a spectral summation, we obtain the result for the general unitary representation. 

To construct this ODE in a form that can be solved explicitly, we employ the expression of $\Omega_{G}$ in Iwasawa coordinates. The Iwasawa coordinates formula for $\Omega_{G}$ that we compute in Section~\ref{section:casimiroperator_generalliegroup} is valid for any connected semisimple Lie group with finite center. 
As a consequence, the resulting equation involves both the Casimir operator $\Omega_{G}$ and  $\Omega_{M}$ for the subgroup $M$. Therefore, we require the vector $v$ to be a joint eigenvector of $\Omega_{G}$ and $\Omega_{M}$. Accordingly, it is necessary to decompose the unitary representation into $G$- and $M$-irreducible representations.

Later, we  specialize to $G = \mathrm{SO}(n,1)^{\circ}$ with $n > 2$, since for $n = 2$ the Levi factor is $M = \{\pm I\}$.  
The corresponding asymptotic analysis for sector average in the case $G = \mathrm{SL}(2,\mathbb{R})$ has already been carried out by Corso-Ravotti~\cite{Corso2022LargeHC}.

 Our method proceeds by resolving the $K$-average into $M$-types and analyzing each component through the joint eigenvalues of the Casimir operators $(\Omega_G,\Omega_M)$. This yields, for each spectral component, a second-order ODE with explicitly determined coefficients, allowing the individual spectral contributions to be identified and assembled into a full asymptotic expansion.
A crucial additional ingredient in our approach is the use of Weyl-type counting estimates for the joint spectrum. These provide quantitative control on the growth of spectral parameters and multiplicities, which is necessary to ensure convergence of the resulting series and to obtain uniform bounds for the remainder terms.

Compared with the approach of Edwards~\cite{edwards2018symmetricsubgroup,edwardsrateofexpandinghorospheres,SamuelC.Edwards2017JournalofModernDynamics}, his work is based on a global representation-theoretic analysis that yields effective equidistribution with explicit optimal error terms, but without an explicit term-by-term reconstruction of the full spectral expansion. For that reason, his method does not require Weyl-law input in the same explicit form as ours, and it leads to an effective estimate rather than a full asymptotic expansion.

For semisimple groups of higher real rank, it would be natural to extend the present analysis to the action of higher-dimensional split tori $A$, going beyond the one-parameter geodesic flow. We plan to pursue this direction in future work.

\medskip
\noindent
\textbf{Outline of This Paper}
Section~\ref{section:casimiroperator_generalliegroup} computes the explicit expression of Casimir operator for general connected Lie group with finite centre in Iwasawa coordinates.
Section~\ref{Section:perliminaries} develops the harmonic–analytic framework on $\mathrm{SO}(n,1)$, including its Lie algebra structure, representation theory, and a Weyl-type upper bound.  Using the Iwasawa coordinates expression of Casimir operator,
Section~\ref{section:construction_of_ODE} derives an ODE governing the sector average $I(t,v,\varphi)$, together with its solutions and estimates for the Sobolev norms of $I(t,v,\varphi)$.
Section~\ref{Section:Asymp_Irreducible_repnt_with_pure_imginary} derives explicit asymptotic expansions of $I(t,v,\varphi)$ for any irreducible representation using an iteration scheme.  
Section~\ref{sec:global_asymptotics} combines the preceding results to derive the full asymptotic expansion for an arbitrary vector $v$ arising from a general unitary representation. As a consequence, one obtains a uniform pointwise asymptotic expansion for a corresponding class of functions $f$ on $X$.

\section{Expression of the Casimir Operator for Semisimple Connected Lie Group}\label{section:casimiroperator_generalliegroup}
This section is devoted to deriving the explicit expression in Iwasawa coordinates of the Casimir operator ${\Omega_
{{G}}}$ for a semisimple connected Lie group $G$. The explicit expression of the Casimir operator ${\Omega_{G}}$ in Iwasawa coordinates on $G$ does not appear in standard references. Classical formulas describe only the radial part of $\Omega_{G}$ on $G/K$ under $K$–bi-invariance. In contrast, our computation retains all terms of the Casimir operator ${\Omega_{G}}$ on $C^{\infty}(G)$. This full expression is essential for the analysis of sector average, since it involves nontrivial $K$-types and therefore is not captured by the radial part alone.

Consider a connected semisimple Lie group $G$ with finite center and a fixed maximal compact subgroup $K \subset G$ with Lie algebra $\mathfrak{k}$. Denote by $\mathfrak{g}$ the Lie algebra of $G$, and write $\theta$ for the Cartan involution associated with $K$. This yields the Cartan decomposition $\mathfrak{g} = \mathfrak{k} \oplus \mathfrak{p}.$
Choose a Cartan subgroup $A$ compatible with $K$. The Killing form on $\mathfrak{g}$ is written as
\begin{equation*}
B(X,Y)=\mathrm{Tr}(\mathrm{ad}_{X}\,\mathrm{ad}_{Y}).
\end{equation*}
 The Killing form $B$ is negative definite on $\mathfrak{k}$ and positive-definite on $\mathfrak{p}$. Furthermore, $\mathfrak{k}$ and $\mathfrak{p}$ are orthogonal. Then \begin{equation*}
     <X,Y> = - B(X,\theta(Y))
 \end{equation*}
defines a positive-definite inner product on $\mathfrak{g}$. Fix an orthonormal basis $\{k_{i}\}$ of $\mathfrak{k}$ and an orthonromal basis $\{p_{j}\}$. Then $\{k_{i}\}\cup \{p_{j}\}$ is an orthonormal basis of $\mathfrak{g}$, and $\{-k_{i}\}\cup \{p_{j}\}$ is the dual basis with respect to the Killing form.

Let $U(\mathfrak g)$ denote the universal enveloping algebra of
$\mathfrak g$.
Then the Casimir element
$\Omega_G\in Z(U(\mathfrak g))$
is given by
\begin{equation*}
\Omega_G
=
\sum_i -k_i^2
+
\sum_j p_j^2.
\end{equation*}
It is known that the definition of $\Omega_{G}$ does not depend on a choice of the orthonormal basis of $\mathfrak{g}$. 

\subsection{Preliminaries about the Casimir Operator}\label{Section:Casimir_operator}

\begin{defi}\label{definition:right_left_differential_operator}
For a function $f\in C^{\infty}(G)$, define the right and left translation as 
\begin{equation*}
    R_{g}f(h)=f(hg),\quad L_{g}f(h)=f(gh),\ h,g \in G.
\end{equation*}
For $y \in \mathfrak{g}$, the derived right action is
\begin{equation*}
  \mathfrak{R}_{y}f(g) = \left.\frac{\mathrm{d}}{\mathrm{d}t}\,f\bigl(g\exp(ty)\bigr)\right|_{t=0}.
\end{equation*}
Similarly, one can define the left-invariant operator
\begin{equation*}
  \mathfrak{L}_{y}f(g) = \left.\frac{\mathrm{d}}{\mathrm{d}t}\,f\bigl(\exp(ty)g\bigr)\right|_{t=0}.
\end{equation*}
And this action extends to $U(\mathfrak{g})$.
Since the Casimir operator $\Omega_{G}$ is in the center of $U(\mathfrak{g})$, $\mathfrak{R}_{\Omega_{G}}=\mathfrak{L}_{\Omega_{G}}.$
\end{defi}

Given a semisimple Lie algebra $\mathfrak{g}$ with Cartan decomposition $\mathfrak{g} = \mathfrak{k} \oplus \mathfrak{p}$, and a fixed maximal abelian subspace $\mathfrak{a} \subset \mathfrak{p}$, then $\mathrm{ad}\mid_{\mathfrak{a}}$ is self-adjoint and diagonalizable. We define a root as a non-zero linear functional $\alpha \in \mathfrak{a}^*$ (the dual space of $\mathfrak{a}$). The set $\mathfrak{g}_{\alpha}$ is called the root space corresponding to the root $\alpha$, and it is defined as follows: \begin{equation*}
    \mathfrak{g}_{\alpha} = \{X \in \mathfrak{g} \mid \text{ad}_Y(X) = [Y, X] = \alpha(Y)X \text{ for all } Y \in \mathfrak{a}\}.  
\end{equation*}
Then, $\mathfrak{g}$ is decomposed into a direct sum of root spaces \begin{equation*}
    \mathfrak{g} = \mathfrak{g}_0 \oplus \bigoplus_{\alpha \in \Sigma} \mathfrak{g}_\alpha, 
\end{equation*}
where $\Sigma$ is the restricted root system of $(\mathfrak{g}, \mathfrak{a})$.
Note that different root spaces are orthogonal with respect to Killing form.
Fix an order on $\mathfrak{a}^{*}$ and denote by $\Sigma^{+}$ by the corresponding set of positive roots. Let \begin{equation*}
    \mathfrak{n} = \sum\limits_{\alpha \in \Sigma^{+}} \mathfrak{g}_{\alpha}.
\end{equation*}
Then the Lie algebra $\mathfrak{g}$ can be decomposed as $\mathfrak{g}=\mathfrak{k}\oplus \mathfrak{a} \oplus \mathfrak{n}$. Also, the Lie group $G$ can be decomposed as $KAN$ where $A=\exp(\mathfrak{a})$ and $N=\exp(n)$. This is known as the Iwasawa decomposition.
Moreover, the product map is known to be a diffeomorphism. 

We have 
$\mathfrak{g}_{0} = \mathfrak{m} \oplus \mathfrak{a}$, 
where $\mathfrak{m}$ is the centralizer of $\mathfrak{a}$ in $\mathfrak{k}$. 
Fix root vectors $n_i \in \mathfrak{g}_{\alpha_i}$ corresponding to positive roots $\alpha_i$, normalized so that $\|n_i\| = 1$, so that they form an orthonormal basis of $\mathfrak{n}$. 
Then $\theta(n_i)$ are also orthonormal vectors, and satisfy 
$\theta(n_i) \in \mathfrak{g}_{-\alpha_i}$.

Further, let $\{m_{j}\}$ be an orthonormal basis of $\mathfrak{m}$ and $\{a_{l}\}$ be an orthonormal basis of $\mathfrak{a}$. Then we have a complete orthonormal basis of $\mathfrak{g}$: 
\begin{equation*}
    \{n_{i}\}\cup \{\theta(n_{i})\} \cup \{m_{j}\} \cup \{ a_{l}\}.
\end{equation*}
Let  
\begin{equation*}
    k_{i} = \frac{ n_{i} + \theta(n_{i})}{\sqrt{ 2 }} \in \mathfrak{k}, \ \ 
p_{i} =\frac{ n_{i}-\theta(n_{i})}{\sqrt{ 2 }} \in \mathfrak{p}.
\end{equation*}
And  $\theta(n_{s})=\frac{k_{s}-p_{s}}{\sqrt{ 2 }}$, $n_{s}=\frac{k_s + p_{s}}{\sqrt{2}}.$
This gives us 
\begin{equation*}
    p_{s} = -k_{s}+ \sqrt{ 2 } n_{s}.
\end{equation*}
Then $\{k_{i}\}\cup \{p_{i}\}\cup \{m_{j}\}\cup \{a_{l}\}$ forms an orthonormal basis of $\mathfrak{g}$. We decompose the Casimir operator according to the orthogonal splitting of $\mathfrak{g}$ into $\mathfrak{k}$ and its orthogonal complement. More precisely, we define
\begin{equation*}
\Omega_{K} := - \sum_{i} k_{i}^{2} \;-\; \sum_{j} m_{j}^{2},
\qquad
\Omega_{\mathfrak{p}} := \sum_{l} a_{l}^{2} \;+\; \sum_{i} p_{i}^{2}.
\end{equation*}
Here $\Omega_{K}$ is the Casimir operator of $K$, while $\Omega_{\mathfrak{p}}$ denotes the complementary part corresponding to the non-compact directions.

With this notation, the Casimir operator of $G$ can be written as
\begin{equation*}
\Omega_{G}
=
\Omega_{K} + \Omega_{\mathfrak{p}}.
\end{equation*}
From now on, we denote 
\begin{equation*}
    x_{i}= \begin{cases} a_{i}, & 1 \leq i \leq r = \text{dim}(\mathfrak{a}), \\ p_{i-r}, & r < i \leq \text{dim}(\mathfrak{p})=d. \end{cases}
\end{equation*}
Then $\{ x_{i}\}$ is an orthonormal basis of $\mathfrak{p}$.

\subsection{Explicit Expression of Casimir Operator in Iwasawa Coordinates}
Now the goal is to find a formula for $\mathfrak{L}_{\Omega_{G}}$ in terms of the Iwasawa coordinates. 
For $g \in G$, we write $g=k_{g}a_{g}n_{g}$ for its Iwasawa decomposition. The subscript is omitted if it is clear in the context. 
\begin{defi}\label{definition:differential_operator_Iwasawa}
    For a smooth function $f \in C^{\infty}(G)$ and $g =kan \in G$. We introduce the notation of differential operators as follows: 
\begin{align*}
\text{For } x \in \mathfrak{k},\quad 
\mathfrak{K}_{x}f(g) &= \left.\frac{d}{dt}\, f(k\exp(tx)an)\right|_{t=0}, \\
\text{For } x \in \mathfrak{a},\quad 
\mathfrak{A}_{x}f(g) &= \left.\frac{d}{dt}\, f(k\exp(tx)an)\right|_{t=0}, \\
\text{For } x \in \mathfrak{n},\quad 
\mathfrak{N}_{x}f(g) &= \left.\frac{d}{dt}\, f(ka\exp(tx)n)\right|_{t=0}.
\end{align*}
\end{defi}

\begin{lem}\label{lemma:commuting_differential_Iwasawa_coordinate_operator}
Let $x \in \mathfrak{k}$, $y \in \mathfrak{a}$, and $z \in \mathfrak{n}$.
Then the corresponding Iwasawa coordinates differential operators commute pairwise, namely
\begin{equation*}
[\mathfrak{K}_{x} , \mathfrak{A}_{y}] = 0, \qquad
[\mathfrak{K}_{x} , \mathfrak{N}_{z}] = 0, \qquad
[\mathfrak{A}_{y} , \mathfrak{N}_{z}] = 0.
\end{equation*}
\end{lem}

\begin{proof}
    We verify that the operators $\mathfrak{K}_{x}$ and $\mathfrak{A}_{y}$ commute for all $x \in \mathfrak{k}$ and $y \in \mathfrak{a}$. The remaining commutation relations follow by analogous arguments. Starting from the definitions, for $g=kan$, 
\begin{align*}
\mathfrak{K}_{x}\mathfrak{A}_{y}(f(g)) 
&= \mathfrak{K}_{x}\left( \left. \frac{\mathrm{d}}{\mathrm{d}t} f\big(k\exp(ty)\, a n\big) \right|_{t=0} \right) \\
&= \left. \frac{\partial}{\partial s} \left( \frac{\partial}{\partial t} f\big(k\exp(sx)\exp(ty)\, a n\big) \right) \right|_{t=s=0} \\
&= \left. \frac{\partial^2}{\partial s\partial t}\, f\big(k\exp(sx)\exp(ty)\, a n\big) \right|_{t=s=0}.
\end{align*}

By the definition of the differential operators, the group elements $\exp(sx)$ and $\exp(ty)$ act on the $K$- and $A$-coordinates respectively in the Iwasawa decomposition. The order of differentiation can be interchanged:
\begin{equation*}
\mathfrak{K}_{x}\mathfrak{A}_{y}(f(g)) = \mathfrak{A}_{y}\mathfrak{K}_{x}(f(g)).
\end{equation*}
\end{proof}

\begin{lem} \label{orthogonalrelationship}
For $g=kan \in G$, define 
$\omega_{ij}(g)=\langle \mathrm{Ad}(k)x_{i},x_{j}\rangle .$
Then the matrix $(\omega_{ij}(g))$ is orthogonal, and it satisfies 
$\omega_{ij}(k^{-1})=\omega_{ji}(k).$
\end{lem}

This follows because $\mathrm{Ad}(k)$ preserves the chosen inner product and therefore acts by an orthogonal transformation on the orthonormal basis $\{x_i\}$.

First, we want to compute $\mathfrak{L}_{x_{i}}$. Let $g=kan \in G$, then 
\begin{align*}
\mathfrak{L}_{x_{i}}f(g) &= \left.\frac{d}{dt}[f(\exp(tx_i)kan)]\right|_{t=0}=\left.\frac{d}{dt}[f(k\exp(t \mathrm{Ad}(k^{-1})x_i)an)]\right|_{t=0}.
\end{align*}
We have explicit expansion of $\mathrm{Ad}(k^{-1})x_{i}$, 
\begin{equation*}
    \mathrm{Ad}(k^{-1})x_{i} = \sum\limits_{j} \omega_{ij}(k^{-1})x_{j}
= \sum\limits_{j} \omega_{ji}(k) x_{j}.
\end{equation*}
Plugging into the formula of $\mathfrak{L}_{x_{i}}f(g)$, it becomes 
\begin{align*}
\mathfrak{L}_{x_{i}}f(g)&=\mathfrak{R}_{\mathrm{Ad}(k^{-1})x_{i}} R_{an}f(k) =\sum\limits_{j}\omega_{ji}(k)\mathfrak{R}_{x_{j}} R_{an}f(k)\\
& = \sum\limits_{j=1}^{r} \omega_{ji}(k) \mathfrak{A}_{a_{j}}f(g) + \sum\limits_{j=r+1}^{d} \omega_{ji}(k)\mathfrak{R}_{p_{j-r}} R_{an }f(k).
\end{align*}
Hence the terms in the second summation $\mathfrak{R}_{p_{j-r}} R_{an}f(k)$ becomes $$
-\mathfrak{K}_{k_{j-r}}f(g) + \sqrt{ 2 } \mathfrak{R}_{n_{j-r}} R_{an} f(k),
$$
where $\mathfrak{R}_{n_{j-r}} R_{an}f(k)$ can be expanded as follows,
\begin{align*}
\mathfrak{R}_{n_{j-r}}R_{an}f(k) 
= \left.\frac{d}{dt}\,f\!\big(k \exp(tn_{j-r})an\big)\right|_{t=0}
   = \frac{d}{dt}\,f\!\big(ka\,\exp(\mathrm{Ad}(a^{-1})tn_{j-r})\,n\big) 
= e^{-\alpha_{j-r}(\log a)}\,\mathfrak{N}_{n_{j-r}}f(g).
\end{align*}

With the Iwasawa decomposition $g=kan$, we conclude that  
\begin{equation*}
    \mathfrak{L}_{x_{i}} f(g) = \sum\limits_{j =1}^{r} \omega_{ji}(k) \mathfrak{A}_{a_{j}}f(g) + \sum\limits_{j=r+1} ^{d} \omega_{ji}(k) [-\mathfrak{K}_{k_{j-r}}f(g)+\sqrt{ 2 }e^{-\alpha_{j-r}(\log a)}\mathfrak{N}_{n_{j-r}}f(g)].
\end{equation*}
For simplicity, we write $$
\mathfrak{L}_{x_{i}} f(g) = - \Delta^{K}_{i} f(g) + \Delta^{A}_{i} f(g) +\sqrt{ 2 } \Delta^{N}_{i} f(g),
$$
where 
\begin{align*}
\Delta^{K}_{i} f(g) &= \sum\limits_{j=r+1}^{d} \omega_{ji}(k)\mathfrak{K}_{k_{j-r}}f(g), \quad  \Delta^{A}_{i} f(g) = \sum\limits_{j=1}^{r} \omega_{ji}(k) \mathfrak{A}_{a_{j}} f(g), \\
 \Delta^{N}_{i} f(g) &= \sum\limits_{j=r+1}^{d} \omega_{ji}(k) e^{-\alpha_{j-r}(\log a)} \mathfrak{N}_{n_{j-r}}f(g).
\end{align*}
We have 
\begin{equation*}
    \mathfrak{L}_{x^{2}_{i}} =\mathfrak{L}^{2}_{x_{i}} = (-\Delta^{K}_{i}+\Delta^{A}_{i}+\sqrt{ 2}\Delta^{N}_{i})^{2}.
\end{equation*}
Therefore, the next step is to compute all the terms of $\mathfrak{L}_{x_{i}}^{2}$. 

\subsubsection{Computation of $\sum\limits_{i}\Delta^{K}_{i}\Delta^{A}_{i}$.}\label{subsubsec:computation_of_KA}
Let us start on $\sum\limits_{i}\Delta^{K}_{i}\Delta^{A}_{i}$.
For $y \in \mathfrak{k}$ and $j= 1,\dots,r$, as above by Iwasawa decomposition, $g =kan$. 
\begin{align*}
\mathfrak{K}_{y} [\omega_{ji}(k)\mathfrak{A}_{a_{j}}f(kan)]
=&\left.\frac{d}{dt}[\omega_{ji}(k \exp(ty)) \mathfrak{A}_{a_{j}}f(k\exp(ty)an)]\right|_{t=0}\\
=& \left.\frac{d}{dt}\omega_{ji}(k \exp(ty))\right|_{t=0}  \mathfrak{A}_{a_{j}}f(kan)+ \omega_{ji}(k) \mathfrak{K}_{y} \mathfrak{A}_{a_{j}}f(kan).
\end{align*}
When $j \in \{1,\dots,r\}$, $x_{j}=a_{j}$. The first derivative that we got can be further calculated as 
\begin{align*}
\left.\frac{d}{dt}\,\omega_{ji}\big(k e^{ty}\big)\right|_{t=0}
= \left.\frac{d}{dt}\,\langle \mathrm{Ad}(k e^{ty})x_{j},x_{i}\rangle \right|_{t=0}
   = \left.\frac{d}{dt}\,\langle e^{t\mathrm{ad}_{y}} a_{j},\,\mathrm{Ad}(k^{-1})x_{i}\rangle \right|_{t=0} 
= \langle [y,a_{j}],\,\mathrm{Ad}(k^{-1})x_{i}\rangle .
\end{align*}
Then by expressing $[y,a_j]$ in terms of the basis $\{x_l\}$: $$  
    [y,a_j] = \sum_{l=1}^{d} \langle [y,a_j], x_l \rangle x_l,  
    $$
the formulae becomes
\begin{align*}
\left.\frac{d}{dt}[\omega_{ji}(k \exp(ty))]\right|_{t=0} &= \sum\limits_{l=1}^{d} \omega_{il}(k^{-1})\langle [y,a_{j}],x_{l} \rangle =\sum\limits_{l=1}^{d} \omega_{li}(k) \langle  [y,a_{j}],x_{l} \rangle .
\end{align*}

Hence we conclude that 
\begin{align*}
&\mathfrak{K}_{y} [\omega_{ji}(k)\mathfrak{A}_{a_{j}}f(kan)]
= \sum\limits_{l=1}^{d} \omega_{li}(k) \langle  [y,a_{j}],x_{l} \rangle \mathfrak{A}_{a_{j}}f(kan)+ \omega_{ji}(k) \mathfrak{K}_{y}\mathfrak{A}_{a_{j}}f(kan).
\end{align*}
From this, we obtain 
\begin{align*}
\Delta^{K}_{i}\Delta^{A}_{i} f(g) &= \sum\limits_{s=r+1}^{d} \sum\limits_{j=1}^{r} \omega_{si}(k)\mathfrak{K}_{k_{s-r}}[\omega_{ji}(k)\mathfrak{A}_{a_{j}}f(g)]\\
&= \sum\limits_{s=r+1}^{d} \sum\limits_{j=1}^{r}\sum\limits_{l=1}^{d} \omega_{si}(k)\omega_{li}(k) \langle  [k_{s-r},a_{j}],x_{l} \rangle \mathfrak{A}_{a_{j}}f(kan) + \sum\limits_{s=r+1}^{d} \sum\limits_{j=1}^{r} \omega_{si}(k)\omega_{ji}(k) \mathfrak{K}_{k_{s-r}}\mathfrak{A}_{a_{j}}f(kan).
\end{align*}
Since $\{\omega_{ij}\}$ is orthogonal as defined in Lemma~\ref{orthogonalrelationship}, 
$$
\sum\limits_{i} \omega_{si}(k)\omega_{li}(k)=\delta_{sl}.
$$
Therefore, the second term in the above expression for $\Delta^{K}_{i}\Delta^{A}_{i}$ vanishes after summation over $i$. 
We obtain
\begin{equation*}
    \sum\limits_{i} \Delta^{K}_{i}\Delta^{A}_{i}f(g) = \sum\limits_{s=r+1}^{d} \sum\limits_{j=1}^{r} \langle  [k_{s-r},a_{j}],x_{s} \rangle \mathfrak{A}_{a_{j}}f(kan) = \sum\limits_{s=r+1}^{d} \sum\limits_{j=1}^{r} \langle  [k_{s-r},a_{j}],p_{s-r} \rangle \mathfrak{A}_{a_{j}}f(kan). 
\end{equation*}

Using the invariance of Killing form, 
\begin{equation*}
    \langle [k_{s-r},a_{j}],p_{s-r} \rangle = -B([k_{s-r},a_{j}],p_{s-r})= B(a_{j},[k_{s-r},p_{s-r}]),
\end{equation*}
where 
\begin{equation*}
    [k_{s-r},p_{s-r}] = \left[ \frac{n_{s-r}+\theta(n_{s-r})}{\sqrt{ 2 }}, \frac{n_{s-r}-\theta(n_{s-r})}{\sqrt{ 2 }} \right]= [\theta(n_{s-r}),n_{s-r}].
\end{equation*}
Therefore, the Killing form becomes 
\begin{align*}
\langle [k_{s-r},a_{j}],p_{s-r} \rangle 
&= B(a_{j},[\theta(n_{s-r}),n_{s-r}])
= -B(a_{j},[n_{s-r},\theta(n_{s-r})]) \\
&= B([n_{s-r},a_{j}],\theta(n_{s-r}))
= -\alpha_{s-r}(a_{j}) \cdot \langle n_{s-r},n_{s-r} \rangle \\
&= -\alpha_{s-r}(a_{j}).
\end{align*}

Denoting $2\rho=\sum\limits_{s=r+1}^{d} \alpha_{s-r}$, finally, we obtain\begin{equation*}
    \sum\limits_{i} \Delta^{K}_{i}\Delta^{A}_{i}f(g) = -\sum\limits_{j=1}^{r} 2\rho(a_{j}) \mathfrak{A}_{a_{j}}f(g).
\end{equation*}
\subsubsection{Computation of $\sum\limits_{i}\Delta^{A}_{i}\Delta^{K}_{i}$.}
Next we compute $\sum_{i}\Delta^{A}_{i}\Delta^{K}_{i}$. For a fixed index $i$ one has
\begin{align*}
\Delta^{A}_{i}\Delta^{K}_{i} f(g)
&= \sum_{s=r+1}^{d}\sum_{j=1}^{r} \omega_{ji}(k)\, \mathfrak{A}_{a_{j}}\!\left( \omega_{si}(k)\, \mathfrak{K}_{k_{s-r}} f (g)\right) \\
&= \sum_{s=r+1}^{d}\sum_{j=1}^{r} \Big( \omega_{ji}(k)\,(\mathfrak{A}_{a_{j}}\omega_{si})(k)\, \mathfrak{K}_{k_{s-r}} f(g)
+ \omega_{ji}(k)\,\omega_{si}(k)\, \mathfrak{A}_{a_{j}}\mathfrak{K}_{k_{s-r}} f(g) \Big).
\end{align*}
The first term vanishes due to the fact that $(\mathfrak{A}_{a_{j}}\omega_{si})(k)=0$ and the second term vanishes too after summing over $i$ because of Lemma~\ref{orthogonalrelationship}.

\subsubsection{Computation of $\sum\limits_{i}\Delta^{K}_{i}\Delta^{N}_{i}$.}
Furthermore, the computation of $\sum\limits_{i}\Delta^{K}_{i}\Delta^{N}_{i}$ is similar to what we did for $\sum\limits_{i}\Delta^{K}_{i}\Delta^{A}_{i}$.
We get 
\begin{equation*}
   \begin{aligned}
\Delta^{K}_{i}\Delta^{N}_{i}f(g)
=& \sum\limits_{s=r+1}^{d}\sum\limits_{j=r+1}^{d} \sum\limits_{l=1}^{d}\omega_{si}(k)\omega_{li}(k)e^{-\alpha_{j-r}(\log a)}\langle [k_{s-r},p_{j-r}],x_{l} \rangle \mathfrak{N}_{n_{j-r}}f(g)\\
+& \sum\limits_{s=r+1}^{d}\sum\limits_{j=r+1}^{d}\omega_{si}(k)\omega_{ji}(k)e^{-\alpha_{j-r}(\log a)}\mathfrak{K}_{k_{s-r}}\mathfrak{N}_{n_{j-r}}f(g).
\end{aligned} 
\end{equation*}
So that the summation over $i$  becomes
\begin{align*}
\sum\limits_{i}\Delta^{K}_{i}\Delta^{N}_{i}f(g)
=&\sum\limits_{s=r+1}^{d}\sum\limits_{j=r+1}^{d}e^{-\alpha_{j-r}(\log a)}\langle [k_{s-r},p_{j-r}],p_{s-r} \rangle \mathfrak{N}_{n_{j-r}}f(g)\\
+& \sum\limits_{s=r+1}^{d} e^{-\alpha_{s-r}(\log a)} \mathfrak{K}_{k_{s-r}}\mathfrak{N}_{n_{s-r}}f(g),
\end{align*}
where the inner product is dealt same as above: 
\begin{equation*}
    \langle [k_{s-r},p_{j-r}],p_{s-r} \rangle = B([k_{s-r},p_{j-r}],p_{s-r}) 
 = -B(p_{j-r}, [k_{s-r},p_{s-r}]).
\end{equation*}
And we have
\begin{equation}\label{equation:computation_of_Lie_bracket}
    \begin{aligned}
        [k_{s-r},p_{s-r}]
&=
\frac{1}{2}
\big[
n_{s-r}+\theta(n_{s-r}),
\,n_{s-r}-\theta(n_{s-r})
\big] \\
&=
\frac{1}{2}
\Big(
[n_{s-r},n_{s-r}]
-
[n_{s-r},\theta(n_{s-r})]
+
[\theta(n_{s-r}),n_{s-r}]
-
[\theta(n_{s-r}),\theta(n_{s-r})]
\Big).
    \end{aligned}
\end{equation}
Since $[X,X]=0$ and $[n_{s-r},\theta(n_{s-r})]=-[\theta(n_{s-r}),n_{s-r}]$, the above expression reduces to
\begin{equation*}
[k_{s-r},p_{s-r}]
=
[\theta(n_{s-r}),n_{s-r}].
\end{equation*}

Note that this element lies in $\mathfrak{g}_{0} \cap \mathfrak{p} = \mathfrak{a}$. In particular, this element is orthogonal to $p_{j-r}$. Hence, the inner product $\langle [k_{s-r},p_{j-r}],p_{s-r} \rangle=0$, and $$
\sum\limits_{i}\Delta^{K}_{i}\Delta^{N}_{i}f(g) = \sum\limits_{s=r+1}^{d} e^{-\alpha_{s-r}(\log a)} \mathfrak{K}_{k_{s-r}}\mathfrak{N}_{n_{s-r}} f(g).
$$

\subsubsection{Computation of $\sum\limits_{i}\Delta^{N}_{i}\Delta^{K}_{i}$.}
\begin{align*}
&\Delta^{N}_{i}\Delta^{K}_{i} f(g)
=\sum\limits_{s=r+1}^{d}\sum\limits_{j=r+1}^{d}\omega_{si}(k)\omega_{ji}(k)e^{-\alpha_{j-r}(\log a)} \mathfrak{N}_{n_{j-r}}\mathfrak{K}_{k_{s-r}}f(g).
\end{align*}
Again, by the orthogonal relationship in Lemma~\ref{orthogonalrelationship}, we obtain \begin{equation*}
    \sum\limits_{i=1}^{d} \Delta^{N}_{i}\Delta^{K}_{i} = \sum\limits_{s=r+1}^{d} e^{-\alpha_{s-r}(\log a)} \mathfrak{N}_{n_{s-r}} \mathfrak{K}_{k_{s-r}} f(g).
\end{equation*}
\subsubsection{Computation of $\sum\limits_{i}\Delta^{N}_{i}\Delta^{A}_{i}$.}
Similarly, due to Lemma~\ref{orthogonalrelationship}, Since $s\in\{r+1,\dots,d\}$ and $j\in\{1,\dots,r\}$, one has $s\neq j$, and therefore the entire sum vanishes,
\begin{equation*}
    \sum\limits_{i}^{d}\Delta^{N}_{i}\Delta^{A}_{i}
= \sum\limits_{i=1}^{d} \sum\limits_{s=r+1}^{d}\sum\limits_{j=1}^{r}\omega_{si}(k)\omega_{ji}(k) e^{-\alpha_{s-r}(\log a)}\mathfrak{N}_{n_{s-r}}\mathfrak{A}_{a_{j}}f(g) 
=0.
\end{equation*}

\subsubsection{Computation of $\sum\limits_{i}\Delta^{A}_{i}\Delta^{N}_{i}$.}
For any $y \in \mathfrak{a}$, we compute the composition of two differential operators $\mathfrak{A}_{y}$ and $\mathfrak{N}_{n_{s}}$, 
\begin{equation*}
\mathfrak{A}_{y}\!\bigl[e^{-\alpha_{s}(\log a)}\,\mathfrak{N}_{n_{s}}f(g)\bigr]
=
\bigl(\mathfrak{A}_{y} e^{-\alpha_{s}(\log a)}\bigr)\,\mathfrak{N}_{n_{s}}f(g)
+
e^{-\alpha_{s}(\log a)}\,\mathfrak{A}_{y}\mathfrak{N}_{n_{s}}f(g).
\end{equation*}
The first term can be simplified immediately, 
\begin{equation*}
\mathfrak{A}_{y}\bigl(e^{-\alpha_{s}(\log a)}\bigr)
=
-\alpha_{s}(y)\,e^{-\alpha_{s}(\log a)}.
\end{equation*}
Hence, 
\begin{align*}
\Delta^{A}_{i}\Delta^{N}_{i} f(g)
=& \sum_{j=1}^{r}\omega_{ji}(k)\,\mathfrak{A}_{a_{j}}
\left(\sum_{s=r+1}^{d}\omega_{si}(k)\,e^{-\alpha_{s-r}(\log a)}\,\mathfrak{N}_{n_{s-r}} f(g)\right)\\
=& -\sum_{s=r+1}^{d}\sum_{j=1}^{r}\omega_{si}(k)\omega_{ji}(k)\,\alpha_{s-r}(a_j)\,e^{-\alpha_{s-r}(\log a)}\,
\mathfrak{N}_{n_{s-r}} f(g) \\
&+\sum_{s=r+1}^{d}\sum_{j=1}^{r}\omega_{si}(k)\omega_{ji}(k)\,e^{-\alpha_{s-r}(\log a)}\,
\mathfrak{A}_{a_{j}}\mathfrak{N}_{n_{s-r}} f(g).
\end{align*}
So $\sum\limits_{i=1}^{d} \Delta^{A}_{i}\Delta^{N}_{i}f(g)$ vanishes because of Lemma~\ref{orthogonalrelationship}.

\subsubsection{Computation of  $\sum\limits_{i}(\Delta^{K}_{i})^{2}$.} By the similar process as Section~\ref{subsubsec:computation_of_KA}, we obtain 
\begin{align*}
(\Delta^{K}_{i})^{2}f(g) = \Delta^{K}_{i}\Delta^{K}_{i} f(g) 
=& \sum\limits_{s=r+1}^{d}\sum\limits_{j=r+1}^{d} \omega_{si}(k)\mathfrak{K}_{k_{s-r}}[\omega_{ji}(k)\mathfrak{K}_{k_{j-r}}f(g)]\\
=& \sum\limits_{s=r+1}^{d}\sum\limits_{j=r+1}^{d}\sum\limits_{l=1}^{d} \omega_{si}(k) \omega_{li}(k) \langle [k_{s-r},p_{j-r}],x_{l} \rangle \mathfrak{K}_{k_{j-r}} f(g) \\
+& \sum\limits_{s=r+1}^{d} \sum\limits_{j=r+1}^{d} \sum\limits_{l=1}^{d} \omega_{si}(k)\omega_{li}(k) \mathfrak{K}_{k_{s-r}} \mathfrak{K}_{k_{j-r}} f(g),
\end{align*}
so that 
\begin{equation*}
    \sum\limits_{i} (\Delta^{K}_{i})^{2} f(g)
= \sum\limits_{s=r+1}^{d}\sum\limits_{j=r+1}^{d} \langle [k_{s-r},p_{j-r}],p_{s-r} \rangle \mathfrak{K}_{k_{j-r}} f(g)
+ \sum\limits_{s=r+1}^{d}   \mathfrak{K}_{k_{s-r}}^{2}f(g).
\end{equation*}

The inner product $\langle [k_{s-r},p_{j-r}],p_{s-r} \rangle= -B(p_{j-r},[k_{s-r},p_{s-r}])$, where the Lie bracket $[k_{s-r},p_{s-r}]=[\theta(n_{s-r}),n_{s-r}]$ as computed in~\eqref{equation:computation_of_Lie_bracket}. Note that this element lies in $\mathfrak{g}_{0} \cap \mathfrak{p} = \mathfrak{a}$. In particular, this element is orthogonal to $p_{j-r}$. Hence, the inner product $\langle [k_{s-r},p_{j-r}],p_{s-r} \rangle=0$, and 
\begin{align*}
\sum_{i} (\Delta^{K}_{i})^{2}f(g)
&= \sum_{s=r+1}^{d}\mathfrak{K}_{k_{s-r}}^{2}f(g) = \sum_{s=r+1}^{d}\mathfrak{K}_{k_{s-r}}^{2}\bigl(R_{an}f\bigr)(k)
= \sum_{s=r+1}^{d}\mathfrak{R}_{k_{s-r}}^{2}\bigl(R_{an}f\bigr)(k) \\
&= \mathfrak{R}_{-\Omega_{K}}\bigl(R_{an}f\bigr)(k)+\mathfrak{R}_{\Omega_{M}}\bigl(R_{an}f\bigr)(k) = \mathfrak{L}_{-\Omega_{K}}f(g)+\mathfrak{K}_{\Omega_{M}}f(g).
\end{align*}
In the last step we use that $\Omega_{G}$ lies in the center of $U(\mathfrak g)$, which allows us to replace the right action on $R_{an}f$ by the left action on $f$ at $g=kan$.

\subsubsection{Computation of $\sum\limits_{i=1}^{d}(\Delta^{A}_{i})^{2}$}
\begin{align*}
(\Delta^{A}_{i})^{2}f(g)= \sum\limits_{s=1}^{r}\sum\limits_{j=1}^{r} \omega_{si}(k)\omega_{ji}(k) \mathfrak{A}_{a_{s}}\mathfrak{A}_{a_{j}} f(g).
\end{align*}
Thanks to Lemma~\ref{orthogonalrelationship}, we get $$
\sum\limits_{i=1}^{d}(\Delta^{A}_{i})^{2}f(g) = \sum\limits_{s=1}^{r} (\mathfrak{A}_{a_{s}})^{2}f(g).
$$
\subsubsection{Computation of $\sum\limits_{_{i}}(\Delta^{N}_{i})^{2}$.}
\begin{equation*}
    (\Delta^{N}_{i})^{2}f(g) = \sum\limits_{s=r+1}^{d} \sum\limits_{j=r+1}^{d} \omega_{si}(k)\omega_{ji}(k) e^{-\alpha_{s-r}(\log a)-\alpha_{j-r}(\log a)}\mathfrak{N}_{n_{s-r}}\mathfrak{N}_{n_{j-r}} f(g).
\end{equation*}
Once again, by Lemma~\ref{orthogonalrelationship}, the formula simplifies to 
\begin{align*}
\sum\limits_{i=1}^{d}(\Delta^{N}_{i})^{2}f(g) = \sum\limits_{s=r+1}^{d}  e^{-2\alpha_{s-r}(\log a)}\mathfrak{N}_{n_{s-r}}^{2} f(g).
\end{align*}

Now we are ready to assemble all those terms together. Recall that 
\begin{equation*}
    \Omega_{G}= \Omega_{K} + \sum\limits_{i=1}^{d} x_{i}^{2},\ \ 
\mathfrak{L}^{2}_{x_{i}} = (-\Delta^{K}_{i}+\Delta^{A}_{i}+\sqrt{ 2 }\Delta^{N}_{i})^{2}.
\end{equation*}
Hence we obtain the final theorem.
\begin{thm}[Derivative formula for the Casimir operator]\label{theorem:Casimiroperatorforsemisimpleconnected}
Let $G$ be a connected semisimple Lie group with Iwasawa decomposition $G=KAN$.  
Write $g=kan\in G$ and select orthonormal bases
\begin{equation*}
\{a_{1},\dots,a_{r}\}\subset\mathfrak a,\quad
\{n_{1},\dots,n_{m}\}\subset\mathfrak n,\quad
\{k_{1},\dots,k_{m}\}\subset\mathfrak k ,
\end{equation*}
so that $[a,n_{j}]=\alpha_{j}(a)\,n_{j}$ for every $a\in\mathfrak a$.  
Let $\rho=\tfrac12\sum\limits_{\alpha\in\Sigma^{+}}\alpha$.  
Then for every $f\in C^{\infty}(G)$,
\begin{align*}
\mathfrak{L}_{\Omega_{G}}f(g)
&=\sum_{i=1}^{r}\Bigl(\mathfrak{A}_{a_i}^{2}+2\rho(a_i)\mathfrak{A}_{a_i}\Bigr)f(g)\\
&\quad+\sum_{j=1}^{m}\Bigl[-2\sqrt{2}\,e^{-\alpha_{j}(\log a)}\,\mathfrak{K}_{k_j}\mathfrak{N}_{n_j}
      +2\,e^{-2\alpha_{j}(\log a)}\,\mathfrak{N}_{n_j}^{2}\Bigr]f(g)\\
&\quad+\mathfrak{K}_{\Omega_{M}}f(g),
\end{align*}
where ${\Omega_{M}}$ denotes the Casimir operator of the centralizer $M=Z_{K}(A)$.
\end{thm}
\section{Preliminaries on Harmonic Analysis on \texorpdfstring{$SO(n,1)$}{SO(n,1)}}\label{Section:perliminaries}
\subsection{\texorpdfstring{Lie-Algebraic Framework for $SO(n,1)$}{Lie-Algebraic Framework for SO(n,1)}}\label{Liealgebra}

The non-compact real rank-one Lie group $SO(n,1)^{\circ}$, with $n\geq 3$, is the identity component of
\begin{equation*}
SO(n,1)
=
\left\{
A\in GL_{n+1}(\mathbb R)
\,\middle|\,
A^{T}JA=J,\ \det A=1
\right\},
\qquad
J=\operatorname{diag}(1,\ldots,1,-1).
\end{equation*}
Its Lie algebra is
\begin{equation*}
\mathfrak g=\mathfrak{so}(n,1)
=
\left\{
X\in\mathfrak{gl}_{n+1}(\mathbb R)
\,\middle|\,
X^{T}J+JX=0
\right\}.
\end{equation*}

The Cartan involution $\theta(X)=-X^{T}$
induces the Cartan decomposition $\mathfrak g=\mathfrak k\oplus\mathfrak p,$
where
\begin{equation*}
\mathfrak k
=
\{X\in\mathfrak g\mid \theta(X)=X\}
\simeq
\mathfrak{so}(n),
\qquad
\mathfrak p
=
\{X\in\mathfrak g\mid \theta(X)=-X\}.
\end{equation*}

Let $E_{ij}$ denote the matrix with a $1$ in the $(i,j)$-entry and zeros elsewhere. We fix the one-dimensional Cartan subspace
\begin{equation*}
\mathfrak a=\mathbb R a_1,
\qquad
a_1=E_{1,n+1}+E_{n+1,1}.
\end{equation*}

The Killing form is
\begin{equation*}
B(X,Y)
=
\operatorname{tr}\bigl(\operatorname{ad}X\circ\operatorname{ad}Y\bigr),
\qquad
X,Y\in\mathfrak g.
\end{equation*}
For $\mathfrak{so}(n,1)$, one has
\begin{equation*}
B(X,Y)
=
(n-1)\operatorname{tr}(XY).
\end{equation*}
In particular, $B(a_1,a_1)
=
2(n-1).$

In Section~\ref{section:casimiroperator_generalliegroup}, the Casimir operator was defined using the positive definite inner product
\begin{equation*}
\langle X,Y\rangle_B
=
-B(X,\theta Y).
\end{equation*}
With respect to this inner product,
\begin{equation*}
\langle a_1,a_1\rangle_B
=
2(n-1).
\end{equation*}

For the hyperbolic normalization, we instead use the rescaled inner product
\begin{equation*}
\langle X,Y\rangle
=
\frac{1}{2(n-1)}
\langle X,Y\rangle_B
=
-\frac{1}{2(n-1)}B(X,\theta Y).
\end{equation*}
So $\langle a_1,a_1\rangle=1.$

Throughout the rest of this section, all orthonormal bases and Casimir operators are taken with respect to this normalized inner product.

Define the restricted root spaces by
\begin{equation*}
\mathfrak g_{\alpha}
=
\{X\in\mathfrak g\mid [a_1,X]=X\},
\qquad
\mathfrak g_{-\alpha}
=
\{X\in\mathfrak g\mid [a_1,X]=-X\}.
\end{equation*}
Thus the restricted root system is $\Sigma=\{\pm\alpha\},
\alpha(a_1)=1.$
Moreover, $\dim\mathfrak g_{\alpha}
=
\dim\mathfrak g_{-\alpha}
=
n-1.$
The corresponding subgroups are $N=\exp(\mathfrak n),
M=Z_K(A).$
Then the Iwasawa decompositions hold at both the Lie algebra and group levels: $\mathfrak g
=
\mathfrak k\oplus\mathfrak a\oplus\mathfrak n,
G=KAN,$
where $K=SO(n)$.

Choose an orthonormal basis $\{n_i\}_{i=1}^{n-1}
\subset
\mathfrak n$
with respect to the normalized inner product. Then $\{\theta(n_i)\}_{i=1}^{n-1}
\subset
\mathfrak g_{-\alpha}$
is also orthonormal. Define
\begin{equation*}
k_i
=
\frac{1}{\sqrt 2}\bigl(n_i+\theta(n_i)\bigr)
\in
\mathfrak k,
\qquad
p_i
=
\frac{1}{\sqrt 2}\bigl(n_i-\theta(n_i)\bigr)
\in
\mathfrak p.
\end{equation*}

Let $\{m_j\}$ be an orthonormal basis of $\mathfrak m$. Then $\{k_i\}_{i=1}^{n-1}
\cup
\{m_j\}$
is an orthonormal basis of $\mathfrak k$, while $\{a_1\}
\cup
\{p_i\}_{i=1}^{n-1}$
is an orthonormal basis of $\mathfrak p$.

\begin{coro}\label{thm:Expression_of_casimir_operator}
Let $G=\mathrm{SO}(n,1)^\circ$ with Iwasawa decomposition $G=KAN$. Write $g=kan,
a=\exp(ta_1).$
Then, for every $f\in C^\infty(G)$,
\begin{align*}
\mathfrak L_{\Omega_G}f(g)
={}&
\mathfrak A_{a_1}^2f(g)
+
(n-1)\mathfrak A_{a_1}f(g)\\
&+
\sum_{i=1}^{n-1}
\left[
-2\sqrt2\,e^{-t}\,
\mathfrak K_{k_i}\mathfrak N_{n_i}f(g)
+
2e^{-2t}\,
\mathfrak N_{n_i}^2f(g)
\right]\\
&+
\mathfrak K_{\Omega_M}f(g),
\end{align*}
where $\Omega_M$ denotes the Casimir operator of $M=Z_K(A)\simeq SO(n-1).$
\end{coro}

\subsection{Representation Theory of Orthogonal Groups}\label{representationtheory}

In this section, we provide a brief overview of the representation theory of  $SO(n,1)$, summarizing Thieleker’s work \cite{ernestgenerallorentzgroup,ernestquasisimpleirredrepnt} or one can refer to~\cite[Chapter VII]{knapprepntofsemisimple}.  

Let $G=SO(n,1)^{\circ}$ and $\eta$ be an irreducible unitary representation of $M$, and let
$\nu \in \mathfrak{a}_{\mathbb{C}}^{*}$. Consider the normalized induced module
$I^{\infty}(\eta,\nu)$ consisting of all smooth functions
$F\in C^{\infty}(G,V_{\eta})$ satisfying
\begin{equation*}
F(xman)
=
e^{-(\nu+\rho)(\log a)}\,\eta(m)^{-1}F(x),
\qquad x\in G,\ m\in M,\ a\in A,\ n\in N.
\end{equation*}
The left-regular action
\begin{equation*}
\bigl(\pi^{\mathrm{ind}}_{\eta,\nu}(g)F\bigr)(x)=F(g^{-1}x),
\qquad g,x\in G,
\end{equation*}
defines a continuous representation of $G$ on $I^{\infty}(\eta,\nu)$.

For $\nu\in i\mathfrak{a}^{*}$ we endow $I^{\infty}(\eta,\nu)$ with the
$G$-invariant Hermitian form
\begin{equation*}
\lVert F\rVert^{2}
=
\int_{K} \bigl\lVert F(k)\bigr\rVert^{2}_{V_{\eta}}\,dk,
\end{equation*}
and write $I(\eta,\nu)$ for the corresponding Hilbert space completion.
Then $\pi^{\mathrm{ind}}_{\eta,\nu}$ extends to a unitary representation
of $G$ on $I(\eta,\nu)$.
For certain additional real parameters $\nu$ (the complementary series, the end-point representation and discrete series),
$\pi^{\mathrm{ind}}_{\eta,\nu}$ admits an alternative $G$-invariant
Hilbert space structure and is therefore unitarizable, although it is not unitary
with respect to the preceding norm. In general, $I(\eta,\nu)$ need not be
irreducible; reducibility may occur for special values of $\nu$.

The irreducible unitary representations of $G=\mathrm{SO}(n,1)$ consist of the principal series, complementary series, end-point representations, and (when $n$ is even) discrete series representations. A complete classification together with the admissible parameter ranges may be found in Thieleker~\cite{ernestgenerallorentzgroup}.

Let $\rho=(n-1)/2$ and let $\eta$ be an irreducible unitary representation of $M$. For any irreducible unitary representation $\pi$ arising from the Langlands parameter $(\eta,\nu)$, the Casimir operator acts by the scalar
\begin{equation*}
\pi(\Omega)
=
\nu^{2}-\rho^{2}
+
\eta(\Omega_{M}),
\end{equation*}
where $\Omega_M$ denotes the Casimir operator of $M$.

The Casimir eigenvalue of an irreducible unitary representation $\tau$ of compact subgroup $K=SO(n)$ with highest weight $\Lambda_{\tau}$ is given by
\begin{equation*}
\tau(\Omega_{K})
=
\left\langle
\Lambda_{\tau}+\sigma,
\Lambda_{\tau}+\sigma
\right\rangle
-
\left\langle
\sigma,
\sigma
\right\rangle,
\end{equation*}
where $\sigma$ is half the sum of the positive roots of $\mathfrak{k}_{\mathbb{C}}$.

Another important ingredient in the study of unitary representations is the matrix
coefficient. Let $\pi$ be a unitary representation of $G$ on a Hilbert space
$H$ with no non-zero $G$-invariant vectors. $\pi$ admits a unitary direct integral decomposition into irreducible
unitary representations. In the present rank-one setting, the irreducible unitary
representations of $G$ may be parametrized (up to equivalence) by pairs
$(\eta,\nu)$, where $\eta\in\widehat{M}$ and
$\nu\in\mathfrak{a}_{\mathbb{C}}^{*}$. Thus there exist a Borel measure $\mu_{\pi}$ on
$\widehat{M}\times \mathfrak{a}_{\mathbb{C}}^{*}$ and a multiplicity function
$m_{\pi}(\eta,\nu)\in\{0,1,2,\dots,\infty\}$ such that
\begin{equation*}
H
\simeq
\int_{\widehat{M}\times \mathfrak{a}_{\mathbb{C}}^{*}}^{\oplus}
V_{\eta,\nu}^{\oplus\, m_{\pi}(\eta,\nu)}\,
d\mu_{\pi}(\eta,\nu),
\end{equation*}
where $V_{\eta,\nu}$ carries an unitary representation,
denoted $\pi_{\eta,\nu}$.

Fix such a unitary direct integral decomposition, and let $\mu_{\pi}$ be the associated
spectral measure on $\widehat{M}\times \mathfrak{a}_{\mathbb{C}}^{*}$.
Define its support by
\begin{equation*}
\operatorname{Spec}(\pi)
:=
\operatorname{supp}(\mu_{\pi})
=
\Bigl\{(\eta,\nu)\in \widehat{M}\times \mathfrak{a}_{\mathbb{C}}^{*}
\;\Big|\;
\mu_{\pi}(U)>0 \text{ for every open neighborhood }U\ni(\eta,\nu)
\Bigr\}.
\end{equation*}
In the real rank-one case, since $\mathfrak a = \mathbb{R}\,\mathrm a_1$ is one-dimensional, every $\nu \in \mathfrak a_{\mathbb C}^*$ is uniquely determined by its value $\nu(\mathrm a_1)$. We therefore identify $\nu$ with the complex number $\nu(\mathrm a_1)$, and set
\begin{equation*}
\nu(\pi)
:=
\sup\Bigl\{\Re(\nu)\;\Big|\;(\eta,\nu)\in \operatorname{Spec}(\pi)\Bigr\}.
\end{equation*}
By the unitarity classification of irreducible representations of $G$, one has $0 \leq \nu(\pi) \leq \rho,$
and equality $\nu(\pi)=\rho$ occurs only for the trivial representation.

Let $X=\Gamma\backslash G$. 
For the right regular representation of $G$ on $L^{2}_{0}(X)$, defined in~\eqref{equation:L^2orthogonal_constant}, we shall write
$\nu\!\left(\Gamma\right)$ to denote the spectral gap parameter of this representation.

\begin{thm}[Matrix coefficient decay for $SO^{0}(n,1)$, following Trombi~\cite{TROMBI197883}]\label{thm:matrix-coeff-decay-eps-form}
Let $G=SO^{0}(n,1)$ and $\rho=\frac{n-1}{2}$. Let $(\pi,H)$ be a unitary
representation of $G$ with no non-zero $G$-invariant vectors, and let
$\nu(\pi)\in[0,\rho)$ be defined as above.
Then for every pair of $K$-finite vectors $v,w\in H$, there exist constants
$C_{v,w}>0$ and $q_{v,w}\in\mathbb{N}$ such that for all $t\in\mathbb{R}$,
\begin{equation*}
\bigl|\langle \pi(a_t)v, w \rangle\bigr|
\;\le\;
C_{v,w}\,(1+|t|)^{q_{v,w}}\,
\exp\!\Bigl(\bigl(-\rho+\nu(\pi)\bigr)\,|t|\Bigr).
\end{equation*}
\end{thm}

\subsection{Sobolev Norms and Embedding Inequality}\label{sec:sobolev}

Let $(\pi,H)$ be a unitary representation of $G$ and let $H^{\infty}$ denote
its space of smooth vectors.  Let $\Omega_{G}$ and $\Omega_{K}$ be the Casimirs of $G$ and $K$, and set
\begin{equation*}
{\Delta}
:= {-\Omega_{G} + 2\Omega_{K}} .
\end{equation*}
The operator $1+\mathfrak{L}_{\Delta}$ is positive, essentially self-adjoint, and
commutes with the right-regular $G$-action. We may define $(1+\mathfrak{L}_{\Delta})^{m/2}$ on $H^{\infty}$ for every $m\in\mathbb N$. 
\begin{defi}[Sobolev norms on $H^{\infty}$]\label{def:sobolevnorm_all_derivatives}
Fix a basis $\{X_{1},\dots,X_{d}\}$ of $\mathfrak g$.  For a multi-index
$\alpha=(\alpha_{1},\dots,\alpha_{k})$ with $1\le \alpha_{j}\le d$, write
$|\alpha|=k, X^{\alpha}=X_{\alpha_{1}}\cdots X_{\alpha_{k}}\in U(\mathfrak g),$
with the convention $|\emptyset|=0$ and $X^{\emptyset}=1$.
For $v\in H^{\infty}$ define the Sobolev norm of order $m\in\mathbb N$ by
\begin{equation*}
\|v\|^{\prime}_{W_{m, H}} 
=
\left(
\sum_{|\alpha|\le m}\bigl\|d\pi(X^{\alpha})v\bigr\|^{2}_{H}
\right)^{1/2},
\end{equation*}
where $\|\cdot\|_{H}$ denotes the norm of the Hilbert space $H$.
\end{defi}

Since $\|v\|_H \leq \|v\|^{\prime}_{W_m},$ $ v\in H^\infty,$
the completion $W_m(H)$ embeds continuously into $H$. We therefore regard $W_m(H)$ as a dense subspace of $H$, equipped with the Sobolev norm $\|\cdot\|^{\prime}_{W_m}$.

The following definition of Sobolev norms are equivalent to Definition~\ref{def:sobolevnorm_all_derivatives}  as pointed out in~\cite{Strombergsson2013DeviationHorocycle} if one takes regular representation $H=L^{p}(X)$. In this case, the notation of Sobolev norm is shortened as 
\begin{equation*}
    \|\cdot\|^{\prime}_{W_{m, H}} = \|\cdot\|^{\prime}_{W_{m, L^{p}(X)}}=\|\cdot\|^{\prime}_{W_{m,p}}.
\end{equation*}
\begin{defi}[Sobolev norms via $\mathfrak{L}_{\Delta}$]
\label{def:sobolev_norm_via_LDelta}
For $m\in\mathbb N$ and $1\le p\le\infty$, define
\begin{equation*}
W_{m,p}(X)
=
\Bigl\{
v\in H^{\infty}:
(1+\mathfrak L_\Delta)^{m/2}v\in L^p(X)
\Bigr\},
\end{equation*}
equipped with the norm
\begin{equation*}
\|v\|_{W_{m,p}}
=
\bigl\|
(1+\mathfrak L_\Delta)^{m/2}v
\bigr\|_{L^{p}}.
\end{equation*}
\end{defi}

\begin{lem}
\label{lem:sobolev_shift_compact}
Let $s, k\in\mathbb Z_{\geq 0}$. 
Assume that $v \in W_{s+k,2}(X)$ satisfies
\begin{equation*}
\mathfrak{L}_{\Omega}\,v = \mu\,v,
\qquad
\mathfrak{L}_{\Omega_{K}}\,v = \tau(\Omega_{K})\,v ,
\end{equation*}
for some $\mu\in\mathbb R$ and some $K$-type $\tau\in\widehat K$. 
Then for all $k\in\mathbb N$,
\begin{equation*}
\|v\|^{2}_{W_{s+k,2}}
=
\bigl(1 - \mu + 2\,\tau(\Omega_{K})\bigr)^{k}\,
\|v\|^{2}_{W_{s,2}} .
\end{equation*}
\end{lem}

\begin{proof}
For $k=1$, using the self-adjointness of $(1+\mathfrak{L}_{\Delta})$ in
$L^{2}(X)$,
\begin{equation*}
\|v\|^{2}_{W_{s+1,2}}
=
\bigl\langle (1+\mathfrak{L}_{\Delta})^{\,s+1}v,\;v\bigr\rangle
=
\bigl\langle (1+\mathfrak{L}_{\Delta})^{\,s}v,\;v\bigr\rangle
+
\bigl\langle (1+\mathfrak{L}_{\Delta})^{\,s}v,\;\mathfrak{L}_{\Delta}v\bigr\rangle .
\end{equation*}
Since $\mathfrak{L}_{\Delta}v=(-\mu + 2\,\tau(\Omega_{K}))v$, we obtain
\begin{equation*}
\|v\|^{2}_{W_{s+1,2}}
=
\bigl(1 - \mu + 2\,\tau(\Omega_{K})\bigr)\,
\|v\|^{2}_{W_{s,2}} .
\end{equation*}
Iterating this identity yields
\begin{equation*}
\|v\|^{2}_{W_{s+k,2}}
=
\bigl(1 - \mu + 2\,\tau(\Omega_{K})\bigr)^{k}\,
\|v\|^{2}_{W_{s,2}} .
\end{equation*}
\end{proof}

 In particular, the right regular representation $(R, L^{2}(X))$ of $G$ is unitary. Consequently, the Sobolev norms introduced above for a general unitary representation $(\pi,H)$ apply directly to $(R,L^{2}(X))$. Concretely, for $m\in\mathbb N$ and $f\in C^{\infty}(X)$, we
write the Sobolev norm of $f$ with order $m\in\mathbb N$ in two equivalent ways: 
\begin{equation*}
\|f\|_{W_{m,p}}
:=
\bigl\|(1+\mathfrak{L}_{\Delta})^{m/2}f\bigr\|_{L^{p}}.
\end{equation*}
 and we also write
\begin{equation*}
\|f\|^{\prime}_{W_{m,p}}
:=
\left(
\sum_{|\alpha|\le m}\bigl\|\mathfrak R_{X^{\alpha}}f\bigr\|_{L^{p}}^{2}
\right)^{1/2},
\end{equation*}
where $\mathfrak R_{X}$ denotes the right-invariant differential operator on $X$
associated to $X\in\mathfrak g$.

The following result is from Lemma~5 in~\cite{edwardsrateofexpandinghorospheres}.

\begin{prop}[Sobolev embedding on a compact quotient]\label{Prop:Sobolev-inequlity}
Assume that $X=\Gamma\backslash G$ is compact. Let an integer $s>\frac{\dim G}{2}=\frac{n(n+1)}{4}.$
Then there exists a constant $C_{s,\Gamma}>0$ such that, for every
$f\in W_{s,2}(X)$,
\begin{equation*}
\|f\|_{L^\infty(X)}
\le
C_{s,\Gamma}\|f\|_{W_{s,2}}.
\end{equation*}
\end{prop}

\subsection{Weyl-Type Upper Bound}\label{section:weyl-type-upper-bound}
In this subsection we decompose $L^{2}(X)$ according to $K$-types and study
the resulting blocks separately.

Let $\tau \in \widehat K$ be an irreducible unitary representation of $K$ on
the finite-dimensional Hilbert space $V_{\tau}$. We denote by
\begin{equation*}
E_{\tau}
:=
G \times_{K} V_{\tau}
\longrightarrow
Y:=X/K
\end{equation*}
the homogeneous vector bundle associated to $\tau$. Here $K$ acts on
$G \times V_{\tau}$ by
\begin{equation*}
(g,v)\cdot k
:=
(gk,\tau(k)^{-1}v).
\end{equation*}
Under the standard identification of the $\tau$-isotypic component of
$L^{2}(X)$ with $L^{2}(Y,E_{\tau})\otimes V_{\tau}$, the operator $-\mathfrak R_{\Omega_G-\Omega_K}$
is identified with a nonnegative, essentially self-adjoint, elliptic
second-order differential operator on $E_{\tau}$. Equivalently,
\begin{equation*}
-\mathfrak R_{\Omega_G-\Omega_K}
=
\nabla^{\tau,*}\nabla^{\tau}\otimes \operatorname{Id}_{V_{\tau}},
\end{equation*}
where $\nabla^{\tau}$ is the natural connection on $E_{\tau}$.

Moreover, since $\Omega_K$ acts on the $\tau$-isotypic component by the
scalar $\tau(\Omega_K)$, we have
\begin{equation*}
\mathfrak R_{\Omega_G}
=
\mathfrak R_{\Omega_G-\Omega_K}
+
\tau(\Omega_K)
=
-\nabla^{\tau,*}\nabla^{\tau}\otimes \operatorname{Id}_{V_{\tau}}
+
\operatorname{Id}\otimes \tau(\Omega_K).
\end{equation*}
We then apply a uniform local spectral-density upper bound to the family
$\nabla^{\tau,*}\nabla^{\tau}$, equivalently to the family
$-\mathfrak R_{\Omega_G-\Omega_K}$ restricted to the $\tau$-isotypic
components, and obtain the desired Weyl-type upper bound for
$1+\mathfrak L_{\Delta}$.

We now turn to a uniform local spectral-density upper bound for the Laplace-type operators $\nabla^{\tau,*}\nabla^\tau$.
For each $\tau\in\widehat K$, the operator $\nabla^{\tau,*}\nabla^\tau$ on $L^2(Y,E_\tau)$ is a nonnegative self-adjoint elliptic operator on the compact manifold $Y$, hence has discrete spectrum with finite multiplicities. We write
\begin{equation*}
\mathrm{Spec}\bigl(\nabla^{\tau,*}\nabla^\tau;L^2(Y,E_\tau)\bigr)
=
\{\lambda_j(\tau)\}_{j\ge1},
\qquad
\lambda_j(\tau)\nearrow\infty,
\end{equation*}
where eigenvalues are repeated according to multiplicity.

For each $\tau\in \widehat K$, let $V_\tau$ be the representation space of $\tau$, let $r_\tau=\operatorname{rank}(E_\tau),$
and for $S\ge0$ let
\begin{equation*}
\Pi_S^\tau:=\mathbf 1_{[0,S]}(\nabla^{\tau,*}\nabla^\tau)
\end{equation*}
be the spectral projector. Let $K_S^\tau(y,y')\in \operatorname{Hom}((E_\tau)_{y'},(E_\tau)_y)$
be its Schwartz kernel, and define
\begin{equation*}
e_\tau(y;S):=\operatorname{tr}_{(E_\tau)_y}K_S^\tau(y,y).
\end{equation*}
Recall $d_Y=\dim Y$ and set $d_{K}=\dim K$.

Recall $1+\mathfrak L_\Delta=1-\mathfrak R_{\Omega_G}+2\mathfrak R_{\Omega_K}$
acting on scalar $L^2(X)$.
\begin{lem}\label{lemma:compact_weyl_law}
Let
\begin{equation*}
N_{\Delta}(T):=
    \#\{\text{eigenvalues of }1+\mathfrak L_\Delta\text{ on }L^2(X)\text{ not exceeding }T\},
\qquad T\ge 1,
\end{equation*}
counted with multiplicity. There exists a constant $C>0$, depending only on $(G,K,\Gamma)$ and the
chosen $\operatorname{Ad}(K)$-invariant inner product on $\mathfrak g$, such that
\begin{equation*}
N_{\Delta}(T)\le C\,T^{\frac{d_Y+d_{K}}{2}},
\qquad T\ge 1.
\end{equation*}
\end{lem}
\begin{proof}
    For each $\tau\in\widehat K$, let
\begin{equation*}
N_\tau(T):=\sum_{\lambda\le T} m_{\tau,\lambda},
\qquad T\ge 1,
\end{equation*}
where $m_{\tau,\lambda}$ denotes the multiplicity of the eigenvalue $\lambda$
of $\nabla^{\tau,*}\nabla^\tau$ on $L^2(Y,E_\tau)$.
Since
\begin{equation*}
N_\tau(T)=\operatorname{tr}\Pi_T^\tau=\int_Y e_\tau(y;T)\,dy,
\end{equation*}
for every $\tau\in\widehat K$ and every $T\ge 1$, one has
\begin{equation}
\label{eq:Ntau_bound_merged_cor}
N_\tau(T)\ll_X r_\tau\,T^{d_Y/2}.
\end{equation}
In particular, see Theorem~6.2 in~\cite{Miatelloweyllaw}. In the notation of that paper, the spectral heat trace $\phi_\tau(s)$ of the elliptic operator on the $\tau$-vector bundle is related to the counting function through an integral transform.
Theorem~6.2 establishes considerably stronger results; however, for the present application it is enough to observe that any upper bound for $\phi_\tau(s)$ as $s\to0$ implies the corresponding upper bound for the counting function.

Now consider the scalar $\tau$-isotypic block of $L^2(X)$.
By the standard $K$-isotypic decomposition,
\begin{equation*}
L^2(X)_\tau \simeq L^2(Y,E_{\tau^\vee})\otimes V_\tau.
\end{equation*}
After relabeling $\tau^\vee$ by $\tau$, the operator $1+\mathfrak L_\Delta$
acts on this block as $1+\nabla^{\tau,*}\nabla^\tau+\tau(\Omega_K).$
Therefore, if $\lambda$ is an eigenvalue of $\nabla^{\tau,*}\nabla^\tau$ with
multiplicity $m_{\tau,\lambda}$ on $L^2(Y,E_\tau)$, then the corresponding eigenvalue
$1+\lambda+\tau(\Omega_K)$
of $1+\mathfrak L_\Delta$ on scalar $L^2(X)$ has multiplicity
$r_\tau\,m_{\tau,\lambda}$.
It follows that
\begin{equation*}
N_{\Delta}(T)
=
\sum_{\substack{
\tau\in\widehat K,\ \lambda\in\mathrm{Spec}(\nabla^{\tau,*}\nabla^\tau)\\
1+\lambda+\tau(\Omega_K)\le T}}
r_\tau\,m_{\tau,\lambda}.
\end{equation*}
If $1+\lambda+\tau(\Omega_K)\le T,$
then necessarily $\lambda\le T$ and $
\tau(\Omega_K)\le T.$
Hence
\begin{equation*}
N_{\Delta}(T)
\le
\sum_{\substack{\tau\in\widehat K\\ \tau(\Omega_K)\le T}}
r_\tau\,N_\tau(T).
\end{equation*}
Using \eqref{eq:Ntau_bound_merged_cor}, we obtain
\begin{equation}
\label{eq:NA_second_reduction_merged_cor}
N_{\Delta}(T)
\ll_X
T^{d_Y/2}
\sum_{\substack{\tau\in\widehat K\\ \tau(\Omega_K)\le T}} r_\tau^2.
\end{equation}

Let $\Lambda(\tau)\in\Lambda_w\cap\mathfrak t_K^+$ be the highest weight of $\tau$.
From
\begin{equation*}
\tau(\Omega_K)=\langle \Lambda(\tau),\Lambda(\tau)+2\rho_K\rangle,
\end{equation*}
we obtain $|\Lambda(\tau)|\le C_1\sqrt T$ whenever $\tau(\Omega_K)\le T.$
By Weyl's dimension formula,
\begin{equation*}
r_\tau\le C_2(1+|\Lambda(\tau)|)^{|R_K^+|}.
\end{equation*}
Therefore
\begin{align*}
\sum_{\substack{\tau\in\widehat K\\ \tau(\Omega_K)\le T}} r_\tau^2
&\ll_K
\sum_{\substack{\Lambda\in\Lambda_w\cap\mathfrak t_K^+\\ |\Lambda|\le C_1\sqrt T}}
(1+|\Lambda|)^{2|R_K^+|} \\
&\ll_K
T^{|R_K^+|}
\#\bigl\{
\Lambda\in\Lambda_w\cap\mathfrak t_K^+:\ |\Lambda|\le C_1\sqrt T
\bigr\}.
\end{align*}
Since $\Lambda_w$ is a full lattice in the $r_K$-dimensional vector space
$\mathfrak t_K^*$ and $\mathfrak t_K^+$ is a rational polyhedral cone,
standard lattice-point counting gives
\begin{equation*}
\#\bigl\{
\Lambda\in\Lambda_w\cap\mathfrak t_K^+:\ |\Lambda|\le C_1\sqrt T
\bigr\}
\ll_K
T^{r_K/2}.
\end{equation*}
Using $2|R_K^+|+r_K=d_{K},$
we conclude from \eqref{eq:NA_second_reduction_merged_cor} that
\begin{equation*}
N_{\Delta}(T)\ll T^{\frac{d_Y+d_{K}}{2}}.
\end{equation*}
\end{proof}
For each eigenvalue $\Lambda\in\mathrm{Spec}(1+\mathfrak L_\Delta)$, let
\begin{equation*}
E_\Lambda:=\ker(1+\mathfrak L_\Delta-\Lambda).
\end{equation*}
Then $E_\Lambda(1+\mathfrak L_\Delta)$ is finite-dimensional and invariant under $\mathfrak R_{\Omega_G}$ and $\mathfrak R_{\Omega_M}$.
Decompose
\begin{equation*}
E_\Lambda=\bigoplus_{(\mu,\varpi)\in\mathcal J(\Lambda)} E_{\Lambda,\mu,\varpi}
\end{equation*}
into joint eigenspaces of $(\mathfrak R_{\Omega_G},\mathfrak R_{\Omega_M})$, where $\mathcal J(\Lambda)$ is the finite set of joint eigenvalue pairs
$(\mu,\varpi)$ occurring in $E_\Lambda(1+\mathfrak L_\Delta)$, and where on each block
$E_{\Lambda,\mu,\varpi}$ one has
\begin{equation*}
1+\mathfrak L_\Delta=\Lambda\,\mathrm{Id},
\qquad
\mathfrak R_{\Omega_G}=\mu\,\mathrm{Id},
\qquad
\mathfrak R_{\Omega_M}=\varpi\,\mathrm{Id}.
\end{equation*}

\begin{coro}[Scalar counting and block summability for $1+\mathfrak L_\Delta$]
\label{cor:block_summability_for_A}
Assume that $K$ is connected.
For every real number $s<-\frac{d_Y+d_{K}}{2},$
one has
\begin{equation*}
\sum_{\Lambda\in\mathrm{Spec}(1+\mathfrak L_\Delta)}
\ \sum_{(\mu,\varpi)\in\mathcal J(\Lambda)}
\Lambda^s
<
\infty.
\end{equation*}
\end{coro}

\begin{proof}

Now fix $\Lambda\in\mathrm{Spec}(1+\mathfrak L_\Delta)$, 
since $1+\mathfrak L_\Delta$ is self-adjoint with discrete spectrum on the
compact manifold $X$, the eigenspace $E_\Lambda$
is finite-dimensional. Because $1+\mathfrak L_\Delta$ commutes with the right
$K$-action, $E_\Lambda$ is a finite-dimensional unitary $K$-module. Hence
\begin{equation*}
E_\Lambda=\bigoplus_{\tau\in F(\Lambda)} E_{\Lambda,\tau}
\end{equation*}
as a finite orthogonal direct sum of $K$-isotypic components. On each
$E_{\Lambda,\tau}$, the operator $\mathfrak R_{\Omega_K}$ acts by the scalar
$\tau(\Omega_K)$. Since $(1+\mathfrak L_\Delta)|_{E_\Lambda}=\Lambda\,\mathrm{Id}$ and $1+\mathfrak L_\Delta=1-\mathfrak R_{\Omega_G}+2\mathfrak R_{\Omega_K},$
we obtain on $E_{\Lambda,\tau}$:
\begin{equation*}
\mathfrak R_{\Omega_G}=(1+2\tau(\Omega_K)-\Lambda)\,\mathrm{Id}.
\end{equation*}
Thus $\mathfrak R_{\Omega_G}$ is already scalar on each $K$-isotypic block.

Now each $E_{\Lambda,\tau}$ is $M$-stable, because $M\subset K$. Since $M$ is
compact, $E_{\Lambda,\tau}$ decomposes as a finite orthogonal direct sum of
irreducible $M$-modules. On each irreducible $M$-module, the operator $\mathfrak R_{\Omega_M}$ acts by a scalar. Therefore $E_\Lambda$ decomposes as a finite
orthogonal direct sum of joint eigenspaces of $(\mathfrak R_{\Omega_G},\mathfrak R_{\Omega_M})$:
\begin{equation*}
E_\Lambda(1+\mathfrak L_\Delta)=\bigoplus_{(\mu,\varpi)\in\mathcal J(\Lambda)} E_{\Lambda,\mu,\varpi}.
\end{equation*}
Since each nonzero block $E_{\Lambda,\mu,\varpi}$ has dimension at least $1$, we have $\#\mathcal J(\Lambda)\le \dim E_\Lambda.$

Define the block counting function
\begin{equation*}
M(T):=
\#\bigl\{(\Lambda,\mu,\varpi):\ \Lambda\le T,\ (\mu,\varpi)\in\mathcal J(\Lambda)\bigr\}.
\end{equation*}
Then by Lemma~\ref{lemma:compact_weyl_law},
\begin{equation*}
M(T)
=
\sum_{\Lambda\le T}\#\mathcal J(\Lambda)
\le
\sum_{\Lambda\le T}\dim E_\Lambda
=
N_{\Delta}(T)\ll T^{\frac{d_Y+d_{K}}{2}}.
\end{equation*}

Moreover, every eigenvalue $\Lambda$ of $1+\mathfrak L_\Delta$ satisfies
$\Lambda\ge 1$. Since on the scalar
$\tau$-isotypic block,
\begin{equation*}
1+\mathfrak L_\Delta=1+\nabla^{\tau,*}\nabla^\tau+\tau(\Omega_K),
\end{equation*}
each eigenvalue is of the form $1+\lambda+\tau(\Omega_K)\ge 1.$
Now let $s<-\frac{d_Y+d_{K}}{2}.$
Then dyadic summation gives
\begin{align*}
\sum_{\Lambda\in\mathrm{Spec}(1+\mathfrak L_\Delta)}
\ \sum_{(\mu,\varpi)\in\mathcal J(\Lambda)}
\Lambda^s
&=
\sum_{m\ge 0}
\sum_{\substack{\Lambda\in\mathrm{Spec}(1+\mathfrak L_\Delta),\ (\mu,\varpi)\in\mathcal J(\Lambda)\\
2^m\le \Lambda<2^{m+1}}}
\Lambda^s \\
&\le
\sum_{m\ge 0}(2^m)^s
\#\bigl\{(\Lambda,\mu,\varpi):\ \Lambda<2^{m+1}\bigr\} \\
&\ll
\sum_{m\ge 0}2^{m\left(s+\frac{d_Y+d_{K}}{2}\right)}<\infty.
\end{align*}
\end{proof}

\section{Construction of the ODE and Its Solutions}\label{section:construction_of_ODE} 

\subsection{Reduction to the ODE}
Let $(\pi,H)$ be a unitary representation of $G$, and let $v\in H^\infty$. Assume that $v$ is a joint eigenvector of the infinitesimal Casimirs of $G$
and $M$:
\begin{equation}\label{equation:eigenvalue_equation}
d\pi(\Omega_{G})v=\mu v,
\qquad
d\pi(\Omega_{M})v=\upvarpi v .
\end{equation}
By Corollary~\ref{thm:Expression_of_casimir_operator}, we demonstrate that the eigenvalue equations~\eqref{equation:eigenvalue_equation},
which are partial differential equations. And this yields an ODE for the corresponding $K$-averages $I(t, v,\varphi)$, when these averages are considered as functions of the parameter $t$.

We define the $H$-valued orbit map
\begin{equation*}
f_v:G\to H,
\qquad
f_v(g):=\pi(g)v.
\end{equation*}
When $g=kan$ is written in Iwasawa coordinates with $k\in K$, $a\in A$, and $n\in N$,
we will often restrict $f_v$ to the $KA$-slice, that is, to points with $n=e$, and write
\begin{equation*}
f_v(ka)=\pi(ka)v,
\qquad (k,a)\in K\times A.
\end{equation*}

Let $y\in\mathfrak g$. The corresponding right-invariant differential operator
acting on smooth $H$-valued functions $f_{v}:G\to H$ is defined by
\begin{equation*}
(\mathfrak R_y f_v)(g)
=
\left.\frac{d}{dt}\right|_{t=0} \pi(g\exp(ty))v
=
\pi(g)\,d\pi(y)v.
\end{equation*}
Likewise, for $n\in\mathfrak n$, the Iwasawa coordinates differential operator
$\mathfrak N_n$ satisfies, along the $KA$-slice,
\begin{equation*}
(\mathfrak N_n f_v)(ka)
=
(\mathfrak R_n f_v)(ka)
=
\pi(ka)\,d\pi(n)v.
\end{equation*}

\begin{lem}[Stability of the $\Omega_{M}$–eigenspace under $\mathfrak n$–derivatives]
\label{lem:OmegaM_stable_under_n}
Let $v\in H^{\infty}$ and assume that the $KA$–slice restriction
$f_{v}(ka)=\pi(ka)v$ satisfies
\begin{equation*}
\mathfrak{K}_{\Omega_{M}} f_{v} = \upvarpi\, f_{v}
\quad\text{on }KA .
\end{equation*}
Then for every $n\in\mathfrak n$ and every integer $j\ge 0$,
\begin{equation*}
d\pi(\Omega_{M})\, d\pi(n)^{j} v
=
\upvarpi\, d\pi(n)^{j} v .
\end{equation*}

\end{lem}

\begin{proof}
Recall that by Lemma~\ref{lemma:commuting_differential_Iwasawa_coordinate_operator} we have
$[\mathfrak{K}_{\Omega_{M}},\, \mathfrak{N}_{n}] = 0$ and hence $[\mathfrak{K}_{\Omega_{M}}, \mathfrak{N}_{n}^{j}] = 0 $.
Applying this to $f_{v}$ gives
\begin{equation*}
\mathfrak{K}_{\Omega_{M}}(\mathfrak{N}_{n}^{j} f_{v})
=
\mathfrak{N}_{n}^{j}(\mathfrak{K}_{\Omega_{M}} f_{v})
=
\upvarpi\, \mathfrak{N}_{n}^{j} f_{v}.
\end{equation*}
Evaluating at $ka_{t}$, so the $N$–coordinate is trivial,
\begin{equation*}
(\mathfrak{N}_{n}^{j} f_{v})(ka_{t})
=
\pi(ka_{t})\, d\pi(n)^{j} v .
\end{equation*}
Therefore, since $M$ commutes with $a_t$,
\begin{equation*}
\mathfrak{K}_{\Omega_{M}}(\mathfrak{N}_{n}^{j} f_{v})(ka_{t})
=
\pi(ka_{t})\, d\pi(\Omega_{M})\, d\pi(n)^{j}v,
\end{equation*}
while the eigenfunction identity yields
\begin{equation*}
\mathfrak{K}_{\Omega_{M}}(\mathfrak{N}_{n}^{j} f_{v})(ka_{t})
=
\upvarpi\, \pi(ka_{t})\, d\pi(n)^{j}v.
\end{equation*}
Cancelling $\pi(ka_{t})$ proves the claim.
\end{proof}

\begin{rem}\label{remark:three_eigen_relations}
There are three related eigen-relations:
\begin{align*}
\text{(i)}\;\mathfrak{K}_{\Omega_{M}}f_{v}=\upvarpi f_{v}\quad\text{on }G,\ \
\text{(ii)}\;\mathfrak{K}_{\Omega_{M}}f_{v}=\upvarpi f_{v}\quad\text{on }KA,\ \
\text{(iii)}\;d\pi(\Omega_{M})v=\upvarpi v\quad\text{in }V.
\end{align*}
Restriction gives (i)$\Rightarrow$(ii).  
Because $\Omega_{M}$ is built from $\mathfrak m$ and $M$ centralizes $A$,
\begin{equation*}
\mathfrak{K}_{\Omega_{M}} f_{v}(ka)=\pi(ka)\, d\pi(\Omega_{M})v ,
\end{equation*}
and the invertibility of $\pi(ka)$ shows
\begin{equation*}
\text{(ii)}\;\Longleftrightarrow\;\text{(iii)}.
\end{equation*}
Thus the slice eigenfunction condition (ii) is precisely equivalent to
the algebraic eigenvector condition (iii), while the global condition
(i) is strictly stronger and generally not implied by (iii).
\end{rem}

\begin{prop}[ODE for the $K$–average of a vector]\label{prop:vector_ODE}
Let $(\pi,H)$ be a unitary representation of $G$, let $v\in H^{\infty}$, and
for $\varphi\in C^{\infty}(K)$, recall
\begin{equation*}
I(t,v,\varphi)
=\int_{K}\varphi(k)\,\pi(ka_{t})v\,dk .
\end{equation*}
Assume that $v$ is a joint eigenvector of the infinitesimal Casimirs of $G$
and $M$:
\begin{equation*}
d\pi(\Omega_{G})v=\mu v,
\qquad
d\pi(\Omega_{M})v=\upvarpi v .
\end{equation*}

Then $I(t,v,\varphi)$ satisfies the inhomogeneous second–order ODE
\begin{equation*}
I''(t,v,\varphi)
+(n-1)I'(t,v,\varphi)
+(\upvarpi-\mu)\,I(t,v,\varphi)
=
2e^{-t}G_{v,\varphi}(t),
\qquad t>0,
\end{equation*}
where
\begin{equation}\label{equataion: G_v}
G_{v,\varphi}(t)
:=
-\sum_{s=2}^{n}\Bigl[
\sqrt{2}\,I\bigl(t, d\pi(n_{s-1})v, \mathfrak{K}_{k_{s-1}}\varphi\bigr)
+e^{-t}\,I\bigl(t, d\pi(n_{s-1})^{2}v,\varphi\bigr)
\Bigr].
\end{equation}
\end{prop}

\begin{proof}
Let $f_{v}(ka_t)=\pi(ka_t)v$.  Since $d\pi(\Omega_{G})v=\mu v$ and $d\pi(\Omega_{M})v=\upvarpi v$, the Iwasawa coordinates
Casimir identity (see Corollary~\ref{thm:Expression_of_casimir_operator} for rank one case) gives,
for every $k\in K$ and $t>0$,
\begin{align*}
\mu\,f_{v}(ka_{t})
&=
\bigl(\mathfrak{A}_{a_{1}}^{2}+(n-1)\mathfrak{A}_{a_{1}}\bigr)f_{v}(ka_{t}) \\
&\quad
+\sum_{j=1}^{n-1}\Bigl[
-2\sqrt{2}\,e^{-t}\mathfrak{K}_{k_{j}}\mathfrak{N}_{n_{j}}f_{v}(ka_{t})
+2e^{-2t}\mathfrak{N}_{n_{j}}^{2}f_{v}(ka_{t})
\Bigr]
+\mathfrak{K}_{\Omega_{M}} f_{v}(ka_{t}).
\end{align*}
By the remark~\ref{remark:three_eigen_relations} preceding the proposition,
$\mathfrak{K}_{\Omega_{M}}f_{v}(ka_{t})=\upvarpi\,f_{v}(ka_{t})$.
Integrating against $\varphi(k)$ gives
\begin{align*}
I''(t,v,\varphi)
\;+\;
(n-1)\,I'(t,v,\varphi)+ (\upvarpi-\mu)\,I(t,v,\varphi)
&=2e^{-t}G_{v,\varphi}(t),
\end{align*}
where \begin{equation*}
    G_{v,\varphi}(t) = -\sum_{j=1}^{n-1}\bigg[-\sqrt{2}\int_{K}
    \mathfrak{K}_{k_{j}} \mathfrak{N}_{n_{j}} f_{v}(k a_{t})\,\varphi(k)\,dk
+e^{-t}\int_{K}
    \mathfrak{N}_{n_{j}}^{2} f_{v}(k a_{t})\,\varphi(k)\,dk\bigg] .
\end{equation*}
By changing variables:
\begin{equation*}
\int_{K}\mathfrak{K}_{k_{j}}\mathfrak{N}_{n_{j}}f_{v}(ka_{t})\,\varphi(k)\,dk
=
-\int_{K}\mathfrak{N}_{n_{j}}f_{v}(ka_{t})\,\mathfrak{K}_{k_{j}}\varphi(k)\,dk.
\end{equation*}

Along the $KA$–slice, $\mathfrak{N}_{n_{j}}f_{v}(ka_{t})=\pi(ka_{t})\,d\pi(n_{j})v,$
so Lemma~\ref{lem:OmegaM_stable_under_n}
allows us to replace all $\mathfrak{N}_{n_{j}}$–derivatives on $f_{v}$ by
$d\pi(n_{j})$ acting on $v$.  Hence
\begin{equation*}
G_{v,\varphi}(t)
=
-\sum_{j=1}^{n-1}\Bigl[\sqrt{2}
I\bigl(t,\, d\pi(n_{j})v,\, \mathfrak{K}_{k_{j}}\varphi\bigr)
+e^{-t}\, I\bigl(t,\, d\pi(n_{j})^{2}v,\, \varphi\bigr)
\Bigr].
\end{equation*}
Substituting these expressions and rearranging yields the stated ODE.
\end{proof}
Next step is to estimate the upper bound of Sobolev norm of $I(t,v,\varphi)$ and $G_{v,\varphi}(t)$.
\begin{lem}[Explicit Sobolev growth of the $K$-average]
\label{lem:sobolev_boundedness_of_I_explicit}
For any fixed integer $k\ge 0$, there exists a constant $C_k>0$ such that for all
$t\ge 0$, all $v\in H^\infty$, and all $\varphi\in C^\infty(K)$,
\begin{equation*}
\|I(t,v,\varphi)\|_{W_{k,H}}
\le
C_k e^{kt}\,\|\varphi\|_{L^1(K)}\,\|v\|_{W_{k,H}}.
\end{equation*}
\end{lem}

\begin{proof}
Let $r\in\{0,\dots,k\}$. For each integer $r\ge 0$, let
\begin{equation*}
\mathcal I_r:=\{1,\dots,d\}^r,
\end{equation*}
where $d=\dim\mathfrak g$. For $r=0$, we use the convention $\mathcal I_0:=\{\emptyset\}.$
If $\beta=(\beta_1,\dots,\beta_r)\in\mathcal I_r$, we write
\begin{equation*}
X^\beta:=X_{\beta_1}\cdots X_{\beta_r}\in U(\mathfrak g),
\end{equation*}
and for $\beta=\emptyset$ we set $X^\emptyset:=1.$
For $v\in H^\infty$, define the vector of all derivatives of order $r$ by
\begin{equation*}
\mathbf D_r(v)
:=
\bigl(d\pi(X^\beta)v\bigr)_{\beta\in\mathcal I_r}
\in H^{\mathcal I_r}.
\end{equation*}
We equip $H^{\mathcal I_r}$ with the product Hilbert norm
\begin{equation*}
\|\mathbf w\|_{H^{\mathcal I_r}}^2
:=
\sum_{\beta\in\mathcal I_r}\|w_\beta\|_H^2,
\qquad
\mathbf w=(w_\beta)_{\beta\in\mathcal I_r}\in H^{\mathcal I_r}.
\end{equation*}
For $g\in G$, let
\begin{equation*}
\Pi_r(g):H^{\mathcal I_r}\longrightarrow H^{\mathcal I_r}
\end{equation*}
be the diagonal action defined by
\begin{equation*}
\bigl(\Pi_r(g)\mathbf w\bigr)_\beta
:=
\pi(g)w_\beta,
\qquad
\beta\in\mathcal I_r.
\end{equation*}
Since $\pi(g)$ is unitary on $H$, the operator $\Pi_r(g)$ is unitary on
$H^{\mathcal I_r}$.

Since $v\in H^\infty$, differentiation may be passed under the integral, and thus
\begin{equation*}
d\pi(X^\beta) I(t,v,\varphi)
=
\int_K \varphi(u)\, d\pi(X^\beta)\pi(u a_t)v\,du.
\end{equation*}
Using
\begin{equation*}
d\pi(D)\pi(g)=\pi(g)\,d\pi(\mathrm{Ad}(g^{-1})D),
\qquad D\in U(\mathfrak g),\ g\in G,
\end{equation*}
we obtain
\begin{equation*}
d\pi(X^\beta) I(t,v,\varphi)
=
\int_K \varphi(u)\,\pi(u a_t)\,
d\pi\!\bigl(\mathrm{Ad}((u a_t)^{-1})X^\beta\bigr)v\,du.
\end{equation*}
Now
\begin{equation*}
\mathrm{Ad}((u a_t)^{-1})=\mathrm{Ad}(a_t^{-1})\,\mathrm{Ad}(u^{-1}).
\end{equation*}
Since $\mathrm{Ad}(u^{-1})$ is orthogonal on $\mathfrak g$ with respect to the fixed inner
product, it has operator norm $1$. On the other hand, by the root relations in the
rank-one case,
\begin{equation*}
\mathrm{Ad}(a_t^{-1})n_i=e^{-t}n_i,\qquad
\mathrm{Ad}(a_t^{-1})\theta n_i=e^{t}\theta n_i,\qquad
\mathrm{Ad}(a_t^{-1})m_j=m_j,\qquad
\mathrm{Ad}(a_t^{-1})a_1=a_1,
\end{equation*}
so $\|\mathrm{Ad}(a_t^{-1})\|_{\mathrm{End}(\mathfrak g)}\le e^t.$
Hence $\|\mathrm{Ad}((u a_t)^{-1})\|_{\mathrm{End}(\mathfrak g)}\le e^t.$
Here $\|\cdot\|_{\mathrm{End}(\mathfrak g)}$ denotes the operator norm on
$\mathrm{End}(\mathfrak g)$ induced by the fixed inner product on $\mathfrak g$.

Let $E_r$ be the span of all words of length $r$ in the fixed basis. Then
$\mathrm{Ad}((u a_t)^{-1})$ preserves $E_r$, and the induced operator on $E_r$ has norm at most
$e^{rt}$. Therefore, for every word $\beta$ of length $r$,
\begin{equation*}
\mathrm{Ad}((u a_t)^{-1})X^\beta
=
\sum_{|\gamma|=r} c_{\beta\gamma}(u,t)\,X^\gamma,
\end{equation*}
where the coefficient matrix $C_r(u,t)=\bigl(c_{\beta\gamma}(u,t)\bigr)$ satisfies
\begin{equation*}
\|C_r(u,t)\|_{\ell^2\to\ell^2}\le e^{rt}.
\end{equation*}
Here $\|\cdot\|_{\ell^2\to\ell^2}$ denotes the operator norm on the finite-dimensional coefficient space indexed by words of length $r$.

Collecting all words of length $r$, we get
\begin{equation*}
\mathbf D_r\bigl(I(t,v,\varphi)\bigr)
=
\int_K \varphi(u)\,\Pi_r(u a_t)\,C_r(u,t)\,\mathbf D_r(v)\,du,
\end{equation*}
where $\Pi_r(u a_t)$ denotes the diagonal action of $\pi(u a_t)$ on the finite product
space $H^{\mathcal I_r}$. Since $\pi(u a_t)$ is unitary, so is $\Pi_r(u a_t)$, and hence
\begin{align*}
\bigl\|\mathbf D_r(I(t,v,\varphi))\bigr\|_H
&\le
\int_K |\varphi(u)|\,
\bigl\|\Pi_r(u a_t)\,C_r(u,t)\,\mathbf D_r(v)\bigr\|_H\,du \\
&=
\int_K |\varphi(u)|\,
\bigl\|C_r(u,t)\,\mathbf D_r(v)\bigr\|_H\,du \\
&\le
e^{rt}\,\|\varphi\|_{L^1(K)}\,\|\mathbf D_r(v)\|_H.
\end{align*}
Since $r\le k$, this gives
\begin{equation*}
\bigl\|\mathbf D_r(I(t,v,\varphi))\bigr\|_H
\le
e^{kt}\,\|\varphi\|_{L^1(K)}\,\|\mathbf D_r(v)\|_H.
\end{equation*}

Summing over $r=0,\dots,k$, and using Definition~\ref{def:sobolevnorm_all_derivatives} of the Sobolev norm, we conclude
\begin{equation*}
\|I(t,v,\varphi)\|_{W_{k,H}}
\le
e^{kt}\,\|\varphi\|_{L^1(K)}\,\|v\|_{W_{k,H}}
\end{equation*}
for the Sobolev norm defined using the fixed basis. Since all Sobolev norms coming from
Definition~\ref{def:sobolevnorm_all_derivatives} are equivalent, the stated estimate follows
with a constant $C_k>0$ depending only on the Sobolev norm convention.

Finally, for fixed $t\ge 0$ and $\varphi\in C^\infty(K)$, the above estimate shows that
$v\mapsto I(t,v,\varphi)$ is bounded on the dense subspace $H^\infty\subset W_{k}(H)$,
hence extends uniquely to a bounded operator on $W_{k}(H)$.
\end{proof}

\begin{lem}[Sobolev estimate for $G_{v,\varphi}(t)$]
\label{lem:sobolev_bound_G}
Let $(\pi,H)$ be an irreducible unitary representation of $G=\mathrm{SO}(n,1)^{\circ}$,
and let $k\ge 0$. Then there exists a constant $C_k>0$ such that for all $t\ge 0$,
all $v\in H^\infty$, and all $\varphi\in C^\infty(K)$,
\begin{equation*}
\|G_{v,\varphi}(t)\|_{W_{k,H}}
\le
C_k\, e^{kt}\,
\|v\|_{W_{k+2,H}}\,
\|\varphi\|_{W_{1,1}}.
\end{equation*}
\end{lem}

\begin{proof}
$G_{v,\varphi}(t)$ is a finite linear combination of terms 
$I\bigl(t, d\pi(n_j)v, \mathfrak{K}_{k_j}\varphi\bigr),\ I\bigl(t, d\pi(n_j)^2 v, \varphi\bigr),$
with coefficients bounded uniformly in $t\ge 0$.
Thus it suffices to estimate the Sobolev norm of each such term.
By Lemma~\ref{lem:sobolev_boundedness_of_I_explicit}, 
For the first type of term $I\bigl(t, d\pi(n_j)v, \mathfrak{K}_{k_j}\varphi\bigr),$
we obtain
\begin{equation*}
\|I(t, d\pi(n_j)v, \mathfrak{K}_{k_j}\varphi)\|_{W_{k,H}}
\le
C_k e^{kt}\,
\|\mathfrak{K}_{k_j}\varphi\|_{L^1(K)}\,
\|d\pi(n_j)v\|_{W_{k,H}}.
\end{equation*}
For the second type $I\bigl(t, d\pi(n_j)^2 v, \varphi\bigr),$
we similarly have
\begin{equation*}
\|I(t, d\pi(n_j)^2 v, \varphi)\|_{W_{k,H}}
\le
C_k e^{kt}\,
\|\varphi\|_{L^1(K)}\,
\|d\pi(n_j)^2 v\|_{W_{k,H}}.
\end{equation*}
Combining the above estimates and summing over $j=1,\dots,n-1$, we obtain
\begin{equation*}
\|G_{v,\varphi}(t)\|_{W_{k,H}}
\le
C_k e^{kt}
\|v\|_{W_{k+2,H}}
\|\varphi\|_{W_{1,1}}.
\end{equation*}
\end{proof}

\subsection{Solutions of the ODE}\label{section:solution_ODE}
Our method begins with the resolution of a general Cauchy problem, subject to initial conditions imposed at $t=0$, in order to obtain an initial expression for the integral $I(t,v,\varphi)$.  
We now solve explicitly the ODE satisfied by the sector average
$I(t,v,\varphi)$. 

Denote the discriminant of the inhomogeneous ODE in Proposition~\ref{prop:vector_ODE} by
\begin{equation}\label{equation:discriminant}
\mathcal D:=\frac{(n-1)^2}{4}-(\upvarpi-\mu).
\end{equation}
Correspondingly, the roots of the associated characteristic polynomial are defined by
\begin{equation*}
\mathcal{D}\neq 0:\ \ \lambda_{\pm}:=\frac{1-n}{2}\pm \sqrt{\mathcal D};
\qquad
\mathcal{D}=0:\ \ \lambda:=\frac{1-n}{2}.
\end{equation*}
\begin{prop}[Explicit solution formula for $I(t,v,\varphi)$]
\label{prop:explicit_solution_of_I}
Assume the same conditions of Proposition~\ref{prop:vector_ODE}. 
Set
\begin{equation}\label{eq:defi_of_kernel_function}
H_{\mathcal D}(s)
:=
\begin{cases}
e^{\lambda_-s}-\lambda_-K_{\mathcal D}(s),
& \mathcal D\neq 0,\\[0.4em]
(1+s)e^{\lambda s},
& \mathcal D=0.
\end{cases},\quad K_{\mathcal D}(s)
:=
\begin{cases}
\dfrac{e^{\lambda_+s}-e^{\lambda_-s}}{\lambda_+-\lambda_-},
& \mathcal D\neq 0,\\[1.2em]
\lambda s e^{\lambda s},
& \mathcal D=0,
\end{cases}
\qquad s\geq 0.
\end{equation}
Then
\begin{equation}
\label{eq:kernel_solution_all_D}
I(t,v,\varphi)
=
H_{\mathcal D}(t)I(0,v,\varphi)
+
K_{\mathcal D}(t)I(0,d\pi(a_1)v,\varphi)
+
2\int_0^t K_{\mathcal D}(t-r)e^{-r}G_{v,\varphi}(r)\,\mathrm{d}r .
\end{equation}
\end{prop}

 From expression~\eqref{eq:kernel_solution_all_D}, the exponential factors in the integral terms are of the form $e^{(-\lambda_{\pm}-1)r}$, where $\lambda_{\pm}=-\frac{n-1}{2}\pm \sqrt{\mathcal D}$.
In particular, when $n=2$, the subgroup $M$ is trivial, and hence the resulting ODE is considerably simpler and does not involve the more complicated discriminant $\mathcal{D}$. In this case, one has $-\lambda-1=-\frac{1}{2},$
so the integrand in~\eqref{eq:kernel_solution_all_D} possesses intrinsic exponential decay. For a more detailed discussion in the case of $SL(2,\mathbb{R})$, we refer the reader to Corso and Ravotti~\cite{Corso2022LargeHC}.

By contrast, for $n>2$, the real part satisfies $\Re(-\lambda_{+}-1)=\Re(\frac{n-3}{2}+\sqrt{\mathcal{D}})\geq 0$. Hence no decay is present at the level of the integrand, and this necessitates an iterative procedure, as developed in the following Lemma~\ref{cor:uniform_ODE_bootstrap_sobolev_growth} and Section~\ref{Section:Asymp_Irreducible_repnt_with_pure_imginary}, in order to progressively upgrade the decay rate. 

We introduce the following important upper bounds of $H_{\mathcal{D}}(s)$ and $K_{\mathcal{D}}(s)$, which we will use multiplied times in later discussions.

\begin{lem}\label{lem:upper_bounds_of_H_and_k} For any $\mathcal{D} \in \mathbb{R}$, we have the following uniform upper bounds of $H_{\mathcal{D}}(s)$ and $K_{\mathcal{D}}(s)$:
\begin{enumerate}
    \item For $s\geq 0$, we have
       \begin{equation*}
           H_{\mathcal{D}}(s) \ll 
            \begin{cases}
(1+s)e^{(\lambda + \frac{1}{4})s},
& 0 \leq \Re(\sqrt{\mathcal{D}})< \frac{1}{4},\\[0.4em]
e^{\Re(\lambda_+)s},
& \Re(\sqrt{\mathcal{D}})\geq \frac{1}{4},
\end{cases}
   \end{equation*}
   and
         \begin{equation*}
           K_{\mathcal{D}}(s) \ll 
            \begin{cases}
(1+s)e^{(\lambda + \frac{1}{4})s},
& 0 \leq \Re(\sqrt{\mathcal{D}})< \frac{1}{4},\\[0.4em]
e^{\Re(\lambda_+)s},
& \Re(\sqrt{\mathcal{D}})\geq \frac{1}{4},
\end{cases}
   \end{equation*}
     \item For $s<0$, we have
     \begin{equation*}
           K_{\mathcal{D}}(s) \ll 
            \begin{cases}
(1-s)e^{(\lambda - \frac{1}{4})s},
& 0 \leq \Re(\sqrt{\mathcal{D}})< \frac{1}{4},\\[0.4em]
e^{\Re(\lambda_-)s},
& \Re(\sqrt{\mathcal{D}})\geq \frac{1}{4}.
\end{cases}
   \end{equation*}
\end{enumerate}
\end{lem}
\begin{proof}
We begin with $s>0$. For $\sqrt{\mathcal{D}} \in \mathrm{i}\mathbb{R}_{> 0}$,
\begin{equation*}
    K_{\mathcal D}(s) =e^{\lambda s}\, \frac{\sin(\sqrt{D}s)}{\sqrt{D}} \ll (1+s)e^{\lambda s}.
\end{equation*}
For $ 0<\sqrt{\mathcal{D}}< \frac{1}{4}$,
\begin{equation*}
    K_{\mathcal D}(s) =e^{\lambda s}\, \frac{e^{\sqrt{D}s} - e^{-\sqrt{D}s}}{2\sqrt{D}} \ll e^{(\lambda + \frac{1}{4})s}.
\end{equation*}
For $\sqrt{\mathcal{D}}\geq \frac{1}{4}$,
\begin{equation*}
    K_{\mathcal D}(s) =e^{\lambda s}\, \frac{e^{\sqrt{D}s} - e^{-\sqrt{D}s}}{2\sqrt{D}} \ll e^{\lambda_+s}.
\end{equation*}
Then we obtain
\begin{equation*}
K_{\mathcal D}(s)
\ll \begin{cases}
(1+s)e^{\lambda s},
& \sqrt{\mathcal{D}} \in \mathrm{i}\mathbb{R}_{\geq 0},\\[0.4em]
e^{(\lambda + \frac{1}{4})s},
& 0<\sqrt{\mathcal{D}}< \frac{1}{4},\\[0.4em]
e^{\lambda_+s},
& \sqrt{\mathcal{D}}\geq \frac{1}{4},
\end{cases}
\quad
H_{\mathcal D}(s)
\ll \begin{cases}
(1+s)e^{\lambda s},
& \sqrt{\mathcal{D}} \in \mathrm{i}\mathbb{R}_{\geq 0},\\[0.4em]
e^{(\lambda + \frac{1}{4})s},
& 0<\sqrt{\mathcal{D}}< \frac{1}{4},\\[0.4em]
e^{\lambda_+s},
& \sqrt{\mathcal{D}}\geq \frac{1}{4}.
\end{cases}
\end{equation*}
For any $s<0$,
    \begin{equation*}
K_{\mathcal D}(s)
\ll \begin{cases}
(1-s)e^{\lambda s},
& \sqrt{\mathcal{D}} \in \mathrm{i}\mathbb{R}_{\geq 0},\\[0.4em]
e^{(\lambda - \frac{1}{4})s},
& 0<\sqrt{\mathcal{D}}< \frac{1}{4},\\[0.4em]
e^{\lambda_- s},
& \sqrt{\mathcal{D}}\geq \frac{1}{4}.
\end{cases}
\end{equation*}
\end{proof}

\begin{lem}[Uniform ODE bootstrap improvement of Sobolev growth]
\label{cor:uniform_ODE_bootstrap_sobolev_growth}
Let $(\pi,H)$ be a unitary representation of
$G=\mathrm{SO}(n,1)^\circ$.
Let $v$ be a smooth joint eigenvector and $\varphi\in C^\infty(K)$.
Fix an integer $k\geq 0$, and for every $q\in\mathbb Z_{\geq 0}$ define
\begin{equation*}
\alpha_q:=\max\{\Re(\lambda_+), \lambda + \tfrac{1}{4} ,k-q\},~~ \text{and}~~ q_0:=k+\left\lfloor \tfrac{n-1}{2} \right\rfloor + 2,
\end{equation*}
where $\left\lfloor x \right\rfloor$ denotes the integer part of $x$.
Then for any fixed integer $q \geq q_0$, we have
\begin{equation}
\label{eq:uniform_ODE_bootstrap_estimate}
\|I(t,v,\varphi)\|_{W_{k,H}}
\ll_k 
(1+t)\,e^{\alpha_q t}
\|v\|_{W_{k+2q,H}}
\|\varphi\|_{W_{q,1}},
\qquad t\geq 0.
\end{equation}
Especially for $q=q_0$, we have
\begin{equation}
\|G_{v,\varphi}(t)\|_{W_{k,H}}
\ll_k (1+t)\,\max\{e^{\Re(\lambda_+)t}, e^{(\lambda + \frac{1}{4})t}\}\, \|v\|_{W_{k+2q_0+2,H}}\, \|\varphi\|_{W_{q_0+1,1}}.
\end{equation}

\end{lem}

\begin{proof}
We now prove \eqref{eq:uniform_ODE_bootstrap_estimate} by induction on $q$.
For $q=0$, Lemma~\ref{lem:sobolev_boundedness_of_I_explicit} gives
\begin{equation*}
\|I(t,v,\varphi)\|_{W_{k,H}}
\ll_k 
e^{kt}\|\varphi\|_{L^1(K)}\|v\|_{W_{k,H}}.
\end{equation*}
Since $k\leq \alpha_0$,
the claim follows for $q=0$.

Assume that the estimate has been proved for some $q\geq 0$. Thus every average
$I(t,v,\varphi)$ generated at this stage satisfies
\begin{equation}
\label{eq:uniform_induction_hypothesis_all_D}
\|I(t,v,\varphi)\|_{W_{k,H}}
\ll_k 
\begin{cases}
(1+t)\,e^{\alpha_q t}
\|v\|_{W_{k+2q,H}}
\|\varphi\|_{W_{q,1}},
& 0 \leq \Re(\sqrt{\mathcal{D}})< \frac{1}{4},\\[0.4em]
e^{\alpha_q t}
\|v\|_{W_{k+2q,H}}
\|\varphi\|_{W_{q,1}},
& \Re(\sqrt{\mathcal{D}})\geq \frac{1}{4}.
\end{cases}
\end{equation}
Using~\eqref{equataion: G_v}
and applying \eqref{eq:uniform_induction_hypothesis_all_D} to the averages on
the right-hand side, we obtain
\begin{equation}
\label{eq:uniform_G_bootstrap_all_D}
\|G_{v,\varphi}(t)\|_{W_{k,H}}
\ll_k 
\begin{cases}
(1+t)\,e^{\alpha_q t}
\|v\|_{W_{k+2q+2,H}}
\|\varphi\|_{W_{q+1,1}},
& 0 \leq \Re(\sqrt{\mathcal{D}})< \frac{1}{4},\\[0.4em]
e^{\alpha_q t}
\|v\|_{W_{k+2q+2,H}}
\|\varphi\|_{W_{q+1,1}},
& \Re(\sqrt{\mathcal{D}})\geq \frac{1}{4}.
\end{cases}
\end{equation}

By the
$t=0$ case of Lemma~\ref{lem:sobolev_boundedness_of_I_explicit} and Lemma~\ref{lem:upper_bounds_of_H_and_k}:
\begin{align}
\label{eq:homogeneous_kernel_bound_all_D}
&\quad \left\| H_{\mathcal D}(t)I(0,v,\varphi)
+
K_{\mathcal D}(t)I(0,d\pi(a_1)v,\varphi)
\right\|_{W_{k,H}} \notag\\
\ll_k &
\begin{cases}
(1+t)\,\max\{e^{\Re(\lambda_+)t}, e^{(\lambda + \frac{1}{4})t}\}\,
\|v\|_{W_{k+2q+2,H}}
\|\varphi\|_{W_{q+1,1}},
& 0 \leq \Re(\sqrt{\mathcal{D}})< \frac{1}{4},\\[0.4em]
e^{\lambda_+ t}\,
\|v\|_{W_{k+2q+2,H}}
\|\varphi\|_{W_{q+1,1}},
& \Re(\sqrt{\mathcal{D}})\geq \frac{1}{4}.
\end{cases}
\end{align}
Meanwhile, Lemma~\ref{lem:upper_bounds_of_H_and_k} and estimate
\eqref{eq:uniform_G_bootstrap_all_D} give
\begin{equation}\label{eq:forced_kernel_bound_all_D}
    \left\|
\int_0^t K_{\mathcal D}(t-r)e^{-r}G_{v,\varphi}(r)\,\mathrm{d}r
\right\|_{W_{k,H}}
\ll_k 
(1+t)\,
e^{\max\{\Re(\lambda_+), \lambda + \frac{1}{4}, \alpha_q-1\}t}
\|v\|_{W_{k+2q+2,H}}
\|\varphi\|_{W_{q+1,1}}.
\end{equation}
We claim that the implied constant is only dependent on $k$, while is independent of $\mathcal{D}$. 

By Lemma~\ref{lem:upper_bounds_of_H_and_k}, the only non-trivial case is $\sqrt{\mathcal{D}}\geq \frac{1}{4}$ and $\|G_{v,\varphi}(t)\|_{W_{k,H}}
\ll_{k,v,\varphi}
e^{(k-q) t}$. In this case, we have
\begin{equation*}
    \left\|
\int_0^t K_{\mathcal D}(t-r)e^{-r}G_{v,\varphi}(r)\,\mathrm{d}r
\right\|_{W_{k,H}}
\ll_k 
e^{\lambda_+ t}\int_0^t e^{(k-q-1-\lambda_+)r}\,\mathrm{d}r\,
\|v\|_{W_{k+2q+2,H}}
\|\varphi\|_{W_{q+1,1}}.
\end{equation*}
For any $q$ satisfying $|k-q-1-\lambda_+|\geq 1$, we have
\begin{equation*}
  e^{\lambda_+ t}\int_0^t e^{(k-q-1-\lambda_+)r}\,\mathrm{d}r
\ll(e^{\lambda_+ t} + e^{(k-q-1)t})\ll
e^{\alpha_{q+1} t}.
\end{equation*}
Then we obtain
\begin{equation*}
    \left\|
\int_0^t K_{\mathcal D}(t-r)e^{-r}G_{v,\varphi}(r)\,\mathrm{d}r
\right\|_{W_{k,H}}
\ll_k 
e^{\alpha_{q+1} t}\,
\|v\|_{W_{k+2q+2,H}}
\|\varphi\|_{W_{q+1,1}}.
\end{equation*}
Now combining \eqref{eq:homogeneous_kernel_bound_all_D} and
\eqref{eq:forced_kernel_bound_all_D}, we obtain
\begin{equation*}
\|I(t,v,\varphi)\|_{W_{k,H}}
\ll_k
e^{\alpha_{q+1}t}
\|v\|_{W_{k+2q+2,H}}
\|\varphi\|_{W_{q+1,1}}.
\end{equation*}
Thus \eqref{eq:uniform_induction_hypothesis_all_D} holds with $q$ replaced by
$q+1$.

For some $q>0$ satisfies $k-\lambda_+-2<q<k-\lambda_+$, we have $|k-q-1-\lambda_+|< 1$. Then we obtain
\begin{equation*}
  e^{\lambda_+ t}\int_0^t e^{(k-q-1-\lambda_+)r}\,\mathrm{d}r
< e^{\lambda_+ t}\int_0^t e^{r}\,\mathrm{d}r
\ll e^{(\lambda_+ +1) t}.
\end{equation*}
Now we obtain
\begin{equation}\label{eq:need_to_iterate_again}
    \left\|
\int_0^t K_{\mathcal D}(t-r)e^{-r}G_{v,\varphi}(r)\,\mathrm{d}r
\right\|_{W_{k,H}}
\ll_k 
e^{(\lambda_+ +1) t}\,
\|v\|_{W_{k+2q+2,H}}
\|\varphi\|_{W_{q+1,1}},
\end{equation}
which does not satisfy the induction assumption. 

In this situation, we need to use the iteration scheme again. By~\eqref{eq:need_to_iterate_again}, we have
\begin{equation*}
\|I(t,v,\varphi)\|_{W_{k,H}}
\ll_k 
e^{(\lambda_+ +1) t}\,
\|v\|_{W_{k+2q+2,H}}
\|\varphi\|_{W_{q+1,1}}.
\end{equation*}
Using~\eqref{equataion: G_v}, we obtain
\begin{equation*}
\|G_{v,\varphi}(t)\|_{W_{k,H}}
\ll_k 
e^{(\lambda_+ +1) t}\,
\|v\|_{W_{k+2q+4,H}}
\|\varphi\|_{W_{q+2,1}}.
\end{equation*}
Then we have 
\begin{equation*}
\begin{aligned}
        \left\|
\int_0^t K_{\mathcal D}(t-r)e^{-r}G_{v,\varphi}(r)\,\mathrm{d}r
\right\|_{W_{k,H}}
&\ll_k 
 e^{\lambda_+ t}\int_0^t 1\,\mathrm{d}r\,
\|v\|_{W_{k+2q+4,H}}
\|\varphi\|_{W_{q+2,1}}\\
&\ll_k 
 (1+t)e^{\lambda_+ t}\,
\|v\|_{W_{k+2q+4,H}}
\|\varphi\|_{W_{q+2,1}}.
\end{aligned}
\end{equation*}
Hence we know that the implied constant in~\eqref{eq:forced_kernel_bound_all_D} is independent of $\mathcal{D}$.

Combining \eqref{eq:homogeneous_kernel_bound_all_D} and
\eqref{eq:forced_kernel_bound_all_D}, we obtain
\begin{equation*}
\|I(t,v,\varphi)\|_{W_{k,H}}
\ll_k
(1+t)\,e^{\alpha_{q+2}t}
\|v\|_{W_{k+2q+4,H}}
\|\varphi\|_{W_{q+2,1}}.
\end{equation*}
Thus \eqref{eq:uniform_ODE_bootstrap_estimate} holds with $q$ replaced by
$q+2$.
For any $q\geq q_0$, we have $q>k-\Re(\lambda_+) +1$. Then by induction, we know that \eqref{eq:uniform_ODE_bootstrap_estimate} holds for any $q\geq q_0$.
\end{proof}

From now on, due to the Sobolev embedding Theorem~\ref{Prop:Sobolev-inequlity}, we fix an integer $k> \frac{n(n+1)}{4}$.
 
\section{\texorpdfstring{Asymptotic Expression of $I(t, v, \varphi)$ for Irreducible Representations }{Asymptotic Expression of I(t, v, phi) for Irreducible Representations}}\label{Section:Asymp_Irreducible_repnt_with_pure_imginary}
In this section, we consider a joint engenvector $v$ satisfies
\begin{equation*}
d\pi(\Omega_{G})v=\mu v,
\qquad
d\pi(\Omega_{M})v=\upvarpi v.
\end{equation*}
Then we develop an iteration scheme to obtain the asymptotic expansion of $I(t,v,\varphi)$ for a untiary representation $(\pi, H)$ with fixed spectral data $\mathcal{D}$ (see~\eqref{equation:discriminant}) and prove that this asymptotic expansion is uniform bounded with respect to $\mathcal{D}$.

The key observation is that the function $I(t,v,\varphi)$ appears in the explicit expression~\eqref{equataion: G_v} of $G_{v,\varphi}$ recursively. By repeatedly substituting this expression into the solution~\eqref{eq:kernel_solution_all_D}, we derive the asymptotic expansion of $I(t,v,\varphi)$. The another main task is to ensure that this expansion remains uniform with respect to $\mathcal{D}$.

We introduce the iteration strategy and discuss the different cases of the spectral parameter $\mathcal D$ in Section~\ref{subsection:Construction_of_the_iteration}.
Sections~\ref{section:Once_decomposition} and~\ref{section:twice_decomposition} develop the iteration scheme for any fixed $\mathcal D$. We begin with presenting the first iteration step in order to illustrate the main idea of the scheme, and then derive the general recursive formula. Then we obtain the asymptotic expansion of $I(t , v, \varphi)$. Afterwards based on an integral computation lemma, we provide estimates for both the main term and the remainder term.
In Section~\ref{subsection:Third_real_part}, we address the convergence issue arising when $\mathcal D$ is sufficiently large by incorporating additional input from the Matrix Coefficient Theorem~\ref{thm:matrix-coeff-decay-eps-form}. This allows us to derive a refined iteration scheme in the large $\mathcal D$ regime.
Finally, in Section~\ref{section:asymptotic_expansion}, we combine and reorganize the results obtained in the previous sections in order to derive the final asymptotic expansion for any fixed $\mathcal{D}$. And we show that the coefficients are uniform bounded over all $\mathcal{D}$.
\subsection{Construction of the Iteration Scheme}\label{subsection:Construction_of_the_iteration}
The following part is a sketch of the iteration scheme. Recall the definition~\eqref{eq:defi_of_kernel_function} of $K_{\mathcal D}(s)$ and define the integral operators for any bounded function $\Psi: \mathbb{R}\rightarrow H$,
\begin{equation}\label{equation:integraloperator_J}
        \mathcal{J}_{2}(\Psi)(t) := 2\sqrt{2}\int_{0}^{t}K_{\mathcal D}(t-r)\,e^{-r} \, \Psi(r)\, \mathrm{d}r,\qquad \mathcal{J}_{3}(\Psi)(t) := 2\int_{0}^{t}K_{\mathcal D}(t-r)\,e^{-2r} \, \Psi(r)\, \mathrm{d}r.
\end{equation}
Recall the expression~\eqref{eq:kernel_solution_all_D} and denote
\begin{equation}\label{equation:definition_psi0_R_0}
\psi_0 (t, v, \varphi)
:=
 H_{\mathcal D}(t)I(0,v,\varphi)
+
K_{\mathcal D}(t)I(0,d\pi(a_1)v,\varphi),\quad \mathcal{R}_0(t,v,\varphi):=\frac{1}{\sqrt{2}}\mathcal{J}_{2}\left( G_{v,\varphi}\right)(t).
\end{equation}
Then the average $I(t, v, \varphi)$ can be rewritten as 
\begin{equation}\label{eq:I_initial_formula}
    I(t, v, \varphi) = \psi_0 (t, v, \varphi) +  \mathcal{R}_0(t,v,\varphi).
\end{equation}
Plugging in the formula~\eqref{equataion: G_v} into $\mathcal{R}_0(t,v,\varphi)$, we obtain:
\begin{equation}\label{equation:recursive_I}
    I(t, v, \varphi) = \psi_0 (t, v, \varphi) - \sum\limits_{s=2}^{n} \left(\mathcal{J}_2 \left[I(t, \mathrm{d}\pi(n_{s-1})v, \mathfrak{K}_{k_{s-1}}\varphi)\right] + \mathcal{J}_{3}\left[I(t, \mathrm{d}\pi(n_{s-1})^{2}v, \varphi)\right] \right).
\end{equation}

Using this recursive formula, we can start the iterative procedure. A key point is to ensure that $I(t, \mathrm{d}\pi(n_{s-1})v, \mathfrak{K}_{k_{s-1}}\varphi)$ also satisfies the corresponding ODE system. This is guaranteed by Lemma~\ref{lem:OmegaM_stable_under_n}.

By Lemma~\ref{cor:uniform_ODE_bootstrap_sobolev_growth}, we have $\|G_{v,\varphi}(t)\|_{W_{k,H}} \ll_{k,v,\varphi} (1+t)\,\max\{e^{\Re(\lambda_+)t}, e^{(\lambda + \frac{1}{4})t}\}$, Then we have
\begin{equation*}
    \|\mathcal{J}_{2}\left( G_{v,\varphi}\right)(t)\|_{W_{k,H}} \ll_{k,v,\varphi} (1+t)\,\max\{e^{\Re(\lambda_+)t}, e^{(\lambda + \frac{1}{4})t}\}.
\end{equation*}
One can notice that the improvement of such trivial iteration scheme is limited by $\max\{e^{\Re(\lambda_+)t}, e^{(\lambda + \frac{1}{4})t}\}$. 
Since during this whole paper, $k$ is a fixed integer, we will omit the dependence with $k$ of the implied constant in the later proof. 

We now introduce a more refined iteration scheme. The main idea is to decompose the integral into a limiting part over $[0,\infty)$ and a tail part over $[t,\infty)$. For a suitable function $\Psi$, whenever the integral is convergent, we write
\begin{equation}\label{equation:J_2decomposition}
\mathcal{J}_{2}(\Psi)(t)
=
2\sqrt{2}\int_{0}^{\infty} K_{\mathcal D}(t-r)\,e^{-r}\,\Psi(r)\,\mathrm{d}r
-
2\sqrt{2}\int_{t}^{\infty} K_{\mathcal D}(t-r)\,e^{-r}\,\Psi(r)\,\mathrm{d}r.
\end{equation}
Roughly speaking, the first term is of order $O\bigl(e^{\Re(\lambda_{+})t}\bigr),$
while the second term is of order $O\bigl(e^{\Re(\lambda_{+}-1)t}\bigr).$
Therefore the second term contains an additional exponential decay, which yields improved estimates after each iteration step.

In order to perform the splitting, we must ensure that the integrals of the form $\int_{0}^\infty$ converge. This requires precise control of the kernel $K_{\mathcal D}(s)$ for both positive and negative $s$, which is discussed in Lemma~\ref{lem:upper_bounds_of_H_and_k}.

 Therefore, we split the value of $\sqrt{\mathcal{D}}$ into the following two cases:
\begin{align*}
    \textbf{Case a:}\,\sqrt{\mathcal{D}} \in \mathrm{i}\mathbb{R}_{>0} \cup [0,\tfrac{1}{4}), \quad \textbf{Case b:}\,\sqrt{\mathcal{D}} \in [\tfrac{1}{4}, \infty).
\end{align*}
Recall Proposition~\ref{prop:explicit_solution_of_I}; it provides the reason for the distinction between the following two cases.
\begin{enumerate}
    \item \textbf{Case a:} $\sqrt{\mathcal{D}} \in i\mathbb{R}_{>0} \cup [0,\tfrac{1}{4})$, \textbf{Single decomposition.}
In this case, by Lemma~\ref{lem:upper_bounds_of_H_and_k} and~\ref{cor:uniform_ODE_bootstrap_sobolev_growth}, we have
\begin{equation}\label{eq:sufficient_upper_bound_KD_G}
  \|K_{\mathcal{D}}(t-r)\,e^{-r}\,G_{v,\varphi}(r)\|_{W_{k,H}} \ll_{v,\varphi}  e^{(\frac{1-n}{2} + \frac{1}{4})t}\,(1+r-t)(1+r)\,e^{-\frac{1}{2}r}, ~\text{as}~ r \to \infty.
\end{equation}
Hence in this case, the decomposition in~\eqref{equation:J_2decomposition} with the kernel function $K_{\mathcal{D}}(t)$ is justified. 

\item \textbf{Case b:}$\,\sqrt{\mathcal{D}} \in [\tfrac{1}{4}, \infty)$, \textbf{Double decomposition.} In this case, the decomposition~\eqref{equation:J_2decomposition} is no longer applicable with the kernel function $K_{\mathcal{D}}(t)$ which is integral with $e^{\lambda_{+}t}$ and $e^{\lambda_{-}t}$ together. Moreover, the decompositions of integrals with $e^{\lambda_{+}t}$ and $e^{\lambda_{-}t}$ must be justified separately. The separated integrals will be defined in~\eqref{Formula:Naive_integral_operators} and their estimate will be discussed later in~\eqref{eq:estimation_of_J-_act_G}.
\end{enumerate}

We will proceed the iteration scheme following these different cases of decompositions. 
To facilitate the notations in the iteration scheme, we introduce the following multi-index notations, which will be employed throughout the entire Section~\ref{Section:Asymp_Irreducible_repnt_with_pure_imginary}.
Let $\ell$ denote the number of iterations. During the $\ell$-fold iteration, the sequential summation 
is denoted by \begin{equation*}
    \sum_{s_{1}=2}^{n} \sum_{s_{2}=2}^{n} \cdots \sum_{s_{\ell}=2}^{n}=\sum_{\vec{s} \in S_{\ell}},\ \text{where}\ S_{\ell} := \left\{ (s_{1}, s_{2}, \dots, s_{\ell}) \mid 2 \leq s_{k} \leq n,\ 1 \leq k \leq \ell \right\}.
\end{equation*}

Let 
$ \vec{\mathscr{A}}_{\ell} = (\mathscr{A}_{1}, \mathscr{A}_{2}, \dots, \mathscr{A}_{\ell}) \in \{2,3\}^{\ell}$
be a sequence of selection indices.  
Denote operator maps
\begin{equation*}
U_v(\mathscr{A}_{k}, s_{k}) :=
\begin{cases}
\mathrm{d}\pi(n_{s_{k}-1}), & \text{if } \mathscr{A}_{k} = 2, \\
\mathrm{d}\pi(n_{s_{k}-1})^2, & \text{if } \mathscr{A}_{k} = 3,
\end{cases}
\qquad 
U_{\varphi}(\mathscr{A}_{k}, s_{k}) :=
\begin{cases}
\mathfrak{K}_{k_{s_{k}-1}}, & \text{if } \mathscr{A}_{k} = 2, \\
I, & \text{if } \mathscr{A}_{k} = 3,
\end{cases}
\end{equation*}
where $I$ denotes the identity operator.  
We define the iterated factors
\begin{equation*}
v_{\vec{\mathscr{A}}_{\ell}}^{\vec{s}} 
   \;=\; U_v(\mathscr{A}_{\ell}, s_{\ell}) \cdots U_v(\mathscr{A}_{1}, s_{1})\, v, \qquad
\varphi_{\vec{\mathscr{A}}_{\ell}}^{\vec{s}} 
   \;=\; U_{\varphi}(\mathscr{A}_{\ell}, s_{\ell}) \cdots U_{\varphi}(\mathscr{A}_{1}, s_{1})\, \varphi.
\end{equation*}

Recall the formula~\eqref{equataion: G_v} for $G_{v,\varphi}$. We introduce the shorthand
$  G^{\vec{s}}_{\vec{\mathscr{A}}_{\ell}}(t) 
  := G_{\,v^{\vec{s}}_{\vec{\mathscr{A}}_{\ell}},\, \varphi^{\vec{s}}_{\vec{\mathscr{A}}_{\ell}}}(t).$
Explicitly,
\begin{equation}\label{Equation:vector_G}
  G^{\vec{s}}_{\vec{\mathscr{A}}_{\ell}}(t) 
  = - \sum_{s=2}^{n} \Biggl(
     \sqrt{2}\,    I\!\left(t, \mathrm{d}\pi(n_{s-1})v^{\vec{s}}_{\vec{\mathscr{A}}_{\ell}}, \mathfrak{K}_{k_{s-1}}\varphi^{\vec{s}}_{\vec{\mathscr{A}}_{\ell}}\right)
     + e^{-t}  I\!\left(t, \mathrm{d}\pi(n_{s-1})^{2}v^{\vec{s}}_{\vec{\mathscr{A}}_{\ell}}, \varphi^{\vec{s}}_{\vec{\mathscr{A}}_{\ell}}\right)
     \Biggr).
\end{equation}

And we define the product with a fixed left-to-right order for any type of integral operators $\mathcal{J}_{\mathscr{A}_{m}}$:
\begin{equation*}
\prod_{m=1}^{\ell} \mathcal{J}_{\mathscr{A}_{m}} 
:=
\mathcal{J}_{\mathscr{A}_1}
\circ \cdots \circ
\mathcal{J}_{\mathscr{A}_{\ell}}.
\end{equation*}

\subsection{Single Decomposition}\label{section:Once_decomposition}
\textbf{Case a}:\,$\sqrt{\mathcal{D}} \in i\mathbb{R}_{>0} \cup [0,\tfrac{1}{4})$.
In this section, we use the integral operators $ \mathcal{J}_{i}$ defined in~\eqref{equation:integraloperator_J} in the iteration scheme, for $i \in \{2,3\}$, since we can directly use the single decomposition by the estimate in~\eqref{eq:sufficient_upper_bound_KD_G}. 

The integral operator $ \mathcal{J}_{i}$ can be decomposed into the following operators: 
\begin{equation}\label{equation:integral_operator_ring_tilde_J23}
\begin{aligned}
        \mathring{\mathcal{J}}_{2}(\Psi)(t) := 2\sqrt{2}\int_{0}^{\infty} K_{\mathcal D}(t-r)\, e^{-r} \,\Psi(r) \,\mathrm{d}r,\quad 
          \mathring{\mathcal{J}}_{3}(\Psi)(t) := 2\int_{0}^{\infty} K_{\mathcal D}(t-r)\, e^{-2r} \,\Psi(r) \,\mathrm{d}r,\\
\widetilde{\mathcal{J}}_{2}(\Psi)(t) := 2\sqrt{2}\int_{t}^{\infty} K_{\mathcal D}(t-r)\, e^{-r} \, \Psi(r) \,\mathrm{d}r,   \quad 
\widetilde{\mathcal{J}}_{3}(\Psi)(t) := 2\int_{t}^{\infty} K_{\mathcal D}(t-r)\, e^{-2r} \, \Psi(r) \,\mathrm{d}r,
\end{aligned}
\end{equation}
where $K_{\mathcal D}(s)$ is defined in~\eqref{eq:defi_of_kernel_function}. Denote
\begin{equation*}
    \mathring{\psi}_0(t, v, \varphi):= \frac{1}{\sqrt{2}}\mathring{\mathcal{J}}_{2}(G_{v,\varphi})(t),\qquad \widetilde{\mathcal{R}}_{0}(t,v,\varphi):=-\frac{1}{\sqrt{2}}\widetilde{\mathcal{J}}_{2}(G_{v,\varphi})(t).
\end{equation*}

Plugging the expression~\eqref{equataion: G_v} of $G_{v, \varphi}$ into the formula $\widetilde{\mathcal{R}}_{0}(t,v,\varphi)$ and getting
\begin{align*}
\widetilde{\mathcal{R}}_{0}(t,v,\varphi) = \sum_{s=2}^{n} \left(
\widetilde{\mathcal{J}}_{2}
\left(I(t, \mathrm{d}\pi(n_{s-1})v, \mathfrak{K}_{k_{s-1}}\varphi) \right)  +  
\widetilde{\mathcal{J}}_{3}
\left( I(t, \mathrm{d}\pi(n_{s-1})^{2}v, \varphi) \right)\right).
\end{align*}
Recall the expression~\eqref{equation:definition_psi0_R_0} of $\mathcal{R}_{0}$. We separate the limiting integral term $\mathring{\psi}_{0}$ from $\mathcal{R}_{0}(t,v,\varphi)$ and obtain
\begin{equation*}
    \mathcal{R}_{0}(t,v,\varphi) = \mathring{\psi}_0(t,v,\varphi) + \widetilde{\mathcal{R}}_{0}(t,v,\varphi).
\end{equation*}
 This decomposition precisely removes the fixed order component $e^{\Re(\lambda_+)t}$ from $\mathcal{R}_{0}(t,v,\varphi)$. 
Then we obtain a refined version of the recursive formula~\eqref{equation:recursive_I}:
\begin{equation*}
    I(t, v, \varphi) = \psi_0 (t, v, \varphi) +\mathring{\psi}_0(t, v, \varphi)+\widetilde{\mathcal{R}}_{0}(t,v,\varphi).
\end{equation*}
Define
\begin{align*}
    \widetilde{\psi}_1(t, v, \varphi) &:=  \sum_{s=2}^{n} 
\widetilde{\mathcal{J}}_{2}\left(\psi_0(t,\mathrm{d}\pi(n_{s-1})v, \mathfrak{K}_{k_{s-1}}\varphi)\right) +\sum_{s=2}^{n} 
\widetilde{\mathcal{J}}_{3}
\left(\psi_0(t,\mathrm{d}\pi(n_{s-1})^{2}v, \varphi)\right),\\
\mathring{\psi}_1(t, v, \varphi)&:=
\frac{1}{\sqrt{2}} \sum_{s=2}^{n} 
\widetilde{\mathcal{J}}_{2} \circ\mathring{\mathcal{J}}_{2}\left( G_{\mathrm{d}\pi(n_{s-1})v, \mathfrak{K}_{k_{s-1}}\varphi} \right)(t)  +\sum_{s=2}^{n} 
\widetilde{\mathcal{J}}_{3}  \circ\mathring{\mathcal{J}}_{2}
\left( G_{\mathrm{d}\pi(n_{s-1})^{2}v, \varphi} \right)(t) \notag ,\\
    \widetilde{\mathcal{R}}_{1}(t,v,\varphi) &:=  -\frac{1}{\sqrt{2}}\sum_{s=2}^{n} \widetilde{\mathcal{J}}_2\circ\widetilde{\mathcal{J}}_2
\left( G_{\mathrm{d}\pi(n_{s-1})v, \mathfrak{K}_{k_{s-1}}\varphi} \right)(t)-  \sum_{s=2}^{n} 
\widetilde{\mathcal{J}}_3 \circ
\widetilde{\mathcal{J}}_2
\left( G_{\mathrm{d}\pi(n_{s-1})^{2}v, \varphi} \right)(t)  \notag .
\end{align*}
Then we have the first step of the iteration:
\begin{equation*}
\begin{aligned}
    I(t, v, \varphi) &= \psi_0(t, v, \varphi) + \mathring{\psi}_0(t, v, \varphi) +\widetilde{\mathcal{R}}_{0}(t,v,\varphi) \\
    &= \psi_0(t, v, \varphi)  + \mathring{\psi}_0(t, v, \varphi)+ \widetilde{\psi}_1(t, v, \varphi)+\mathring{\psi}_1(t, v, \varphi)+\widetilde{\mathcal{R}}_{1}(t,v,\varphi).
\end{aligned}
\end{equation*}

To facilitate the notation in the iteration scheme, we define the following integral operators acting on bounded functions $\Psi: \mathbb{R} \to H$,
\begin{equation*}
    \widetilde{\mathfrak{J}}_{\vec{\mathscr{A}}_{\ell}}(\Psi)(t)
:=
\prod_{k=1}^{\ell} \widetilde{\mathcal{J}}_{\mathscr{A}_{k}}(\Psi) (t), \qquad
\mathring{\mathfrak{J}}_{\vec{\mathscr{A}}_{\ell}}(\Psi)(t)
:=
\prod_{k=1}^{\ell-1}  \widetilde{\mathcal{J}}_{\mathscr{A}_{k}} 
\left( \mathring{\mathcal{J}}_{\mathscr{A}_{\ell}} (\Psi)\right)(t).
\end{equation*}

Next, we use the notation for compositions of integral operators to introduce the general recursive formula for the $\ell$-th iteration step. The following definitions are well defined by Lemma~\ref{lem:computation_of_frak_J}.
\begin{defi}\label{Defi:limiting_contributions}
  Let $\ell \geq 0$. The limiting contribution $\mathring{\psi}_{\ell}(t, v, \varphi)$ is given by
\begin{equation*}
\mathring{\psi}_{\ell}(t, v, \varphi) =\frac{1}{\sqrt{2}}\sum_{\vec{s} \in S_{\ell}} \sum_{\vec{\mathscr{A}}_{\ell+1} \in \{2,3\}^{\ell} \times \{2\}} 
\mathring{\psi}^{\vec{s}}_{\vec{\mathscr{A}}_{\ell+1}}(t, v, \varphi),\ \text{where}\ \mathring{\psi}^{\vec{s}}_{\vec{\mathscr{A}}_{\ell+1}}(t, v, \varphi)
=\mathring{\mathfrak{J}}_{\vec{\mathscr{A}}_{\ell+1}}
\left( G^{\vec{s}}_{\vec{\mathscr{A}}_{\ell}} \right)(t) .
\end{equation*}
The remainder term admits the form
\begin{equation*}
\widetilde{\mathcal{R}}_{\ell}(t,v,\varphi) = -\frac{1}{\sqrt{2}}\,
\sum_{\vec{s} \in S_{\ell}} 
\sum_{\vec{\mathscr{A}}_{\ell+1} \in \{2,3\}^{\ell} \times \{2\}} 
\widetilde{\mathcal{R}}^{\vec{s}}_{\vec{\mathscr{A}}_{\ell+1}}(t,v,\varphi),\ \text{where}\ \widetilde{\mathcal{R}}^{\vec{s}}_{\vec{\mathscr{A}}_{\ell+1}}(t,v,\varphi)
= \widetilde{\mathfrak{J}}_{\vec{\mathscr{A}}_{\ell+1}} \left( G^{\vec{s}}_{\vec{\mathscr{A}}_{\ell}} \right)(t).
\end{equation*}
\end{defi}

Next, we consider the main terms arising from the integrals of $\psi_0$.

\begin{defi}\label{Defi:polynomial_contributions}
    Let $\ell \geq 1$. The expression $\widetilde{\psi}_{\ell}(t, v, \varphi)$ admits the decomposition
\begin{equation*}
\widetilde{\psi}_{\ell}(t, v, \varphi) =  \sum_{\vec{s} \in S_{\ell}} \sum_{\vec{\mathscr{A}}_{\ell} \in \{2,3\}^{\ell}}
\widetilde{\psi}^{\vec{s}}_{\vec{\mathscr{A}}_{\ell}}(t, v, \varphi),\ \text{where}\ \widetilde{\psi}^{\vec{s}}_{\vec{\mathscr{A}}_{\ell}}(t, v, \varphi)
= \widetilde{\mathfrak{J}}_{\vec{\mathscr{A}}_{\ell}}(\psi_0(t,v^{\vec{s}}_{\vec{\mathscr{A}}_{\ell}}, \varphi^{\vec{s}}_{\vec{\mathscr{A}}_{\ell}})).
\end{equation*}
\end{defi}

Combining the Definitions~\ref{Defi:limiting_contributions} and~\ref{Defi:polynomial_contributions}, the remainder terms admit the recursive decomposition:
\begin{equation*}
        \widetilde{\mathcal{R}}_{\ell-1}(t,v,\varphi) = \widetilde{\psi}_{\ell}(t, v, \varphi) + \mathring{\psi}_{\ell}(t, v, \varphi) + \widetilde{\mathcal{R}}_{\ell}(t,v,\varphi).
\end{equation*}

\begin{prop}[Asymptotic expansion of $I(t,v,\varphi)$]\label{Prop:Refined_Asymptotic}
Let $\ell \geq 1$. Then the asymptotic expansion of the average $I(t, v, \varphi)$ takes the form
\begin{equation*}
I(t, v, \varphi) 
=  \psi_{0}(t, v ,\varphi)
+ \sum_{m =1}^{\ell} \widetilde{\psi}_{m}(t, v, \varphi) + \sum_{m = 0}^{\ell} \mathring{\psi}_{m}(t, v, \varphi)
+ \widetilde{\mathcal{R}}_{\ell}(t,v,\varphi).
\end{equation*}
\end{prop}

\begin{lem}\label{lem:computation_of_frak_J}
    Let $\ell \geq 0$, for any fixed $\sqrt{\mathcal{D}} \in i\mathbb{R}_{>0} \cup [0,\tfrac{1}{4})$. For any positive bounded functions $\Psi: \mathbb{R} \to \mathbb{R}$ satisfying $  \Psi(t) \ll_{\ell} (1+t)\, e^{(\lambda + \frac{1}{4})t}$, we have
\begin{equation*}
             \widetilde{\mathfrak{J}}_{\vec{\mathscr{A}}_{\ell}}(\Psi)(t) \ll_{\ell} (1+t)\,  e^{(\lambda + \frac{1}{4}-\ell)t}, \quad
         \mathring{\mathfrak{J}}_{\vec{\mathscr{A}}_{\ell}}(\Psi)(t) \ll_{\ell} (1+t)\,  e^{(\lambda + \frac{1}{4}-\ell)t}.
\end{equation*}
\end{lem}
\begin{proof}
Recall the definition of the integral operators in~\eqref{equation:integraloperator_J} and~\eqref{equation:integral_operator_ring_tilde_J23}. One notice that the integral operator $\mathfrak{J}_{3}$, $\widetilde{\mathfrak{J}}_3$ or $\mathring{\mathfrak{J}}_{3}$ has one extra $e^{-t}$ compared to the one indexed by 2. Therefore \begin{equation*}
     \widetilde{\mathfrak{J}}_{\vec{\mathscr{A}}_{\ell}}(\Psi)(t) \ll_{\ell}  \widetilde{\mathfrak{J}}_{\{2\}^{\ell}}(\Psi)(t), \quad \mathring{\mathfrak{J}}_{\vec{\mathscr{A}}_{\ell}}(\Psi)(t) \ll_{\ell}  \mathring{\mathfrak{J}}_{\{2\}^{\ell}}(\Psi)(t).
\end{equation*}   

By Lemma~\ref{lem:upper_bounds_of_H_and_k}, for any fixed $\sqrt{\mathcal{D}} \in i\mathbb{R}_{>0} \cup [0,\tfrac{1}{4})$, 
    \begin{align*}
 K_{\mathcal{D}}(s)
 \ll (1-s)e^{(\lambda - \frac{1}{4})s}, \quad \text{for any}\ s <0.
\end{align*}
By induction with $\ell$, for $0\leq m \leq \ell$,
\begin{align*}
\int_t^{\infty} K_{\mathcal{D}}(t-r)
e^{-r}\,(1+r)\,e^{(\lambda + \frac{1}{4}-m)r} \mathrm{d}r &\ll \int_t^{\infty} (1+r)(1+r-t)e^{(\frac{1-n}{2}-\frac{1}{4})(t-r)} \, e^{(\frac{1-n}{2} + \frac{1}{4}-m-1)r}\mathrm{d}r \\
&\ll (1+t)\, e^{(\frac{1-n}{2} + \frac{1}{4}-m-1)t}.
\end{align*}
Then we arrive at 
\begin{align*}
    \widetilde{\mathfrak{J}}_{\{2\}^{\ell}}\left((1+t)\, e^{(\frac{1-n}{2} + \frac{1}{4})t}\right)  \ll_{\ell} (1+t)\,  e^{(\frac{1-n}{2} + \frac{1}{4}-\ell)t}.
\end{align*}
Since 
\begin{equation*}
    \mathring{\mathcal{J}}_{2} \left((1+t)\,  e^{(\lambda + \frac{1}{4})t}\right)= \int_0^{\infty} K_{\mathcal{D}}(t-r)
e^{-r}\,(1+r)\,  e^{(\lambda + \frac{1}{4})r} \mathrm{d}r  \ll e^{(\frac{1-n}{2} + \frac{1}{4})t}.
\end{equation*}
Therefore 
\begin{align*}
   \mathring{\mathfrak{J}}_{\{2\}^{\ell+1}}\left((1+t)\,  e^{(\lambda + \frac{1}{4})t}\right)= \widetilde{\mathfrak{J}}_{\{2\}^{\ell}}\circ \mathring{\mathcal{J}}_{2} \left((1+t)\,  e^{(\lambda + \frac{1}{4})t}\right)\ll_{\ell} (1+t)\,  e^{(\frac{1-n}{2} + \frac{1}{4}-\ell)t}.
\end{align*}
\end{proof}
 As the step $\ell$ of iteration growing, more integral operators $\widetilde{\mathcal{J}}_{i}$ appear, for $i \in \{2,3\}$. By Lemma~\ref{lem:computation_of_frak_J},  the remainder term $\widetilde{\mathcal{R}}_{\ell}(t,v,\varphi)$ and the main term $\widetilde{\psi}_{\ell}(t, v, \varphi)$ exhibit progressively faster exponential decay, with a leading decay rate bounded by $e^{(\Re(\lambda_{+})-\ell)t}$. We present the details in the following lemmas.

\begin{lem}[Convergence of $\widetilde{\mathcal{R}}_{\ell}(t,v,\varphi)$]\label{coro:Convergence_of_R}
  For any $\sqrt{\mathcal{D}} \in i\mathbb{R}_{>0} \cup [0,\tfrac{1}{4})$, let $\ell \geq 1$, then 
\begin{align*}
   \| \widetilde{\mathcal{R}}_{\ell}(t,v,\varphi)\|_{W_{k,H}} \ll_{\ell} (1+t)\,  e^{(\lambda + \frac{1}{4}-\ell-1)t} \| v\|_{W_{2\ell+2q_0  + k+2,H}}\, \|\varphi\|_{W_{\ell+q_0+1,1}},
\end{align*}
where $q_0$ is defined in Lemma~\ref{cor:uniform_ODE_bootstrap_sobolev_growth}. 
\end{lem}

\begin{proof}
By Lemma~\ref{cor:uniform_ODE_bootstrap_sobolev_growth}, and expression~\eqref{Equation:vector_G} of $G^{\vec{s}}_{\vec{\mathscr{A}}_{\ell}} (t)$, we have
\begin{align*}
 \|G^{\vec{s}}_{\vec{\mathscr{A}}_{\ell}} (t)\|_{W_{k,H}} \ll (1+t)\,e^{(\lambda + \frac{1}{4})t}\, \|v\|_{W_{2\ell+2q_0+k+2,H}}\, \|\varphi\|_{W_{\ell+q_0+1,1}}.
\end{align*}
Recall Definition~\ref{Defi:limiting_contributions}, the remainder term $\widetilde{\mathcal{R}}_{\ell}(t,v,\varphi)$ satisfies
\begin{align}\label{formula:rough_estimation_of_R_refined}
  \|\widetilde{\mathcal{R}}_{\ell}(t,v,\varphi)\|_{W_{k,H}} &\ll_{ \ell} 
 \sum_{\vec{s} \in S_{\ell}} \sum_{\vec{\mathscr{A}}_{\ell+1} \in  \{2,3\}^{\ell} \times \{2\}}
 \|\widetilde{\mathcal{R}}^{\vec{s}}_{\vec{\mathscr{A}}_{\ell+1}}(t,v,\varphi)\|_{W_{k,H}}\notag \\
 &\ll_{ \ell}  \sum_{\vec{s} \in S_{\ell}} \sum_{\vec{\mathscr{A}}_{\ell+1} \in  \{2,3\}^{\ell} \times \{2\}}
\left|\widetilde{\mathfrak{J}}_{\vec{\mathscr{A}}_{\ell+1}}\bigl(\|G^{\vec{s}}_{\vec{\mathscr{A}}_{\ell}} (t)\|_{W_{k,H}}\bigr)\right| \notag\\
 & \ll_{\ell} \sum_{\vec{s} \in S_{\ell}} \sum_{\vec{\mathscr{A}}_{\ell+1} \in  \{2,3\}^{\ell} \times \{2\}}
\left| \widetilde{\mathfrak{J}}_{\vec{\mathscr{A}}_{\ell+1}}\left((1+t)\,  e^{(\lambda + \frac{1}{4})t}\right)\right| \,
\|v\|_{W_{2\ell+2q_0+k+2,H}}\,
\|\varphi\|_{W_{\ell+q_0+1,1}}\notag\\
& \ll_{\ell}
\left| \widetilde{\mathfrak{J}}_{\{2\}^{\ell+1}}\left((1+t)\,  e^{(\lambda + \frac{1}{4})t}\right)\right| \,
\|v\|_{W_{2\ell+2q_0+k+2,H}}\,
\|\varphi\|_{W_{\ell+q_0+1,1}}.
\end{align}
    By the definition of the operators $\mathcal{J}_i$ in~\eqref{equation:integraloperator_J}, $i \in\{2,3\}$, and the recursive relation~\eqref{equation:recursive_I}, one observes that each occurrence of $\mathcal{J}_{3}$ contributes an additional factor $e^{-t}$ and introduces one more derivative on $v$, compared with $\mathcal{J}_{2}$. 

Consequently, in the iteration of length $\ell$, the maximal Sobolev order arises in the extremal case where all operators are of type $\mathcal{J}_{3}$, leading to a loss of $2\ell$ derivatives. On the other hand, the slowest decay is attained in the opposite extremal case where all operators are of type $\mathcal{J}_{2}$, therefore determines the lowest order of the exponent in the expansion.

Plugging Lemma~\ref{lem:computation_of_frak_J} into \eqref{formula:rough_estimation_of_R_refined}, for $\ell \geq 1$, we have
\begin{align*}
 \|\widetilde{\mathcal{R}}_{\ell}(t,v,\varphi)\|_{W_{k,H}} &\ll_{ \ell} (1+t)\,  e^{(\lambda + \frac{1}{4}-\ell-1)t} \,
\|v\|_{W_{2\ell+2q_0+k+2,H}}\,
\|\varphi\|_{W_{\ell+q_0+1,1}}.
\end{align*}
\end{proof}

\begin{lem}\label{coro:Convergence_of_psi}
Under the same assumption in Lemma~\ref{coro:Convergence_of_R}.
\begin{enumerate}
    \item For $\widetilde{\psi}_{m}(t, v, \varphi)$, $ 1 \leq m \leq \ell$, we have
    \begin{align*}
    \| \widetilde{\psi}_{m}(t, v, \varphi)\|_{W_{k,H}} \ll_{m} (1+t)\,  e^{(\lambda + \frac{1}{4}-m)t} \, \|v\|_{W_{2\ell+2q_0 + k+1,H}}\, \|\varphi\|_{W_{\ell+q_0,1}}.
\end{align*}

 \item For $ \mathring{\psi}_{m}(t, v, \varphi)$, $ 0 \leq m \leq \ell$, we have
    \begin{align*}
    \| \mathring{\psi}_{m}(t, v, \varphi)\|_{W_{k,H}} &\ll_{m} (1+t)\,  e^{(\lambda + \frac{1}{4}-m)t} \, \|v\|_{W_{2\ell +2q_0+ k+2,H}}\, \|\varphi\|_{W_{\ell +q_0+1,1}}.
\end{align*}
\end{enumerate}

\end{lem}
\begin{proof}
Recall Definitions~\ref{Defi:limiting_contributions} and~\ref{Defi:polynomial_contributions}. Then by Lemma~\ref{cor:uniform_ODE_bootstrap_sobolev_growth}, similarly as in formula~\eqref{formula:rough_estimation_of_R_refined}, we have
\begin{align*}
   \| \widetilde{\psi}_{m}(t, v, \varphi)\|_{W_{k,H}} &\ll_{\ell} \left|\widetilde{\mathfrak{J}}_{\{2\}^{m}}\left((1+t)\,  e^{(\lambda + \frac{1}{4})t}\right)\right|\, \|v\|_{W_{2m +2q_0+k+1,H}}\, \|\varphi\|_{W_{m +q_0,1}}, \\ 
    \| \mathring{\psi}_{m}(t, v, \varphi)\|_{W_{k,H}} &\ll_{\ell} \left|\mathring{\mathfrak{J}}_{\{2\}^{m+1}}\left((1+t)\,  e^{(\lambda + \frac{1}{4})t}\right)\right|\, \|v\|_{W_{2m +2q_0+k+2,H}}\, \|\varphi\|_{W_{m +q_0+1,1}}.
\end{align*}
By Lemma~\ref{lem:computation_of_frak_J}, we arrive at the claimed estimates. 
\end{proof}

\subsection{Double Decomposition}\label{section:twice_decomposition}
In this section, in order to discuss \textbf{Case b}, for $\sqrt{\mathcal{D}} \in [\tfrac{1}{4}, \infty)$, we have to introduce the following integral operators acting on bounded functions $\Psi: \mathbb{R} \to H$: for $\sqrt{\mathcal{D}} \neq0$,
\begin{equation}\label{Formula:Naive_integral_operators}
\mathcal{J}_{2}^{\pm}(\Psi)(t) := \sqrt{\frac{2}{\mathcal{D}}} e^{\lambda_{\pm}t} \int_0^t e^{(-1-\lambda_{\pm})r} \Psi(r)\,\mathrm{d}r, \qquad \mathcal{J}_{3}^{\pm}(\Psi)(t) := \frac{1}{\sqrt{\mathcal{D}}} e^{\lambda_{\pm}t} \int_0^t e^{(-2-\lambda_{\pm})r} \Psi(r)\,\mathrm{d}r.
\end{equation}
Then we use the integral operators $\mathcal{J}_{i}^{\pm}$, to analyze the asymptotic behavior of $I(t,v,\varphi)$. 

Recall the operator $\mathcal{J}_{i}$ defined in~\eqref{equation:integraloperator_J}, satisfying the following relationship, for $i \in \{2,3\}$,
\begin{equation*}
    \mathcal{J}_{i} = \mathcal{J}_{i}^{-} - \mathcal{J}_{i}^{+}.
\end{equation*}
The convergence of the corresponding integrals
\begin{equation*}
\int_0^\infty e^{(-i+1-\lambda_{\pm}) r} \Psi(r)\,\mathrm{d}r,\quad \int_t^\infty e^{(-i+1-\lambda_{\pm}) r} \Psi(r)\,\mathrm{d}r,
\end{equation*}
must be verified separately for the $+$ and $-$ components.

Recall the expression~\eqref{equation:definition_psi0_R_0}, we rewrite 
\begin{equation}\label{defi:new_symbol_psi_R_53}
    \psi_0(t, v, \varphi):=C^{+}(\mathcal{D},v,\varphi)e^{\lambda_+ t} + C^{-}(\mathcal{D},v,\varphi)e^{\lambda_- t}, \quad \mathcal{R}_0(t,v,\varphi):=-\frac{1}{\sqrt{2}}(\mathcal{J}^{-}_{2}-\mathcal{J}^{+}_{2})\left( G_{v,\varphi}(t)\right),
\end{equation}
where
\begin{equation*}
    C^{+}(\mathcal{D},v,\varphi)= \frac{I(0,d\pi(a_1)v,\varphi)- \lambda_- I(0, v,\varphi)}{\lambda_+-\lambda_-},\qquad C^{-}(\mathcal{D},v,\varphi)= \frac{ \lambda_+ I(0, v,\varphi)-I(0,d\pi(a_1)v,\varphi)}{\lambda_+-\lambda_-}.
\end{equation*}
Recall the expressions~\eqref{equation:definition_psi0_R_0}. Using the integral operators~\eqref{Formula:Naive_integral_operators}, \eqref{eq:I_initial_formula} can be rewritten as:
\begin{equation*}
    I(t,v,\varphi)
=\psi_0 (t, v, \varphi)
+ \mathcal{R}_0(t,v,\varphi).
\end{equation*}
    As in~\eqref{equation:J_2decomposition}, we now begin to decompose the limiting part from $(\mathcal{J}_{2}^{-}-\mathcal{J}_{2}^{+}) (G_{v,\varphi})(t)$. Due to Lemma~\ref{cor:uniform_ODE_bootstrap_sobolev_growth}, for any $\sqrt{\mathcal{D}}\geq\tfrac{1}{4}$, we have the following estimate of the integrand in $\mathcal{J}_{2}^{\pm}(G_{v, \varphi})(t)$:
\begin{align}\label{eq:estimation_of_J-_act_G}
 \text{for}~ \mathcal{J}_{2}^{+}(G_{v, \varphi})(t),\qquad &\| e^{(-1-\lambda_+)t}\,G_{v,\varphi}(t)\|_{W_{k,H}} \ll_{v,\varphi}  e^{(-1-\lambda_+)t}\,(1+t)\,e^{\lambda_+ t} = (1+t)\,e^{-t},\notag\\
       \text{for}~ \mathcal{J}_{2}^{-}(G_{v, \varphi})(t),\qquad &\| e^{(-1-\lambda_-)t}\,G_{v,\varphi}(t)\|_{W_{k,H}} \ll_{v,\varphi}  e^{(-1-\lambda_-)t}\,(1+t)\,e^{\lambda_+ t} = (1+t)\,e^{(2\sqrt{\mathcal{D}}-1)t}.
\end{align}
The estimate~\eqref{eq:estimation_of_J-_act_G} ensures that we can decompose the integral operator $\mathcal{J}_{2}^{+}$ as follows:
\begin{equation*}
    \mathcal{J}_{2}^{+}(G_{v, \varphi})(t) = \sqrt{\frac{2}{\mathcal{D}}} e^{\lambda_+ t}\int_{0}^{\infty}e^{(-1-\lambda_+)r}\,G_{v, \varphi}(r) \,\mathrm{d}r - \sqrt{\frac{2}{\mathcal{D}}} e^{\lambda_+ t}\int_{t}^{\infty} e^{(-1-\lambda_+)r}\, G_{v, \varphi}(r) \,\mathrm{d}r.
\end{equation*}

While the integrand $e^{(-1-\lambda_-)t}\,G_{v,\varphi}(t)$ in $\mathcal{J}_{2}^{-}(G_{v, \varphi})(t)$ may not converge as $t \to \infty$, which depends on the value of $\sqrt{\mathcal{D}}$. Therefore, $\mathcal{J}_{2}^{-}(G_{v, \varphi})(t)$ admits such decomposition only after sufficiently many iteration steps, which we will explain later.

We now define the associated limiting and tail operators by
\begin{equation}\label{Formula:tilde_J_and_0_J}
\begin{aligned}
        \mathring{\mathcal{J}}_{2}^{\pm}(\Psi)(t) &= \sqrt{\frac{2}{\mathcal{D}}} e^{\lambda_{\pm}t}\int_{0}^{\infty}e^{(-1-\lambda_{\pm})r}\,\Psi(r) \,\mathrm{d}r ,\quad
\mathring{\mathcal{J}}_{3}^{\pm}(\Psi)(t) = \frac{1}{\sqrt{\mathcal{D}}} e^{\lambda_{\pm}t}\int_{0}^{\infty}e^{(-2-\lambda_{\pm})r}\,\Psi(r) \,\mathrm{d}r , \\
    \widetilde{\mathcal{J}}_{2}^{\pm}(\Psi)(t) &= \sqrt{\frac{2}{\mathcal{D}}} e^{\lambda_{\pm}t}\int_{t}^{\infty} e^{(-1-\lambda_{\pm})r}\, \Psi(r) \,\mathrm{d}r ,\quad
\widetilde{\mathcal{J}}_{3}^{\pm}(\Psi)(t) = \frac{1}{\sqrt{\mathcal{D}}} e^{\lambda_{\pm}t}\int_{t}^{\infty} e^{(-2-\lambda_{\pm})r}\, \Psi(r) \,\mathrm{d}r.
\end{aligned}
\end{equation}
Now we begin with the first decomposition by rewriting the expression~\eqref{eq:kernel_solution_all_D} as following:
\begin{align}\label{eq:first_refined_iteration_scheme}
I(t,v,\varphi)
&= \psi_0(t, v, \varphi) + \frac{1}{\sqrt{2}}(\mathcal{J}_{2}^{-}-\mathcal{J}_{2}^{+})(G_{v, \varphi})(t)\notag\\
&= \psi_0(t, v, \varphi)+\frac{1}{\sqrt{2}}\mathring{\mathcal{J}}_{2}^{+}(G_{v, \varphi})(t)- \frac{1}{\sqrt{2}}\left(\mathcal{J}_{2}^{-} + \widetilde{\mathcal{J}}_{2}^{+}\right)(G_{v, \varphi})(t)\notag \\
&=\psi_0(t, v, \varphi)+\mathring{\psi}_0(t, v, \varphi)+ \widetilde{\mathcal{R}}_{0}(t,v,\varphi),
\end{align}
where
\begin{align*}
    \mathring{\psi}_0(t, v, \varphi):=\frac{1}{\sqrt{2}}\mathring{\mathcal{J}}_{2}^{+}(G_{v, \varphi})(t), \qquad \widetilde{\mathcal{R}}_{0}(t,v,\varphi):=-\frac{1}{\sqrt{2}}\left(\mathcal{J}_{2}^{-} + \widetilde{\mathcal{J}}_{2}^{+}\right)(G_{v, \varphi})(t).
\end{align*}
By Lemma~\ref{cor:uniform_ODE_bootstrap_sobolev_growth}, the decomposition in~\eqref{eq:first_refined_iteration_scheme} can isolate the limiting term $\mathring{\psi}_0$ as part of the main contribution. Consequently, no additional term  $e^{\lambda_+ t}$ is generated from $\widetilde{\mathcal{R}}_{0}(t,v,\varphi)$ in the subsequent iteration process.

Define
\begin{align*}
    \widetilde{\psi}_1(t, v, \varphi) :=&  \sum_{s=2}^{n} 
\left(\mathcal{J}_{2}^{-} + \widetilde{\mathcal{J}}_{2}^{+}\right)\left(\psi_0(t,\mathrm{d}\pi(n_{s-1})v, \mathfrak{K}_{k_{s-1}}\varphi)\right) +\sum_{s=2}^{n} 
\left(\mathcal{J}_{3}^{-} + \widetilde{\mathcal{J}}_{3}^{+}\right)
\left(\psi_0(t,\mathrm{d}\pi(n_{s-1})^{2}v, \varphi)\right),\\
\mathring{\psi}_1(t, v, \varphi):=&
\frac{1}{\sqrt{2}} \sum_{s=2}^{n} 
\left(\mathcal{J}_{2}^{-} + \widetilde{\mathcal{J}}_{2}^{+}\right) \circ\mathring{\mathcal{J}}_{2}^{+}\left( G_{\mathrm{d}\pi(n_{s-1})v, \mathfrak{K}_{k_{s-1}}\varphi} \right)(t)  \\
&+\sum_{s=2}^{n} 
\left(\mathcal{J}_{3}^{-} + \widetilde{\mathcal{J}}_{3}^{+}\right)  \circ\mathring{\mathcal{J}}_{2}^{+}
\left( G_{\mathrm{d}\pi(n_{s-1})^{2}v, \varphi} \right)(t),\\
    \widetilde{\mathcal{R}}_{1}(t,v,\varphi) :=&  -\frac{1}{\sqrt{2}}\sum_{s=2}^{n} \left(\mathcal{J}_{2}^{-} + \widetilde{\mathcal{J}}_{2}^{+}\right)\circ\left(\mathcal{J}_{2}^{-} + \widetilde{\mathcal{J}}_{2}^{+}\right)
\left( G_{\mathrm{d}\pi(n_{s-1})v, \mathfrak{K}_{k_{s-1}}\varphi} \right)(t) \\
& -\sum_{s=2}^{n} 
\left(\mathcal{J}_{3}^{-} + \widetilde{\mathcal{J}}_{3}^{+}\right)\circ
\left(\mathcal{J}_{2}^{-} + \widetilde{\mathcal{J}}_{2}^{+}\right)
\left(G_{\mathrm{d}\pi(n_{s-1})^{2}v, \varphi}\right)(t).
\end{align*}
Then we have the first step of the iteration:
\begin{equation*}
\begin{aligned}
    I(t, v, \varphi) &= \psi_0(t, v, \varphi) + \mathring{\psi}_0(t, v, \varphi) +\widetilde{\mathcal{R}}_{0}(t,v,\varphi) \\
    &= \psi_0(t, v, \varphi)  + \mathring{\psi}_0(t, v, \varphi)+ \widetilde{\psi}_1(t, v, \varphi)+\mathring{\psi}_1(t, v, \varphi)+\widetilde{\mathcal{R}}_{1}(t,v,\varphi).
\end{aligned}
\end{equation*}

To describe the proceeding iteration scheme, we introduce the following notations and definitions. Let
\begin{equation*}
\ell_+ = \left\lfloor \lambda_{+} - \lambda_{-} \right\rfloor + 2.
\end{equation*}
We define the following integral operators acting on suitable bounded functions $\Psi: \mathbb{R} \to H$: for $i\in\{2,3\}$,
\begin{align*}
 \widetilde{\mathfrak{J}}_{\vec{\mathscr{A}}_{\ell}}(\Psi)(t):=
\prod_{k=1}^{\ell} \left(\mathcal{J}^{-}_{\mathscr{A}_{k}} + \widetilde{\mathcal{J}}^{+}_{\mathscr{A}_{k}} \right)(\Psi)(t),\qquad
    \mathring{\mathfrak{J}}_{\vec{\mathscr{A}}_{\ell}}(\Psi)(t):=
\prod_{k=1}^{\ell-1} \left( \mathcal{J}^{-}_{\mathscr{A}_{k}} + \widetilde{\mathcal{J}}^{+}_{\mathscr{A}_{k}} \right) \circ 
\mathring{\mathcal{J}}^{+}_{\mathscr{A}_{\ell }} (\Psi)(t).
\end{align*}
The following definitions are well defined by Lemma~\ref{lem:computation_of_frak_J_53}.
\begin{defi}\label{definition:twotypepsiandtildeR}
\begin{enumerate}[leftmargin=*,itemindent=0pt]
    \item  Let $1\leq \ell \leq \ell_+$. The main term contribution $\widetilde{\psi}_{\ell}(t, v, \varphi)$ admits the expression
\begin{equation*}
\widetilde{\psi}_{\ell}(t, v, \varphi) = \sum_{\vec{s} \in S_{\ell}} \sum_{\vec{\mathscr{A}}_{\ell} \in \{2,3\}^{\ell}}
\widetilde{\psi}^{\vec{s}}_{\vec{\mathscr{A}}_{\ell}}(t, v ,\varphi),\ \text{where}\ \widetilde{\psi}^{\vec{s}}_{\vec{\mathscr{A}}_{\ell}}(t, v ,\varphi)
=\widetilde{\mathfrak{J}}_{\vec{\mathscr{A}}_{\ell}}\left(\psi_0(t,v^{\vec{s}}_{\vec{\mathscr{A}}_{\ell}}, \varphi^{\vec{s}}_{\vec{\mathscr{A}}_{\ell}})\right).
\end{equation*}
    \item Let $0\leq \ell \leq \ell_+$. The limiting contribution $\mathring{\psi}_{\ell}(t, v, \varphi)$ admits the expression
\begin{equation*}
\mathring{\psi}_{\ell}(t, v, \varphi) = \frac{1}{\sqrt{2}}\sum_{\vec{s} \in S_{\ell}} \sum_{\vec{\mathscr{A}}_{\ell+1} \in \{2,3\}^{\ell} \times \{2\}} 
\mathring{\psi}^{\vec{s}}_{\vec{\mathscr{A}}_{\ell+1}}(t, v ,\varphi),\ \text{where}\ \mathring{\psi}^{\vec{s}}_{\vec{\mathscr{A}}_{\ell+1}}(t, v ,\varphi)
= \mathring{\mathfrak{J}}_{\vec{\mathscr{A}}_{\ell+1}}
\left( G^{\vec{s}}_{\vec{\mathscr{A}}_{\ell}} \right)(t) .
\end{equation*} 
  
\item  Let $0\leq \ell \leq \ell_+$. The remainder term admits the expression
\begin{equation*}
\widetilde{\mathcal{R}}_{\ell}(t, v ,\varphi) = -\frac{1}{\sqrt{2}}
\sum_{\vec{s} \in S_{\ell}} 
\sum_{\vec{\mathscr{A}}_{\ell+1} \in \{2,3\}^{\ell} \times \{2\}} 
\widetilde{\mathcal{R}}^{\vec{s}}_{\vec{\mathscr{A}}_{\ell+1}}(t, v ,\varphi), \text{where}\ \widetilde{\mathcal{R}}^{\vec{s}}_{\vec{\mathscr{A}}_{\ell+1}}(t, v ,\varphi)
= \widetilde{\mathfrak{J}}_{\vec{\mathscr{A}}_{\ell+1}} \left(  G^{\vec{s}}_{\vec{\mathscr{A}}_{\ell}} \right)(t).
\end{equation*}
\end{enumerate}   
\end{defi}
We iteratively apply the decomposition in~\eqref{eq:first_refined_iteration_scheme}, and obtain the recursive relation:
\begin{equation}\label{eq:recursive_relation_in_DD_first}
     \widetilde{\mathcal{R}}_{\ell-1}(t,v,\varphi) = \widetilde{\psi}_{\ell}(t, v, \varphi) +\mathring{\psi}_{\ell}(t, v, \varphi) + \widetilde{\mathcal{R}}_{\ell}(t,v,\varphi), \quad\text{for}\quad 1\leq \ell \leq \ell_+.
\end{equation}
After $\ell_+$ iterations,  we have
\begin{equation*}
   \| \widetilde{\mathcal{R}}_{\ell_+}(t,v,\varphi) \|_{W_{k,H}} \ll_{\ell,v,\varphi} (1+t)\,\max\{e^{(\lambda_+ - \ell_+-1)t}, e^{\lambda_- t}\} \ll_{\ell,v,\varphi} (1+t)\,e^{\lambda_- t},
\end{equation*}
which is enough to justify the second decomposition: $\mathcal{J}_{2}^{-} = \mathring{\mathcal{J}}_{2}^{-} - \widetilde{\mathcal{J}}_{2}^{-}.$
The integral operators are defined in~\eqref{Formula:tilde_J_and_0_J}.
We continue the iteration scheme with inserting the second decomposition and obtain:
\begin{equation}\label{equation:doubledecomposition_sketch}
  \begin{aligned}
       \widetilde{\mathcal{R}}^{\vec{s}}_{\vec{\mathscr{A}}_{\ell_+ +1}}(t, v ,\varphi)
       &=  \widetilde{\mathfrak{J}}_{\vec{\mathscr{A}}_{\ell_+ +1}} \left( G^{\vec{s}}_{\vec{\mathscr{A}}_{\ell_+}} \right)(t) =\left(  \mathring{\mathcal{J}}_{\mathscr{A}_{1}}^{-} - \widetilde{\mathcal{J}}_{\mathscr{A}_{1}}^{-} + \widetilde{\mathcal{J}}_{\mathscr{A}_1}^{+}\right)\circ \widetilde{\mathfrak{J}}_{\vec{\mathscr{A}}_{\ell_+}} \left( G^{\vec{s}}_{\vec{\mathscr{A}}_{\ell_+}} \right)(t)\\
       &=\mathring{\mathcal{J}}_{\mathscr{A}_1}^{-}  \circ \widetilde{\mathfrak{J}}_{\vec{\mathscr{A}}_{\ell_+ }} \left(  G^{\vec{s}}_{\vec{\mathscr{A}}_{\ell_+}} \right)(t) - \left( \widetilde{\mathcal{J}}_{\mathscr{A}_1}^{-}- \widetilde{\mathcal{J}}_{\mathscr{A}_1}^{+}\right) \circ \widetilde{\mathfrak{J}}_{\vec{\mathscr{A}}_{\ell_+ }} \left(  G^{\vec{s}}_{\vec{\mathscr{A}}_{\ell_+}} \right)(t).
  \end{aligned}
\end{equation}

To describe the proceeding iteration scheme, we introduce the following notations and definitions. Define the following integral operators acting on suitable bounded functions $\Psi: \mathbb{R} \to H$: for $i\in\{2,3\}$,
\begin{align}\label{defi:frak_J_in_twice_decomposition}
        \mathring{\widetilde{\mathfrak{J}}}_{\vec{\mathscr{A}}_{\ell}}(\Psi)(t):=&
\prod_{k=1}^{\ell - \ell_+ -1 } \left( \widetilde{\mathcal{J}}^{-}_{\mathscr{A}_{k}} - \widetilde{\mathcal{J}}^{+}_{\mathscr{A}_{k}} \right)
\circ \mathring{\mathcal{J}}^{-}_{\mathscr{A}_{\ell - \ell_+}} \circ 
\prod_{k=\ell - \ell_+ +1}^{\ell} \left(\mathcal{J}^{-}_{\mathscr{A}_{k}} + \widetilde{\mathcal{J}}^{+}_{\mathscr{A}_{k}} \right) (\Psi)(t), \notag\\
      \widetilde{\mathring{\mathfrak{J}}}_{\vec{\mathscr{A}}_{\ell}}(\Psi)(t):=&
\prod_{k=1}^{\ell - \ell_+ -1 } \left( \widetilde{\mathcal{J}}^{-}_{\mathscr{A}_{k}} - \widetilde{\mathcal{J}}^{+}_{\mathscr{A}_{k}} \right) 
\prod_{k=\ell - \ell_+}^{\ell-1} \left(\mathcal{J}^{-}_{\mathscr{A}_{k}} + \widetilde{\mathcal{J}}^{+}_{\mathscr{A}_{k}} \right) 
\circ \mathring{\mathcal{J}}^{-}_{\mathscr{A}_{\ell}} (\Psi)(t),\\
\widetilde{\widetilde{\mathfrak{J}}}_{\vec{\mathscr{A}}_{\ell}}(\Psi)(t):=&
\prod_{k=1}^{\ell - \ell_+} \left( \widetilde{\mathcal{J}}^{-}_{\mathscr{A}_{k}} - \widetilde{\mathcal{J}}^{+}_{\mathscr{A}_{k}} \right)
\prod_{k=\ell - \ell_+ +1}^{\ell} \left(\mathcal{J}^{-}_{\mathscr{A}_{k}} + \widetilde{\mathcal{J}}^{+}_{\mathscr{A}_{k}} \right)(\Psi)(t). \notag
\end{align}
We introduce the definitions of the terms for the general recursive formula, which are well-defined by Lemma~\ref{lem:computation_of_frak_J_53}. 
\begin{defi}\label{definition:manymanypsiandtildetildeR}
\begin{enumerate}[leftmargin=*,itemindent=0pt]
 \item Let $ \ell \geq \ell_+$. The limiting contribution $\mathring{\widetilde{\psi}}_{\ell}(t)$ admits the expression
\begin{equation*}
\mathring{\widetilde{\psi}}_{\ell}(t, v ,\varphi) = \frac{1}{\sqrt{2}}\sum_{\vec{s} \in S_{\ell}} \sum_{\vec{\mathscr{A}}_{\ell+1} \in \{2,3\}^{\ell} \times \{2\}} 
\mathring{\widetilde{\psi}}^{\vec{s}}_{\vec{\mathscr{A}}_{\ell+1}}(t, v ,\varphi), \text{where}\ \mathring{\widetilde{\psi}}^{\vec{s}}_{\vec{\mathscr{A}}_{\ell+1}}(t, v ,\varphi)
= \mathring{\widetilde{\mathfrak{J}}}_{\vec{\mathscr{A}}_{\ell+1}}
\left( G^{\vec{s}}_{\vec{\mathscr{A}}_{\ell}} \right)(t).
\end{equation*}

\item  Let $\ell > \ell_+$. The main term contribution $\widetilde{\widetilde{\psi}}_{\ell}(t)$ admits the expression
\begin{equation*}
\widetilde{\widetilde{\psi}}_{\ell}(t, v ,\varphi)= \sum_{\vec{s} \in S_{\ell}} \sum_{\vec{\mathscr{A}}_{\ell} \in \{2,3\}^{\ell}} 
\widetilde{\widetilde{\psi}}^{\vec{s}}_{\vec{\mathscr{A}}_{\ell}}(t, v ,\varphi), \text{where}\ \widetilde{\widetilde{\psi}}^{\vec{s}}_{\vec{\mathscr{A}}_{\ell}}(t, v ,\varphi)
= \widetilde{\widetilde{\mathfrak{J}}}_{\vec{\mathscr{A}}_{\ell}}\left(\psi_0(t,v^{\vec{s}}_{\vec{\mathscr{A}}_{\ell}}, \varphi^{\vec{s}}_{\vec{\mathscr{A}}_{\ell}})\right).
\end{equation*}

 \item Let $ \ell > \ell_+$. The limiting contribution $\widetilde{\mathring{\psi}}_{\ell}(t)$ admits the expression
\begin{equation*}
\widetilde{\mathring{\psi}}_{\ell}(t, v ,\varphi) = \frac{1}{\sqrt{2}}\sum_{\vec{s} \in S_{\ell}} \sum_{\vec{\mathscr{A}}_{\ell+1} \in \{2,3\}^{\ell} \times \{2\}} 
\widetilde{\mathring{\psi}}^{\vec{s}}_{\vec{\mathscr{A}}_{\ell+1}}(t, v ,\varphi), \text{where}\ \widetilde{\mathring{\psi}}^{\vec{s}}_{\vec{\mathscr{A}}_{\ell+1}}(t, v ,\varphi)
= \widetilde{\mathring{\mathfrak{J}}}_{\vec{\mathscr{A}}_{\ell+1}}
\left( G^{\vec{s}}_{\vec{\mathscr{A}}_{\ell}} \right)(t).
\end{equation*}
  
\item  Let $\ell \geq \ell_+$. The remainder term admits the expression
\begin{equation}\label{equation:tildetildeR_l}
\widetilde{\widetilde{\mathcal{R}}}_{\ell}(t, v ,\varphi) = -\frac{1}{\sqrt{2}}
\sum_{\vec{s} \in S_{\ell}} 
\sum_{\vec{\mathscr{A}}_{\ell+1} \in \{2,3\}^{\ell} \times \{2\}} 
\widetilde{\widetilde{\mathcal{R}}}^{\vec{s}}_{\vec{\mathscr{A}}_{\ell+1}}(t, v ,\varphi), \text{where}\ \widetilde{\widetilde{\mathcal{R}}}^{\vec{s}}_{\vec{\mathscr{A}}_{\ell+1}}(t, v ,\varphi)
= \widetilde{\widetilde{\mathfrak{J}}}_{\vec{\mathscr{A}}_{\ell+1}} \left(  G^{\vec{s}}_{\vec{\mathscr{A}}_{\ell}} \right)(t).
\end{equation}
\end{enumerate}   
\end{defi}

Recall~\eqref{equation:doubledecomposition_sketch}, then we decompose $\widetilde{\mathcal{R}}_{\ell_+}(t,v,\varphi)$ as following:
\begin{equation*}
    \widetilde{\mathcal{R}}_{\ell_+}(t,v,\varphi)= \mathring{\widetilde{\psi}}_{\ell_+}(t, v, \varphi) + \widetilde{\widetilde{\mathcal{R}}}_{\ell_+}(t,v,\varphi).
\end{equation*}
Plugging~\eqref{eq:first_refined_iteration_scheme} into~\eqref{equataion: G_v}, we rewrite $\widetilde{\widetilde{\mathcal{R}}}_{\ell_+ +1}(t,v,\varphi)$ as:
\begin{equation}\label{eq:second_recursive_second_step}
    \widetilde{\widetilde{\mathcal{R}}}_{\ell_+}(t,v,\varphi)= \widetilde{\widetilde{\psi}}_{\ell_+ +1}(t, v, \varphi)+ \widetilde{\mathring{\psi}}_{\ell_+ +1}(t, v, \varphi)+ \mathring{\widetilde{\psi}}_{\ell_+ +1}(t, v, \varphi) + \widetilde{\widetilde{\mathcal{R}}}_{\ell_+ +1}(t,v,\varphi).
\end{equation}
This is what we called double decomposition.

We iteratively apply the decomposition in~\eqref{eq:second_recursive_second_step}, and obtain the following recursive relation:
\begin{alignat*}{2}
      \widetilde{\widetilde{\mathcal{R}}}_{\ell-1}(t,v,\varphi) &= \widetilde{\widetilde{\psi}}_{\ell}(t, v, \varphi)+\widetilde{\mathring{\psi}}_{\ell}(t, v, \varphi) +\mathring{\widetilde{\psi}}_{\ell}(t, v, \varphi) + \widetilde{\widetilde{\mathcal{R}}}_{\ell}(t,v,\varphi), &\quad&\text{for}\quad \ell > \ell_+.
\end{alignat*} 
Combining this above recursive relation with~\eqref{eq:recursive_relation_in_DD_first}, we arrive at the following result.

\begin{prop}[Full asymptotic expansion for $\ell > \ell_+$]\label{prop:Full_asymptotic_of_ODE _real_part}
Let $\ell > \ell_+$. Then the asymptotic expansion of the average $I(t,v,\varphi)$ admits the form
\begin{equation*}
\begin{aligned}
    I(t,v,\varphi) 
=&\psi_{0}(t, v ,\varphi)
+ \sum_{m = 1}^{\ell_+} \widetilde{\psi}_{m}(t, v, \varphi) + \sum_{k = 0}^{\ell_+} \mathring{\psi}_{m}(t, v, \varphi)\\
&+ \sum_{m = \ell_+}^{\ell} \mathring{\widetilde{\psi}}_{m}(t,v,\varphi)+ \sum_{m = \ell_+ + 1}^{\ell} \widetilde{\widetilde{\psi}}_{m}(t, v ,\varphi)  + \sum_{m = \ell_+ +1}^{\ell} \widetilde{\mathring{\psi}}_{m}(t,v,\varphi) + \widetilde{\widetilde{\mathcal{R}}}_{\ell}(t,v,\varphi).
\end{aligned}
\end{equation*}
\end{prop}

Next, we prove the convergence of the main terms
$\psi_{m}$,
$\widetilde{\psi}_{m}$,
$\mathring{\psi}_{m}$,
$\mathring{\widetilde{\psi}}_{m}$,
$\widetilde{\widetilde{\psi}}_{m}$,
and
$\widetilde{\mathring{\psi}}_{m}$,
as well as the remainder term
$\widetilde{\widetilde{\mathcal{R}}}_{\ell}(t,v,\varphi)$,
in the Sobolev space $W_{k}(H)$.

We begin with the following standard integral computations.
\begin{lem}\label{lem:computation_of_frak_J_53}
For any fixed $\mathcal{D} \in \left[\tfrac{1}{4},\infty\right)$, let $\ell > \ell_+$. For any bounded positive functions $\Psi_1, \Psi_2: \mathbb{R} \to \mathbb{R}$ satisfying $\Psi_1 =e^{\lambda_+ t} + e^{\lambda_- t}$, $\Psi_2(t) \ll (1+t)\, e^{\lambda_+ t}$, we have
\begin{alignat*}{2}
 \widetilde{\mathfrak{J}}_{\{2\}^{m}}(\Psi_1)(t)
&\ll_{m}
(1+t)\,\max\{e^{(\lambda_+-m)t}, e^{\lambda_- t} \},
&\qquad& \text{with } 1\leq m \leq \ell_+,
\\
\mathring{\mathfrak{J}}_{\{2\}^{m+1}}(\Psi_2)(t)
&\ll_{m}
(1+t)\,\max\{e^{(\lambda_+-m)t}, e^{\lambda_- t} \},
&\qquad& \text{with } 0\leq m \leq \ell_+,
\\
\widetilde{\widetilde{\mathfrak{J}}}_{\{2\}^{m}}(\Psi_1)(t)
&\ll_{m}
(1+t)\,e^{(\lambda_{+}+2  -m)t},
&\qquad& \text{with } \ell_++1\leq m \leq \ell,
\\
\widetilde{\widetilde{\mathfrak{J}}}_{\{2\}^{m+1}}(\Psi_2)(t)
&\ll_{m}
(1+t)\,e^{(\lambda_{+}+1-m)t},
&\qquad& \text{with } \ell_++1\leq m \leq \ell,
\\
\mathring{\widetilde{\mathfrak{J}}}_{\{2\}^{m+1}}(\Psi_2)(t)
&\ll_{m}
(1+t)\,e^{(\lambda_{+}+2  -m)t},
&\qquad& \text{with } \ell_+\leq m \leq \ell,
\\
\widetilde{\mathring{\mathfrak{J}}}_{\{2\}^{m+1}}(\Psi_2)(t)
&\ll_{m}
(1+t)\,e^{(\lambda_{+}+2  -m)t},
&\qquad& \text{with } \ell_++1\leq m \leq \ell.
    \end{alignat*}
\end{lem}
\begin{proof}
For $\sqrt{\mathcal{D}} \geq  \frac{1}{4}$, we have $\frac{1}{\sqrt{\mathcal{D}}}\leq 4$.  We begin with $\widetilde{\mathfrak{J}}_{\{2\}^{m}}(\Psi_1)(t)$, $ 1\leq m \leq \ell_+$. 
By induction, for $j \geq 0$, we need to deal with the following integrals:
     \begin{equation*}
  \begin{aligned}
            \widetilde{\mathcal{J}}^{+}_2(e^{(\lambda_+ -j)  t}) &= \frac{\sqrt{2}e^{\lambda_{+} t}}{\sqrt{\mathcal{D}}}\int_t^{\infty} e^{-\lambda_{+} r}
e^{-r}\,e^{(\lambda_+ -j)  r} \mathrm{d}r \ll e^{\lambda_{+} t} \int_t^{\infty} e^{-(j+1)r}\mathrm{d}r \ll e^{(\lambda_+ -j-1)t}, \\
        \widetilde{\mathcal{J}}^{+}_2(e^{(\lambda_- -j) t}) &= \frac{\sqrt{2}e^{\lambda_{+} t}}{\sqrt{\mathcal{D}}}\int_t^{\infty} e^{-\lambda_{+} r}
e^{-r}\,e^{(\lambda_- -j)  r} \mathrm{d}r \ll e^{\lambda_{+} t} \int_t^{\infty}e^{-(2\sqrt{\mathcal{D}}+j+1)r} \mathrm{d}r \ll e^{(\lambda_- -j-1)t}, \\
       \mathcal{J}^{-}_2(e^{(\lambda_- -j)  t}) &= \frac{\sqrt{2}e^{\lambda_{-} t}}{\sqrt{\mathcal{D}}}\int_0^{t} e^{-\lambda_{-} r}
e^{-r}\,e^{(\lambda_- -j)  r} \mathrm{d}r \ll e^{\lambda_{-} t} \int_0^{t} e^{-(j+1)r}\mathrm{d}r\ll e^{\lambda_- t},\\
       \mathcal{J}^{-}_2(e^{(\lambda_+ -j)  t}) &= \frac{\sqrt{2}e^{\lambda_{-} t}}{\sqrt{\mathcal{D}}}\int_0^{t} e^{-\lambda_{-} r}
e^{-r}\,e^{(\lambda_+ -j)  r} \mathrm{d}r = \frac{\sqrt{2}e^{\lambda_{-} t}}{\sqrt{\mathcal{D}}} \int_0^{t} e^{(2\sqrt{\mathcal{D}}-j-1)r}\mathrm{d}r.
  \end{aligned}
     \end{equation*}
     Noticed that, the integral $\int_0^{t} e^{(2\sqrt{\mathcal{D}}-j-1)r}\mathrm{d}r$ produce the factor $\frac{1}{2\sqrt{\mathcal{D}}-j-1}$, which does not have uniform bound with $\sqrt{\mathcal{D}}$, when $\sqrt{\mathcal{D}} \geq  \frac{1}{4}$. Hence we need to handle the integral $\int_0^{t} e^{(2\sqrt{\mathcal{D}}-j-1)r}\mathrm{d}r$ more carefully.
     
     For any integer $j \geq 0$ such that $|2\sqrt{\mathcal{D}}-j-1| > 1$, we have 
      \begin{equation*}
       \mathcal{J}^{-}_2(e^{(\lambda_+ -j)  t}) = \frac{\sqrt{2}e^{\lambda_{-} t}}{\sqrt{\mathcal{D}}} \int_0^{t} e^{(2\sqrt{\mathcal{D}}-j-1)r}\mathrm{d}r\ll e^{(\lambda_+ -j-1)t}+ e^{\lambda_- t}.
     \end{equation*}
     And for some integer $j \geq 0$ such that $|2\sqrt{\mathcal{D}}-j-1| \leq 1$, we have 
      \begin{equation*}
       \mathcal{J}^{-}_2(e^{(\lambda_+ -j)  t}) = \frac{\sqrt{2}e^{\lambda_{-} t}}{\sqrt{\mathcal{D}}} \int_0^{t} e^{(2\sqrt{\mathcal{D}}-j-1)r}\mathrm{d}r \ll e^{\lambda_{-} t}\int_0^{t} e^{r}\mathrm{d}r\ll e^{(\lambda_- +1) t} +  e^{\lambda_- t}.
     \end{equation*}
     Then for the next induction step, one has
       \begin{equation*}
\begin{aligned}
          ( \mathcal{J}^{-}_2 +   \widetilde{\mathcal{J}}^{+}_2 )( e^{(\lambda_- +1) t} )&\ll e^{\lambda_{-} t} \int_0^{t} 1\,\mathrm{d}r + e^{\lambda_{+} t} \int_t^{\infty}e^{-2\sqrt{\mathcal{D}} r} \mathrm{d}r \ll (1+t)\,e^{\lambda_- t},\\
      ( \mathcal{J}^{-}_2 +   \widetilde{\mathcal{J}}^{+}_2 )((1+t)\, e^{\lambda_-  t} )&\ll e^{\lambda_{-} t} \int_0^{t} (1+r)\,e^{-r}\,\mathrm{d}r + e^{\lambda_{+} t} \int_t^{\infty}(1+r)\,e^{-(2\sqrt{\mathcal{D}} +1)r} \mathrm{d}r  \ll (1+t)\,e^{\lambda_- t}.
\end{aligned}
     \end{equation*}
          Then we arrive at
          \begin{equation*}
                    \widetilde{\mathfrak{J}}_{\{2\}^{m}}(\Psi_1)(t)=\prod_{k=1}^{\ell} \left(\mathcal{J}^{-}_{2} + \widetilde{\mathcal{J}}^{+}_{2} \right)(e^{\lambda_+ t} + e^{\lambda_- t}) \ll_{m} (1+t)\,\max\{e^{(\lambda_+-m)t}, e^{\lambda_- t} \}.
          \end{equation*}

    Then we compute $\mathring{\mathfrak{J}}_{\{2\}^{m+1}}(\Psi_2)(t)$, where $0\leq m \leq \ell_+$. Since 
        \begin{equation*}
         \mathring{\mathcal{J}}^{+}_2 \left((1+t)\,e^{\lambda_+ t}\right) = \frac{\sqrt{2}e^{\lambda_{+} t}}{\sqrt{\mathcal{D}}}\int_0^{\infty} 
e^{-\lambda_{+} r}
e^{-r}\,(1+r)\,e^{\lambda_+ r} \mathrm{d}r \ll e^{\lambda_{+} t} \int_0^{\infty} (1+r)\,e^{-r}\mathrm{d}r \ll e^{\lambda_+ t},
     \end{equation*}
     we have 
      \begin{equation*}
         \mathring{\mathfrak{J}}_{\{2\}^{m+1}}(\Psi_2)(t) \ll \widetilde{\mathfrak{J}}_{\{2\}^{m}}\circ \mathring{\mathcal{J}}^{+}_2 \left((1+t)\,e^{\lambda_+ t}\right)
               \ll_{m} \widetilde{\mathfrak{J}}_{\{2\}^{m}}(e^{\lambda_+ t})
               \ll_{m}(1+t)\,\max\{e^{(\lambda_+-m)t}, e^{\lambda_- t} \} .
  \end{equation*}
  
  As for $\widetilde{\widetilde{\mathfrak{J}}}_{\{2\}^{m}}(\Psi_1)(t)$ and $\widetilde{\widetilde{\mathfrak{J}}}_{\{2\}^{m+1}}(\Psi_2)(t)$, where $\ell_++1\leq m \leq \ell$. Noticed that
  \begin{equation*}
      \left( \widetilde{\mathcal{J}}^{-}_{2} - \widetilde{\mathcal{J}}^{+}_{2} \right) \left((1+t)\,e^{(\lambda_- -j) t}\right) \ll (1+t)\,e^{(\lambda_- -j-1) t}.
  \end{equation*}
By similar computations as $\widetilde{\mathfrak{J}}_{\{2\}^{m}}(e^{\lambda_+ t})$, we have
     \begin{equation*}
         \widetilde{\mathfrak{J}}_{\{2\}^{\ell_+}}\left((1+t)\,e^{\lambda_+ t}\right) \ll (1+t)\,e^{\lambda_- t}.
     \end{equation*} 
Then we have
   \begin{equation*}
         \widetilde{\widetilde{\mathfrak{J}}}_{\{2\}^{m}}(\Psi_1)(t) = \prod_{k=1}^{m - \ell_+} \left( \widetilde{\mathcal{J}}^{-}_{2} - \widetilde{\mathcal{J}}^{+}_{2} \right) \circ \widetilde{\mathfrak{J}}_{\{2\}^{\ell_+}}(e^{\lambda_+ t}+ e^{\lambda_- t})
                \ll_{m}(1+t)\,e^{(\lambda_-(m-\ell_+))t},  
  \end{equation*}
and
   \begin{equation*}
         \widetilde{\widetilde{\mathfrak{J}}}_{\{2\}^{m+1}}(\Psi_2)(t) = \prod_{k=1}^{m+1 - \ell_+} \left( \widetilde{\mathcal{J}}^{-}_{2} - \widetilde{\mathcal{J}}^{+}_{2} \right) \circ \widetilde{\mathfrak{J}}_{\{2\}^{\ell_+}}\left((1+t)\,e^{\lambda_+ t}\right)
                \ll_{m}(1+t)\,e^{(\lambda_-(m+1-\ell_+))t}.
  \end{equation*}
  
Then for $\mathring{\widetilde{\mathfrak{J}}}_{\{2\}^{m+1}}(\Psi_2)(t)$, where $\ell_+\leq m \leq \ell$.  Noticed that
\begin{equation*}
         \mathring{\mathcal{J}}^{-}_2((1+t)\,e^{\lambda_- t}) = \frac{\sqrt{2}e^{\lambda_{-} t}}{\sqrt{\mathcal{D}}}\int_0^{\infty} 
e^{-\lambda_{-} r}
e^{-r}\,(1+r)\,e^{\lambda_- r} \mathrm{d}r \ll e^{\lambda_{-} t} \int_0^{\infty} (1+r)\,e^{-r}\mathrm{d}r \ll e^{\lambda_- t}.
     \end{equation*} 
And we arrive at
  \begin{equation*}
      \begin{aligned}
            \mathring{\widetilde{\mathfrak{J}}}_{\{2\}^{m+1}}(\Psi_2)(t) &\ll \mathring{\widetilde{\mathfrak{J}}}_{\{2\}^{m+1}}\left((1+t)\,e^{\lambda_+ t}\right) = \prod_{k=1}^{m - \ell_+} \left( \widetilde{\mathcal{J}}^{-}_{2} - \widetilde{\mathcal{J}}^{+}_{2} \right) \circ \mathring{\mathcal{J}}^{-}_{2} \circ  \widetilde{\mathfrak{J}}_{\{2\}^{\ell_+}}\left((1+t)\,e^{\lambda_+ t}\right) \\
                &\ll_{m} \prod_{k=1}^{m - \ell_+} \left( \widetilde{\mathcal{J}}^{-}_{2} - \widetilde{\mathcal{J}}^{+}_{2} \right) \left((1+t)\,e^{\lambda_- t}\right)\ll_{m}(1+t)\,e^{(\lambda_-(m-\ell_+))t}.
          \end{aligned}         
  \end{equation*}
  
Finally for $\widetilde{\mathring{\mathfrak{J}}}_{\{2\}^{m}}(\Psi_1)(t)$, where $\ell_++1\leq m \leq \ell$. We obtain
\begin{equation*}
      \begin{aligned}
            \widetilde{\mathring{\mathfrak{J}}}_{\{2\}^{m+1}}(\Psi_2)(t) &= \prod_{k=1}^{m - \ell_+} \left( \widetilde{\mathcal{J}}^{-}_{2} - \widetilde{\mathcal{J}}^{+}_{2} \right) \circ \mathring{\mathfrak{J}}_{\{2\}^{\ell_+ +1}}\left((1+t)\,e^{\lambda_+ t}\right) \\           &\ll_{m} \prod_{k=1}^{m - \ell_+} \left( \widetilde{\mathcal{J}}^{-}_{2} - \widetilde{\mathcal{J}}^{+}_{2} \right) \left((1+t)\,e^{\lambda_- t}\right)\ll_{m}(1+t)\,e^{(\lambda_- -(m-\ell_+))t}.
          \end{aligned}         
  \end{equation*}
\end{proof}
The main terms that are defined in Definition~\ref{definition:manymanypsiandtildetildeR} satisfy the following upper bounds.
\begin{lem}\label{lem:psi_real_first_part}
Recall  $q_0$ defined in Lemma~\ref{cor:uniform_ODE_bootstrap_sobolev_growth}. For any fixed $\mathcal{D} \in \left[\tfrac{1}{4},\infty\right)$, let $\ell > \ell_+$, we have
\begin{alignat*}{2}
\| \widetilde{\psi}_{m}(t, v, \varphi)\|_{W_{k,H}}
&\ll_{m}
(1+t)\,\max\{e^{(\lambda_+-m)t}, e^{\lambda_- t} \} \,
\|v\|_{W_{2\ell+2q_0+k+1,H}}\,
\|\varphi\|_{W_{\ell+q_0,1}},
&\qquad& \text{with } 1\leq m \leq \ell_+,
\\
\| \mathring{\psi}_{m}(t, v, \varphi)\|_{W_{k,H}}
&\ll_{m}
(1+t)\,\max\{e^{(\lambda_+-m)t}, e^{\lambda_- t} \} \,
\|v\|_{W_{2\ell+2q_0+k+2,H}}\,
\|\varphi\|_{W_{\ell+q_0+1,1}},
&\qquad& \text{with } 0\leq m \leq \ell_+,
\\
\| \widetilde{\widetilde{\psi}}_{m}(t, v ,\varphi) \|_{W_{k,H}}
&\ll_{m}
(1+t)\,e^{(\lambda_{+}+2  -m)t}
\|v\|_{W_{2\ell+2q_0+k+1,H}}\,
\|\varphi\|_{W_{\ell+q_0,1}},
&\qquad& \text{with } \ell_++1\leq m \leq \ell,
\\
\| \mathring{\widetilde{\psi}}_{m}(t,v,\varphi)\|_{W_{k,H}}
&\ll_{m}
(1+t)\,e^{(\lambda_{+}+2  -m)t}
\|v\|_{W_{2\ell+2q_0+k+2,H}}\,
\|\varphi\|_{W_{\ell+q_0+1,1}},
&\qquad& \text{with } \ell_+\leq m \leq \ell,
\\
\| \widetilde{\mathring{\psi}}_{m}(t, v ,\varphi) \|_{W_{k,H}}
&\ll_{m}
(1+t)\,e^{(\lambda_{+}+2  -m)t}
\|v\|_{W_{2\ell+2q_0+k+2,H}}\,
\|\varphi\|_{W_{\ell+q_0+1,1}},
&\qquad& \text{with } \ell_++1\leq m \leq \ell.
    \end{alignat*}
\end{lem}

\begin{proof}
Recall the expression of $\psi_0(t,v,\varphi)$ in~\eqref{defi:new_symbol_psi_R_53}. By the
$t=0$ case of Lemma~\ref{lem:sobolev_boundedness_of_I_explicit}, we have the following estimate that is uniform over $\mathcal{D}\geq \frac{1}{4}$:
\begin{equation*}
\| C^{\pm}(\mathcal{D},v,\varphi)\|_{W_{k,H}} \ll \|v\|_{W_{k+1,H}}\,\|\varphi\|_{L^1(K)},
\end{equation*}
then we obtain
          \begin{equation*}
               \| \psi_0(t,v,\varphi)\|_{W_{k,H}}= \|  C^{+}(\mathcal{D},v,\varphi)e^{\lambda_+ t}+ C^{-}(\mathcal{D},v,\varphi)e^{\lambda_- t}\|_{W_{k,H}} \ll( e^{\lambda_+ t}+ e^{\lambda_- t}) \|v\|_{W_{k+1,H}}\,\|\varphi\|_{L^1(K)}.
          \end{equation*}
          
          Recall the definitions of compositions of integral operators in~\eqref{defi:frak_J_in_twice_decomposition} and Definition~\ref{definition:manymanypsiandtildetildeR}. We begin with $\widetilde{\psi}_{m}(t, v, \varphi)$, for $1\leq m \leq \ell_+$.  By Lemma~\ref{lem:computation_of_frak_J_53}, we arrive at 
  \begin{equation*}
   \begin{aligned}
      \| \widetilde{\psi}_{m}(t, v, \varphi)\|_{W_{k,H}} &\ll_{m} \sum_{\vec{s} \in S_{m}} \sum_{\vec{\mathscr{A}}_{m} \in \{2,3\}^{m}}\| \widetilde{\psi}^{\vec{s}}_{\vec{\mathscr{A}}_{m}}(t) \|_{W_{k,H}} \\
      &\ll_{m} \widetilde{\mathfrak{J}}_{\{2\}^{m}}(e^{\lambda_+ t}+ e^{\lambda_- t}) \|v\|_{W_{2m+2q_0+k+1,H}}\,\|\varphi\|_{W_{m+q_0,1}}\\
         &\ll_{m} (1+t)\,\max\{e^{(\lambda_+-m)t}, e^{\lambda_- t} \} \, \|v\|_{W_{2m+2q_0+k+1,H}}\, \|\varphi\|_{W_{m+q_0,1}}.   
    \end{aligned}
  \end{equation*}

    Then we compute $\mathring{\psi}_{m}(t, v, \varphi)$, where $0\leq m \leq \ell_+$. By Lemma~\ref{cor:uniform_ODE_bootstrap_sobolev_growth}, and expression~\eqref{Equation:vector_G} of $G^{\vec{s}}_{\vec{\mathscr{A}}_{m}} (t)$, we have
\begin{equation*}
     \|G^{\vec{s}}_{\vec{\mathscr{A}}_{m}} (t)\|_{W_{k,H}} \ll (1+t)\,e^{\lambda_+ t} \|v\|_{W_{2m+2q_0+k+2,H}}\, \|\varphi\|_{W_{m+q_0+1,1}}.
\end{equation*}
By Lemma~\ref{lem:computation_of_frak_J_53}, we have
     \begin{equation*}
   \begin{aligned}
          \| \mathring{\psi}_{m}(t, v, \varphi)\|_{W_{k,H}} &\ll_{m} \sum_{\vec{s} \in S_{m}} \sum_{\vec{\mathscr{A}}_{m+1} \in \{2,3\}^{m} \times \{2\}} 
 \|\mathring{\psi}^{\vec{s}}_{\vec{\mathscr{A}}_{m+1}}(t, v ,\varphi) \|_{W_{k,H}} \\
          &\ll_{m} \mathring{\mathfrak{J}}_{\{2\}^{m+1}}\left((1+t)\,e^{\lambda_+ t}\right) \|v\|_{W_{2m+2q_0+k+2,H}}\,\|\varphi\|_{W_{m+q_0+1,1}} \\
               &\ll_{m}(1+t)\,\max\{e^{(\lambda_+-m)t}, e^{\lambda_- t} \} \, \|v\|_{W_{2m+2q_0+k+2,H}}\,\|\varphi\|_{W_{m+q_0+1,1}}.
   \end{aligned}
  \end{equation*}

As for $\widetilde{\widetilde{\psi}}_{m}(t, v ,\varphi)$ with
$\ell_++1\leq m \leq \ell$, similarly to the case of
$\widetilde{\psi}_{m}(t, v, \varphi)$, Lemma~\ref{lem:computation_of_frak_J_53} implies that
   \begin{equation*}
      \begin{aligned}
             \| \widetilde{\widetilde{\psi}}_{m}(t, v, \varphi)\|_{W_{k,H}} &\ll_{m} \widetilde{\widetilde{\mathfrak{J}}}_{\{2\}^{m}}(e^{\lambda_+ t}+ e^{\lambda_- t}) \|v\|_{W_{2m+2q_0+k+1,H}}\,\|\varphi\|_{W_{m+q_0,1}} \\
               &\ll_{m}(1+t)\,e^{(\lambda_- -(m-\ell_+))t} \, \|v\|_{W_{2m+2q_0+k+1,H}}\,\|\varphi\|_{W_{m+q_0,1}}.
          \end{aligned}         
  \end{equation*}
  
Then for $\mathring{\widetilde{\psi}}_{m}(t,v,\varphi)$ with
$\ell_+\leq m \leq \ell$, similarly to the case of
$\mathring{\psi}_{m}(t, v, \varphi)$, Lemma~\ref{lem:computation_of_frak_J_53} yields
  \begin{equation*}
      \begin{aligned}
             \| \mathring{\widetilde{\psi}}_{m}(t, v, \varphi)\|_{W_{k,H}} &\ll_{m} \mathring{\widetilde{\mathfrak{J}}}_{\{2\}^{m+1}}\left((1+t)\,e^{\lambda_+ t}\right) \|v\|_{W_{2m+2q_0+k+2,H}}\,\|\varphi\|_{W_{m+q_0+1,1}} \\
               &\ll_{m}(1+t)\,e^{(\lambda_- -(m-\ell_+))t} \, \|v\|_{W_{2m+2q_0+k+2,H}}\,\|\varphi\|_{W_{m+q_0+1,1}}.
          \end{aligned}         
  \end{equation*}
  
Finally for $\widetilde{\mathring{\psi}}_{m}(t, v ,\varphi)$ with
$\ell_++1\leq m \leq \ell$, similarly to the case of
$\mathring{\psi}_{m}(t, v, \varphi)$, Lemma~\ref{lem:computation_of_frak_J_53} implies that
   \begin{equation*}
      \begin{aligned}
             \| \widetilde{\mathring{\psi}}_{m}(t, v, \varphi)\|_{W_{k,H}} &\ll_{m} \widetilde{\mathring{\mathfrak{J}}}_{\{2\}^{m}}\left((1+t)\,e^{\lambda_+ t}\right) \|v\|_{W_{2m+2q_0+k+2,H}}\,\|\varphi\|_{W_{m+q_0+1,1}} \\
               &\ll_{m}(1+t)\,e^{(\lambda_- -(m-\ell_+))t} \, \|v\|_{W_{2m+2q_0+k+2,H}}\,\|\varphi\|_{W_{m+q_0+1,1}}.
          \end{aligned}         
  \end{equation*}
\end{proof}

And the remainder terms have the following upper bounds.

\begin{lem}\label{lem:R_real_first_part}
Recall  $q_0$ defined in Lemma~\ref{cor:uniform_ODE_bootstrap_sobolev_growth}. For any fixed $\mathcal{D} \in \left[\tfrac{1}{4},\infty\right)$, let $\ell > \ell_+$, we have 
\begin{align*}
   \| \widetilde{\widetilde{\mathcal{R}}}_{\ell}(t,v,\varphi)\|_{W_{k,H}} \ll_{\ell}  (1+t)\,e^{(\lambda_{+} +1 - \ell)t} \|v\|_{W_{2\ell +2q_0+k+2,H}}\, \|\varphi\|_{W_{\ell +q_0+1,1}}.
\end{align*}
\end{lem}

\begin{proof}
Recall the expression~\eqref{equation:tildetildeR_l} of $\widetilde{\widetilde{\mathcal{R}}}_{\ell}(t,v,\varphi)$, by Lemma~\ref{lem:computation_of_frak_J_53}, we have 
\begin{equation*}
\begin{aligned}
         \| \widetilde{\widetilde{\mathcal{R}}}_{\ell}(t,v,\varphi)\|_{W_{k,H}} &\ll_{\ell} \sum_{\vec{s} \in S_{\ell}} \sum_{\vec{\mathscr{A}}_{\ell+1} \in  \{2,3\}^{\ell} \times \{2\}}
 \|\widetilde{\widetilde{\mathcal{R}}}^{\vec{s}}_{\vec{\mathscr{A}}_{\ell+1}}(t,v,\varphi)\|_{W_{k,H}}\\
    &\ll_{\ell} | \widetilde{\widetilde{\mathfrak{J}}}_{\{2\}^{\ell+1}}\left((1+t)\,e^{\lambda_+ t}\right)| \,
\|v\|_{W_{2\ell+2q_0+k+2,H}}\,
\|\varphi\|_{W_{\ell+q_0+1,1}}\\
 &\ll_{\ell}  (1+t)\,e^{(\lambda_- - (\ell+1-\ell_+))t}\,
\|v\|_{W_{2\ell+2q_0+k+2,H}}\,
\|\varphi\|_{W_{\ell+q_0+1,1}}.
\end{aligned}
\end{equation*}
\end{proof}

\subsection{Matrix Coefficient Filtering and Refined Iteration for \texorpdfstring{$\sqrt{\mathcal D}>\nu(\pi)$}{sqrt(D) > nu(pi)}}\label{subsection:Third_real_part}
The goal of this section is to provide a refined version of the argument in Section~\ref{section:twice_decomposition} for the regime $\sqrt{\mathcal D}>\nu(\pi)$, where $\nu(\pi)$ denotes the spectral gap parameter associated to a unitary representation $\pi$ of $G$. 

In this regime, a new difficulty arises, since the term $e^{\lambda_{+}t}$ may become exponentially growing. Consequently, additional boundary conditions are required in order to control the corresponding ODE system. Fortunately, since $I(t, v, \varphi)$ arises from a unitary representation, it satisfies additional representation-theoretic constraints, most notably the matrix coefficient decay estimate from Theorem~\ref{thm:matrix-coeff-decay-eps-form}.

Observe that the main terms $\psi_{m}$, $\widetilde{\psi}_{m}$, $\mathring{\psi}_{m}$, $\widetilde{\mathring{\psi}}_{m}$, $\widetilde{\widetilde{\psi}}_{m}$, and $\mathring{\widetilde{\psi}}_{m}$ are all of the form $e^{\lambda_{\pm}t} P_{2m}(e^{-t})$ where $P_{2m}(\cdot)$ is a polynomial of order no more than $2m$. Moreover, in the half-integer regime where $2\sqrt{\mathcal{D}} \in \mathbb{Z}$, the iteration scheme itself remains unchanged, but the detailed computations differ
\begin{equation*}
    \int_0^t e^{(2\sqrt{\mathcal{D}}-m)t} = t+1,\quad \text{if}\ 2\sqrt{\mathcal{D}}=m.
\end{equation*}

Propositions~\ref{Prop:Refined_Asymptotic} and~\ref{prop:Full_asymptotic_of_ODE _real_part} imply that
$I(t, v, \varphi)$ can be rewritten in the form
\begin{equation}\label{equation:informal_expansion_I}
I(t,v,\varphi)
=
\sum_{m=0}^{2\ell} C_m^{+}(v,\varphi)e^{(\lambda_{+}-m)t}
+
\sum_{m=0}^{2\ell}
\bigl(
C_m^{-}(v,\varphi)
+
t\,C_m^{P}(v,\varphi)
\bigr)e^{(\lambda_{-}-m)t}
+
\widetilde{\widetilde{\mathcal{R}}}_\ell(t,v,\varphi).
\end{equation}
The detailed discussion about such expansion will be in the next Section~\ref{section:asymptotic_expansion}.
Note that $v$ lies in an irreducible representation associated with the spectral parameters $\mu$ and $\upvarpi$, which determine the discriminant $\mathcal{D}$ of the ODE defined in Section~\ref{section:solution_ODE}. Consequently, the coefficients $C^{\pm}_{m}(v,\varphi)$ depend implicitly on the parameter $\mathcal{D}$.

For $\ell_+=\left\lfloor \lambda_{+}-\lambda_- \right\rfloor + 2$, 
\begin{equation}\label{equation:Claim_of_R}
    \|\widetilde{\widetilde{\mathcal{R}}}_\ell(t,v,\varphi)\|_{W_{k,H}} \ll_{\ell, v, \varphi} (1+t)\,e^{(\lambda_- -(\ell+1-\ell_+))t}, \quad \text{for} \ \ell>\ell_+.
\end{equation}

Moreover, one can notice that when $\sqrt{\mathcal D}>\nu(\pi)$, the exponential terms $e^{(\lambda_{+}-m)t}$ appearing in this expansion~\eqref{equation:informal_expansion_I} do not automatically agree with the decay rate predicted by the Matrix Coefficient Theorem~\ref{thm:matrix-coeff-decay-eps-form}. Therefore we need to analyze more carefully.

To address this issue, we divide the analysis according to the size of $\sqrt{\mathcal D}$:
\begin{equation*}
\sqrt{\mathcal D}\in\left[\tfrac{1}{4},\nu(\pi)\right],
\qquad 
\sqrt{\mathcal D}\in\left(\nu(\pi),\infty\right).
\end{equation*}
In particular, if $\nu(\pi)<\tfrac{1}{4}$, one may replace $\tfrac{1}{4}$ by $\nu(\pi)$ as the boundary separating single and double decomposition as discussed in Section~\ref{subsection:Construction_of_the_iteration}. In this situation, the range $\sqrt{\mathcal D}\in\left[\tfrac{1}{4},\nu(\pi)\right]$ becomes empty, which does not affect the analysis.

 This observation leads to the following filtering lemma, based on the Matrix Coefficient Theorem~\ref{thm:matrix-coeff-decay-eps-form}.

\begin{lem}[Filtering of $\lambda_{+}$–terms by matrix coefficient decay]
\label{lem:filtering_by_matrix_coeff_decay}
For  the expansion of $I(t, v, \varphi)$ in~\eqref{equation:informal_expansion_I}, we have
\begin{equation*}
C_m^{+}(v,\varphi)=0,\
\text{whenever }
\lambda_{+}-m>\frac{1-n}{2}+\nu(\pi) .
\end{equation*}
\end{lem}
\begin{proof}
Since $v,w\in C^\infty(X)$ are smooth vectors, they need not be $K$-finite.
However, by the standard approximation argument using
$K$-finite vectors,
the same matrix coefficient estimate in Theorem~\ref{thm:matrix-coeff-decay-eps-form} extends to all
smooth vectors.
Therefore, there exists $q_{v,w}\ge 0$ such that
\begin{equation}
\label{eq:matrix_coeff_decay_filter}
|\langle I(t,v,\varphi),w\rangle|
\ll_{v,w,\varphi}
(1+t)^{q_{v,w}} e^{(\frac{1-n}{2}+\nu(\pi)) t}.
\end{equation}

For any $w\in C^\infty(X)$,
we compute
\begin{align*}
\langle I(t,v,\varphi),w\rangle
&=
\left\langle \int_K \varphi(k)\pi(ka_t)v\,dk , w \right\rangle \\
&=
\int_K \varphi(k)\,
\langle \pi(a_t)v,\pi(k^{-1})w\rangle\,dk
=
\langle \pi(a_t)v, \int_K \overline{\varphi(k)}\,\pi(k^{-1})w\,dk \rangle .
\end{align*}

Assume, for contradiction, that there exists $m$ with $\lambda_{+}-m>\frac{1-n}{2}+\nu(\pi)$ and $C_m^{+}(v,\varphi)\neq0 .$
Choose $m_0$ such that
\begin{equation*}
\lambda_{+}-m_0
=
\max\Bigl\{
\lambda_{+}-m:
C_m^{+}(v,\varphi)\neq0,\
\lambda_{+}-m>\frac{1-n}{2}+\nu(\pi)
\Bigr\}.
\end{equation*}
Pair the expansion with $w=C_{m_0}^{+}(v,\varphi).$
By the explicit construction of the coefficients for a smooth joint eigenvector
$v$, the vector $C_{m_0}^{+}(v,\varphi)$ is smooth.
The $m_0$–term contributes $\|C_{m_0}^{+}(v,\varphi)\|_{H}^2\, e^{(\lambda_{+}-m_0)t}.$

All other $\lambda_{+}$–terms have strictly smaller real exponential rate,
while every $\lambda_{-}$–term satisfies $\lambda_{-}-m\le\frac{1-n}{2}+\nu(\pi)$
and therefore decays strictly faster than $e^{(\lambda_{+}-m_0)t}$.
Since $\ell>\ell_+$, we have $\lambda_- + \ell_+ -\ell\le \frac{1-n}{2}+\nu(\pi) .$
By \eqref{equation:Claim_of_R}, the remainder term satisfies
\begin{equation*}
\widetilde{\widetilde{\mathcal{R}}}_\ell(t,v,\varphi)
=
o\!\left(e^{(\lambda_{+}-m_0)t}\right).
\end{equation*}
Hence
\begin{equation*}
\langle I(t,v,\varphi),
C_{m_0}^{+}(v,\varphi)\rangle
=
\|C_{m_0}^{+}(v,\varphi)\|_{H}^2\, e^{(\lambda_{+}-m_0)t}
+
o\!\left(e^{(\lambda_{+}-m_0)t}\right).
\end{equation*}
For sufficiently large $t$ this implies
\begin{equation*}
|\langle I(t,v,\varphi),
C_{m_0}^{+}(v,\varphi)\rangle|
\ge
\frac12
\|C_{m_0}^{+}(v,\varphi)\|_{H}^2\,
e^{(\lambda_{+}-m_0)t},
\end{equation*}
which contradicts \eqref{eq:matrix_coeff_decay_filter} since
$\lambda_{+}-m_0>\frac{1-n}{2}+\nu(\pi)$.

Therefore
\begin{equation*}
C_m^{+}(v,\varphi)=0,\
\text{whenever }
\lambda_{+}-m>\frac{1-n}{2}+\nu(\pi) .
\end{equation*}
\end{proof}
Especially for $\sqrt{\mathcal D}\in\left(\nu(\pi),\infty\right)$, one has
$\lambda_+ > \frac{1-n}{2}+\nu(\pi)$.
Moreover, the discussion following~\eqref{eq:first_refined_iteration_scheme} shows that no new terms $e^{\lambda_+ t}$ arise during the subsequent iteration process. Therefore, the term $e^{\lambda_+ t}$ can be eliminated prior to the iteration, which leads to the following refined iteration scheme.

Now we rewrite $I(t, v, \varphi)$ in~\eqref{eq:first_refined_iteration_scheme} as
\begin{equation*}
I(t, v, \varphi) = e^{\lambda_{-}t} C^{-}(\mathcal{D},v,\varphi) + e^{\lambda_{+}t} C^{+}(\mathcal{D},v,\varphi) +\widetilde{\mathcal{R}}_0(t,v,\varphi),
\end{equation*}
where the coefficients are defined using modified notations,
\begin{equation*}
\begin{cases}
C^{-}(\mathcal{D},v,\varphi) = \frac{\lambda_{+} I(0,v,\varphi)-I(0,d\pi(a_{1})v,\varphi)}{2\sqrt{\mathcal{D}}}, \\
C^{+}(\mathcal{D},v,\varphi) = \frac{I(0,d\pi(a_{1})v,\varphi)- \lambda_{-} I(0,v,\varphi)}{2\sqrt{\mathcal{D}}}+\frac{1}{\sqrt{\mathcal{D}}} \int_{0}^{\infty} e^{(-1-\lambda_{+})r} G_{v, \varphi}(r) \,\mathrm{d}r.
\end{cases}
\end{equation*}
The remainder term is then given by
\begin{align*}
\widetilde{\mathcal{R}}_0(t,v,\varphi) &= -\frac{1}{\sqrt{\mathcal{D}}} \left( e^{\lambda_{-}t} \int_{0}^{t} e^{(-1-\lambda_{-})r} G_{v, \varphi}(r) \,\mathrm{d}r + e^{\lambda_{+}t} \int_{t}^{\infty} e^{(-1-\lambda_{+})r} G_{v, \varphi}(r) \,\mathrm{d}r\right).
\end{align*}
By Lemma~\ref{cor:uniform_ODE_bootstrap_sobolev_growth}, $\widetilde{\mathcal{R}}_0(t,v,\varphi)$ satisfies the upper bound 
\begin{align*}
  \| \widetilde{\mathcal{R}}_0(t,v,\varphi)\|_{W_{k,H}} \ll_{v,\varphi} (1+t)\,(e^{\lambda_- t} + e^{(\lambda_+-1)t}).
\end{align*}
By Lemma~\ref{lem:filtering_by_matrix_coeff_decay}, we have
$C^{+}(\mathcal{D},v,\varphi)=0$. Therefore,
\begin{equation}\label{equation:D_large_I}
I(t, v, \varphi) = \psi_{0}(t, v ,\varphi) +\widetilde{\mathcal{R}}_0(t,v,\varphi), \qquad \text{where}\ \psi_{0}(t, v ,\varphi) := e^{\lambda_{-}t} C^{-}(\mathcal{D},v,\varphi).
\end{equation}

In Lemma~\ref{cor:uniform_ODE_bootstrap_sobolev_growth}, we proved an estimate of the Sobolev norm of $I(t,v,\varphi)$ using the general solution~\eqref{eq:kernel_solution_all_D} of the ODE to impose additional constraints, valid for all $\mathcal{D}$. This yields the upper bound $e^{\max(\Re (\lambda_{+}),\, k-q)t}$. 
However, when $\sqrt{\mathcal{D}}> k+\frac{n-1}{2}$, the contribution of the factor $e^{\lambda_{+} t}$ produces a weaker estimate for $I(t, v, \varphi)$ in Lemma~\ref{cor:uniform_ODE_bootstrap_sobolev_growth} than the bound obtained in Lemma~\ref{lem:sobolev_boundedness_of_I_explicit}, since in this regime one has $e^{kt} < e^{\lambda_{+} t}$.

Consequently, one must adapt the same argument of Lemma~\ref{cor:uniform_ODE_bootstrap_sobolev_growth} by replacing the expression~\eqref{eq:kernel_solution_all_D} with the alternative expression~\eqref{equation:D_large_I}.

\begin{coro}\label{coro:refined_estimation_of_G}
   Assume the same conditions as in Lemma~\ref{cor:uniform_ODE_bootstrap_sobolev_growth}. For $\sqrt{\mathcal D}\in\left(\nu(\pi),\infty\right)$, then for any fixed $p\in\mathbb Z_{\geq 0}$, define
\begin{equation*}
\gamma_p:=\max\{\lambda_-, k-p\}.
\end{equation*}
   we obtain
\begin{equation*}
\|I(t,v,\varphi)\|_{W_{k,H}}
\ll e^{\gamma_p t}\,
\|v\|_{W_{k+2p,H}}
\|\varphi\|_{W_{p,1}}.
\end{equation*}
Especially for $p=q_0$, we have
\begin{equation*}
\|G_{v,\varphi}(t)\|_{W_{k,H}}
\ll \min\{e^{\lambda_+ t}, 1\}\, \|v\|_{W_{2q_0+k,H}}\, \|\varphi\|_{W_{q_0,1}}.
\end{equation*}
\end{coro}
Now we have
\begin{equation*}
      \| \widetilde{\mathcal{R}}_0(t,v,\varphi)\|_{W_{k,H}} \ll_{v,\varphi} e^{\lambda_- t} + \min\{e^{(\lambda_+ -1) t}, e^{-t}\}.
\end{equation*}
Hence we modify the definition of $\ell_+$ into
\begin{equation*}
    \ell_+ = \left\lfloor - \lambda_{-} \right\rfloor +1.
\end{equation*}

\begin{coro}\label{coro:computation_frak_J_532}
    By analogue with Lemma~\ref{lem:computation_of_frak_J_53}. For any $\mathcal{D} \in (\nu(\pi), \infty)$, let $\ell > \ell_+= \left\lfloor - \lambda_{-} \right\rfloor +1$. For any bounded positive functions $\Psi_1, \Psi_2: \mathbb{R} \to \mathbb{R}$ satisfy $\Psi_1 = e^{\lambda_- t}$, $\Psi_2(t) \ll \min\{e^{\lambda_+ t}, 1\}$, we have
\begin{alignat*}{2}
 \widetilde{\mathfrak{J}}_{\{2\}^{m}}(\Psi_1)(t)
&\ll_{m}
e^{\lambda_- t},
&\qquad& \text{with } 1\leq m \leq \ell_+,
\\
 \widetilde{\mathfrak{J}}_{\{2\}^{m+1}}(\Psi_2)(t)
&\ll_{m}
(1+t)\,\min\{e^{(\lambda_+ -m) t}, e^{-m t}\},
&\qquad& \text{with } 0\leq m \leq \ell_+,
\\
\widetilde{\widetilde{\mathfrak{J}}}_{\{2\}^{m}}(\Psi_1)(t)
&\ll_{m}
e^{-(m-1)t},
&\qquad& \text{with } \ell_++1\leq m \leq \ell,
\\
\widetilde{\widetilde{\mathfrak{J}}}_{\{2\}^{m+1}}(\Psi_2)(t)
&\ll_{m}
(1+t)\,\min\{e^{(\lambda_+ -m) t}, e^{-m t}\},
&\qquad& \text{with } \ell_++1\leq m \leq \ell,
\\
\mathring{\widetilde{\mathfrak{J}}}_{\{2\}^{m+1}}(\Psi_2)(t)
&\ll_{m}
(1+t)\, e^{-(m-1) t},
&\qquad& \text{with } \ell_+\leq m \leq \ell.
  \end{alignat*}
\end{coro}

An important observation is that the threshold parameter $\ell_{+}$, which separates the first and second stages of the double decomposition, depends on $\mathcal{D}$. In the regime $\sqrt{\mathcal{D}} > \nu(\pi)$, if $\mathcal{D}$ is sufficiently large, then for a fixed number $\ell$ of iteration steps, the iteration process never reaches the threshold $\ell_{+}$. 

 As a consequence, the remainder term remains $\widetilde{\mathcal{R}}_{\ell}$ from the first stage of the decomposition. Nevertheless, we will show that $\widetilde{\mathcal{R}}_{\ell}$ still satisfies sufficiently strong decay estimates.

\begin{prop}\label{propositon:doubledecomposition_Dbig}
For any fixed finite number of iterations $\ell$ and any fixed spectral data satisfying $\sqrt{\mathcal{D}}>\nu(\pi)$, if $1 \leq \ell \leq \ell_+$, then the asymptotic expansion of the integral $I(t, v, \varphi)$ admits the form
\begin{equation*}
I(t, v, \varphi) 
=\psi_0(t, v, \varphi)+ \sum_{m=1}^{\ell} \widetilde{\psi}_{m}(t, v, \varphi)
+\widetilde{\mathcal{R}}_{\ell}(t,v,\varphi),
\end{equation*}
where the remainder term satisfies 
 \begin{align*}
         \|\widetilde{\mathcal{R}}_{\ell}(t,v,\varphi)\|_{W_{k,H}} (1+t)\, \ll_{\ell} (1+t)\,\min\{e^{(\lambda_+ -\ell) t}, e^{-\ell t}\}\, \|v\|_{W_{2\ell + 2q_0+k+2,H}}\, \|\varphi\|_{W_{\ell +q_0+1,1}}.
    \end{align*}
If $\ell >\ell_+$, the asymptotic expansion of the integral $I(t, v, \varphi)$ takes the form
\begin{equation*}
I(t, v, \varphi) 
= \psi_0(t, v, \varphi)+ \sum_{m=1}^{\ell_+} \widetilde{\psi}_{m}(t, v, \varphi)
+ \sum_{m =\ell_+ + 1}^{\ell} \widetilde{\widetilde{\psi}}_{m}(t, v ,\varphi)  + \sum_{m = \ell_+}^{\ell} \mathring{\widetilde{\psi}}_{m}(t,v,\varphi)
+ \widetilde{\widetilde{\mathcal{R}}}_{\ell}(t,v,\varphi),
\end{equation*}
where the remainder term satisfies 
 \begin{align*}
         \| \widetilde{\widetilde{\mathcal{R}}}_{\ell}(t,v,\varphi)\|_{W_{k,H}}  \ll_{\ell} (1+t)\,\min\{e^{(\lambda_+ -\ell) t}, e^{-\ell t}\}\, \|v\|_{W_{2\ell + 2q_0+k+2,H}}\, \|\varphi\|_{W_{\ell +q_0+1,1}}.
    \end{align*}
\end{prop}

\begin{proof}
By Corollary~\ref{coro:refined_estimation_of_G}, and expression~\eqref{Equation:vector_G} of $G^{\vec{s}}_{\vec{\mathscr{A}}_{\ell}} (t)$, we have
\begin{equation*}
     \|G^{\vec{s}}_{\vec{\mathscr{A}}_{\ell}} (t)\|_{W_{k,H}} \ll \min\{e^{\lambda_+ t}, 1\}\, \|v\|_{W_{2\ell+2q_0+k+2,H}}\, \|\varphi\|_{W_{\ell+q_0+1,1}}.
\end{equation*}
By Corollary~\ref{coro:computation_frak_J_532}, for $\ell \leq \ell_+$, 
\begin{align*}
 \|\widetilde{\mathcal{R}}_{\ell}(t, v,\varphi) \|_{W_{k,H}} &\ll_{\ell}\ |\widetilde{\mathfrak{J}}_{\{2\}^{\ell+1}}(\min\{e^{\lambda_+ t}, 1\})| \, \|v\|_{W_{2\ell+2q_0 + k+2,H}}\left\| \varphi\right\|_{W_{\ell+q_0+1,1}}\\
 &\ll_{\ell}(1+t)\,\min\{e^{(\lambda_+ -\ell) t}, e^{-\ell t}\}\, \|v\|_{W_{2\ell +2q_0+k+2,H}}\left\| \varphi\right\|_{W_{\ell+q_0+1,1}},
\end{align*}  
and for $\ell > \ell_+$, 
 \begin{align*}
         \| \widetilde{\widetilde{\mathcal{R}}}_{\ell}(t,v,\varphi)\|_{W_{k,H}}  &\ll_{\ell}\ |\widetilde{\widetilde{\mathfrak{J}}}_{\{2\}^{\ell+1}}(\min\{e^{\lambda_+ t}, 1\})| \, \|v\|_{W_{2\ell+2q_0 + k+2,H}}\left\| \varphi\right\|_{W_{\ell+q_0+1,1}}\\
         &\ll_{\ell}(1+t)\,\min\{e^{(\lambda_+ -\ell) t}, e^{-\ell t}\}\,\|v\|_{W_{2\ell +2q_0+k+2,H}}\, \|\varphi\|_{W_{\ell +q_0+1,1}}.
    \end{align*}
\end{proof}
By the same argument as in Lemma~\ref{lem:psi_real_first_part}, together with Corollary~\ref{coro:computation_frak_J_532}, we obtain the following result.

\begin{lem}\label{lem:psi_large_D}
Recall  $q_0$ defined in Lemma~\ref{cor:uniform_ODE_bootstrap_sobolev_growth}. For any $\mathcal{D} \in (\nu(\pi), \infty)$ and any $\ell > \ell_+$, we have
\begin{alignat*}{2}
\| \widetilde{\psi}_{m}(t, v, \varphi)\|_{W_{k,H}}
&\ll_{m}
e^{\lambda_- t}\,
\|v\|_{W_{2\ell+2q_0+k+1,H}}\,
\|\varphi\|_{W_{\ell+q_0,1}},
&\qquad& \text{with } 1\leq m \leq \ell_+,
\\
\| \widetilde{\widetilde{\psi}}_{m}(t, v ,\varphi) \|_{W_{k,H}}
&\ll_{m}
e^{-(m-1)t}
\|v\|_{W_{2\ell+2q_0+k+1,H}}\,
\|\varphi\|_{W_{\ell+q_0,1}},
&\qquad& \text{with } \ell_++1\leq m \leq \ell,
\\
\| \mathring{\widetilde{\psi}}_{m}(t,v,\varphi)\|_{W_{k,H}}
&\ll_{m}
(1+t)\,e^{ -(m-1)t}
\|v\|_{W_{2\ell+2q_0+k+2,H}}\,
\|\varphi\|_{W_{\ell+q_0+1,1}},
&\qquad& \text{with } \ell_+\leq m \leq \ell.
 \end{alignat*}   
\end{lem}

\subsection{Sobolev Regularity of the Asymptotic Expansion}\label{section:asymptotic_expansion}
For each fixed spectral datum $\mathcal{D}$, the asymptotic expansion of $I(t,v,\varphi)$ has already been obtained in Proposition~\ref{Prop:Refined_Asymptotic}, Proposition~\ref{prop:Full_asymptotic_of_ODE _real_part}, and Proposition~\ref{propositon:doubledecomposition_Dbig}, according to the corresponding range of $\mathcal{D}$. We now collect the main terms arising from these three cases and rewrite them in a unified form.
Since these propositions correspond to different spectral regimes, some of the terms below may be absent in certain cases; by convention, such terms are understood to be identically zero. 
\begin{defi}\label{defi:Psi_joint_engenvector}
For any fixed $\ell>\frac{n-1}{2}+\nu(\pi)$, 
We define the total main contribution by
\begin{equation*}
\begin{aligned}
\Psi_{\ell}(t,v,\varphi)
=\;
\psi_{0}(t,v,\varphi)
+
\sum_{m=0}^{2\ell}
\Bigl(&
\widetilde{\psi}_{m}(t,v,\varphi)
+
\mathring{\psi}_{m}(t,v,\varphi)\\
&+
\widetilde{\widetilde{\psi}}_{m}(t,v,\varphi)
+
\widetilde{\mathring{\psi}}_{m}(t, v ,\varphi)
+
\mathring{\widetilde{\psi}}_{m}(t,v,\varphi)
\Bigr).
\end{aligned}
\end{equation*}
\end{defi}

And we have the following uniform upper bound of the main terms $ \Psi_{\ell}(t,v,\varphi) $.
\begin{lem}\label{lem:uniform_upper_bound_of_psi}
Recall  $q_0$ defined in Lemma~\ref{cor:uniform_ODE_bootstrap_sobolev_growth}. For any $\mathcal{D} \in \mathbb{R}$, any fixed $\ell>\frac{n-1}{2}+\nu(\pi)$, we have
\begin{equation*}
     \|\Psi_{\ell}(t,v,\varphi) \|_{W_{k,H}}
\ll_{\ell} (1+t)\,e^{(\frac{1-n}{2}+\nu(\pi))t}\,
\|v\|_{W_{2\ell+2q_0+k+2,H}}\,
\|\varphi\|_{W_{\ell+q_0+1,1}}.
\end{equation*}
\end{lem}
\begin{proof}
    Combining Lemma~\ref{coro:Convergence_of_psi}, ~\ref{lem:psi_real_first_part} and~\ref{lem:psi_large_D}, for any $\mathcal{D} \in \mathbb{R}$, if $\ell>\ell_+=\left\lfloor - \lambda_{-} \right\rfloor +1, $ we have
\begin{align*}
\|\psi_{0}(t,v,\varphi)\|_{W_{k,H}}
\ll\ \ &(1+t)\,e^{(\frac{1-n}{2}+\nu(\Gamma))t}\,
\|v\|_{W_{2\ell+2q_0+k+2,H}}\,
\|\varphi\|_{W_{\ell+q_0+1,1}},\\
\|\overset{\star}{\psi}_{m}(t,v,\varphi)\|_{W_{k,H}}
\ll_{m} &(1+t)\,e^{(\frac{1-n}{2}+\nu(\Gamma))t}\,
\|v\|_{W_{2\ell+2q_0+k+2,H}}\,
\|\varphi\|_{W_{\ell+q_0+1,1}},
\end{align*}
where $\overset{\star}{\psi}_{m}$ denotes any of the terms $\widetilde{\psi}_{m},
\mathring{\psi}_{m},
\widetilde{\widetilde{\psi}}_{m},
\mathring{\widetilde{\psi}}_{m},$ and $\widetilde{\mathring{\psi}}_{m},$
whenever the corresponding term is defined in its admissible range of indices.
   
    Especially for $\sqrt{\mathcal{D}} > \nu(\pi)$, and $1\leq \ell \leq \ell_+$, as in Lemma~\ref{lem:psi_large_D},
\begin{equation*}
\| \widetilde{\psi}_{m}(t, v, \varphi)\|_{W_{k,H}}
\ll_{m} 
e^{\lambda_- t}\,
\|v\|_{W_{2\ell+2q_0+k+2,H}}\,
\|\varphi\|_{W_{\ell+q_0+1,1}},
\qquad \text{with } 1\leq m \leq \ell_+.
\end{equation*}
Therefore, for any $\mathcal{D}\in \mathbb{R}, \ell>\frac{n-1}{2}+\nu(\pi)$, in the $\ell$-step iterations, all $\Psi_{\ell}(t,v,\varphi) $ satisfy:
\begin{equation*}
     \|\Psi_{\ell}(t,v,\varphi) \|_{W_{k,H}}
\ll_{\ell} (1+t)\,e^{(\frac{1-n}{2}+\nu(\pi))t}\,
\|v\|_{W_{2\ell+2q_0+k+2,H}}\,
\|\varphi\|_{W_{\ell+q_0+1,1}}.
\end{equation*}
\end{proof}

For any $\mathcal{D}\in \mathbb{R}, \ell>\frac{n-1}{2}+\nu(\pi)$, two types of remainder terms arise: the term $\widetilde{\mathcal{R}}_{\ell}$ obtained in Section~\ref{section:Once_decomposition}, and the term $\widetilde{\widetilde{\mathcal{R}}}_{\ell}$ obtained in Section~\ref{section:twice_decomposition}. And in Section~\ref{subsection:Third_real_part}, as discussed in Proposition~\ref{propositon:doubledecomposition_Dbig}, both types of remainder terms may appear depending on the spectral parameter $\mathcal{D}$.

For simplicity of notation, we denote by $\mathcal{R}_{\ell}(t, v, \varphi)$ the final remainder term corresponding to a given $\mathcal{D}$, regardless of the stage of the iteration from which it arises. The remainder term $\mathcal{R}_{\ell}(t, v, \varphi)$ satisfies the following uniform upper bounds.
\begin{lem}\label{lem:uniform_upper_bound_R}
    Recall  $q_0$ defined in Lemma~\ref{cor:uniform_ODE_bootstrap_sobolev_growth}. For any $\mathcal{D}\in \mathbb{R}$, and any fixed $\ell>\frac{n-1}{2}+\nu(\pi)$, we have
    \begin{equation*}
\|\mathcal{R}_{\ell}(t,v,\varphi)\|_{W_{k,H}}
\ll_{\ell}
(1+t)\,e^{(1-\ell)t}\,
\|v\|_{W_{2\ell+2q_0+k+2,H}}\,
\|\varphi\|_{W_{\ell+q_0+1,1}}.
\end{equation*}
\end{lem}
\begin{proof}
For any $\mathcal{D} \in \mathbb{R}$, if $\ell>\ell_+=\left\lfloor - \lambda_{-} \right\rfloor +1, $ we have
\begin{equation*}
\|\mathcal{R}_{\ell}(t,v,\varphi)\|_{W_{k,H}}
\ll_{\ell}
(1+t)\,e^{(1-\ell)t}\,
\|v\|_{W_{2\ell+2q_0+k+2,H}}\,
\|\varphi\|_{W_{\ell+q_0+1,1}}.
\end{equation*}
Especially for $\sqrt{\mathcal{D}} > \nu(\pi)$, and $1\leq \ell \leq \ell_+$, as in Proposition~\ref{propositon:doubledecomposition_Dbig},
\begin{equation*}
\|\mathcal{R}_{\ell}(t,v,\varphi)\|_{W_{k,H}}
\ll_{\ell}
(1+t)\,e^{(1-\ell)t}\,
\|v\|_{W_{2\ell+2q_0+k+2,H}}\,
\|\varphi\|_{W_{\ell+q_0+1,1}}.
\end{equation*}
Therefore, for any $\mathcal{D} \in \mathbb{R}, \ell>\frac{n-1}{2}+\nu(\pi)$, in the $\ell$-step iterations, the remainder terms satisfy:
\begin{equation*}
\|\mathcal{R}_{\ell}(t,v,\varphi)\|_{W_{k,H}}
\ll_{\ell}
(1+t)\,e^{(1-\ell)t}\,
\|v\|_{W_{2\ell+2q_0+k+2,H}}\,
\|\varphi\|_{W_{\ell+q_0+1,1}}.
\end{equation*}
\end{proof}

After evaluating all the integrals and regrouping the contributions from
$\Psi_{\ell}(t,v,\varphi)$,
$I\left(t,v,\varphi\right)$ can be rewritten as 
\begin{equation}\label{equation:Expansion_of_I_sec5}
\begin{aligned}
    I(t,v,\varphi)&=\Psi_{\ell}(t,v,\varphi)+\mathcal{R}_{\ell}(t,v,\varphi) \\
&= \sum_{m=0}^{2\ell} \left( e^{(\lambda_{-} - m)t} C^{-}_{m}(v,\varphi)+ t\,e^{(\lambda_- - m)t} C^{P}_{m}(v,\varphi) + e^{(\lambda_{+} - m)t} C^{+}_{m}(v,\varphi) \right)
+ \mathcal{R}_{\ell}(t,v,\varphi) ,
\end{aligned}
\end{equation}
where each coefficient $C^{\pm}_{m}(v,\varphi)$ depends on the vector $v$ and the density function $\varphi$, and $C^{+}_{m}(v,\varphi) =0$ whenever $\sqrt{\mathcal{D}} > \nu(\pi)$. 

The rest step is to show that the coefficients $C^{\pm}_{m}(v,\varphi)$ belong to the Sobolev space  $W_{k}(H)$.

\begin{lem}\label{lem:matrix_invertible}
Let $\beta_{1},\dots,\beta_{s}\in\mathbb{R}$ be pairwise distinct, let $m_{1},\dots,m_{s}\in\mathbb{N},$ $
N:=\sum\limits_{j=1}^{s}m_{j},$
and let $\tau_{1},\dots,\tau_{N}\in\mathbb{R}$ be pairwise distinct. Define the scalar $N\times N$ matrix
\begin{equation*}
M_{\tau}
=
\bigl(\tau_i^r e^{\beta_j \tau_i}\bigr)_{
\substack{
1\le i\le N\\
1\le j\le s,\ 0\le r\le m_j-1
}},
\end{equation*}
where the columns are ordered first by $j$ and then by $r$. Then $M_{\tau}$ is invertible.
\end{lem}

\begin{coro}[Sobolev regularity of the coefficients]
\label{coro:Sobolev_regularity_of_polynomial_exponential_coefficients}
Keep the notation of Lemma~\ref{lem:matrix_invertible}, and set $N:=\sum_{j=1}^{s}m_j .$
Let $(\pi,H)$ be a unitary representation of $G$, and let $W_k(H)$ be the Sobolev space of order $k$ associated with this representation. Suppose
\begin{equation*}
F(t)=\sum_{j=1}^{s}\sum_{r=0}^{m_j-1} c_{j,r}\,t^{r}e^{\beta_j t},
\qquad c_{j,r}\in H .
\end{equation*}
Assume that there exist pairwise distinct real numbers $\tau_1,\dots,\tau_N\in\mathbb R$ such that
\begin{equation*}
F(\tau_i)\in W_k(H),
\qquad 1\le i\le N .
\end{equation*}
Then
\begin{equation*}
c_{j,r}\in W_k(H),
\qquad
1\le j\le s,\quad 0\le r\le m_j-1 .
\end{equation*}
\end{coro}

\begin{proof}
Consider the matrix
\begin{equation*}
M_{\tau}
=
\bigl(\tau_i^r e^{\beta_j\tau_i}\bigr)_{
\substack{
1\le i\le N\\
1\le j\le s,\ 0\le r\le m_j-1
}},
\end{equation*}
where the columns are ordered first by $j$ and then by $r$. By Lemma~\ref{lem:matrix_invertible}, the matrix $M_{\tau}$ is invertible.

Order the coefficient family $(c_{j,r})_{1\le j\le s,\ 0\le r\le m_j-1}$ in the same way as the columns of $M_{\tau}$, and regard it as a vector in $H^N$. Define
\begin{equation*}
y:=\bigl(F(\tau_1),\dots,F(\tau_N)\bigr)\in W_k(H)^N .
\end{equation*}
For every $1\le i\le N$, one has
\begin{equation*}
F(\tau_i)
=
\sum_{j=1}^{s}\sum_{r=0}^{m_j-1}
c_{j,r}\,\tau_i^{r}e^{\beta_j\tau_i}.
\end{equation*}
Hence, in $H^N$, $y=M_{\tau}(c_{j,r}).$
Since $M_{\tau}$ is invertible, it follows that $(c_{j,r})=M_{\tau}^{-1}y .$
The entries of $M_{\tau}^{-1}$ are complex scalars, and $W_k(H)$ is a complex linear subspace of $H$. Since $y\in W_k(H)^N$, we obtain
\begin{equation*}
(c_{j,r})\in W_k(H)^N .
\end{equation*}
Therefore
\begin{equation*}
c_{j,r}\in W_k(H),
\qquad
1\le j\le s,\quad 0\le r\le m_j-1 .
\end{equation*}
\end{proof}

Combining the discussion above, we obtain the final theorem for the asymptotic expansion of $I(t, v, \varphi)$ where $v$ is a joint eigenvector.

\begin{thm}
For a unitary representation $(\pi, H)$, consider a smooth joint eigenvector $v \in H^{\infty}$ satisfying: 
\begin{equation*}
d\pi(\Omega_{G})v=\mu v,
\qquad
d\pi(\Omega_{M})v=\upvarpi v .
\end{equation*}
Then for any eigenvalue $\mu$, $\upvarpi$, which determines discriminant $\mathcal{D}(\mu, \upvarpi)$ defined~\eqref{equation:discriminant},
and for any integer $\ell > \tfrac{n-1}{2}+\nu(\pi)$,  the sector average admits the following asymptotic expansion:
    \begin{equation*}
    I(t,v,\varphi)= \sum_{m=0}^{2\ell} \left( e^{(\lambda_{-} - m)t} C^{-}_{m}(v,\varphi)+ t\,e^{(\lambda_- - m)t} C^{P}_{m}(v,\varphi) + e^{(\lambda_{+} - m)t} C^{+}_{m}(v,\varphi) \right)
+\mathcal{R}_{\ell}(t,v,\varphi),
\end{equation*}
which converges in  $W_{k}(H)$, and for any $m$ and any $\star\in \{\pm, P\}$, $C^{\star}_{m}(v, \varphi)\in W_{k}(H)$.
\end{thm}

\section{Asymptotic Expansion of \texorpdfstring{$I(t,v,\varphi)$}{I(t,v,varphi)} for Arbitrary Vector \texorpdfstring{$v$}{v}}\label{sec:global_asymptotics}
In this section, we study the right regular representation on $L^{2}(X)$, where $X=\Gamma\backslash G$ is compact. Using the asymptotic expansion for joint eigenvector $v_{\mu, \upvarpi}$ associated to the Casimir eigenvalue $\mu$ of $\Omega_{G}$ and Casimir eigenvalue $ \upvarpi$ of $\Omega_{M}$ in Section~\ref{Section:Asymp_Irreducible_repnt_with_pure_imginary} and the Weyl law discussed in Section~\ref{section:weyl-type-upper-bound}, we obtain the final asymptotic expansion for arbitrary vector $v \in L^{2}(X)$.

Now we begin with decomposing an arbitrary vector $v$ into summation of joint eigenvector $v_{\mu, \upvarpi}$.
Recall $1+\mathfrak L_\Delta=1-\mathfrak R_{\Omega_G}+2\mathfrak R_{\Omega_K}$.
Since $1+\mathfrak L_\Delta$ is positive, essentially self-adjoint, and elliptic on the compact space $X=\Gamma\backslash G$, its spectrum is discrete, and we have the
orthogonal decomposition
\begin{equation*}
L^2(X)=\widehat{\bigoplus_{\Lambda\in\mathrm{Spec}(1+\mathfrak L_\Delta)}} E_\Lambda(1+\mathfrak L_\Delta),
\end{equation*}
where each eigenspace $E_\Lambda(1+\mathfrak L_\Delta)$ is finite-dimensional.

Because $\mathfrak R_{\Omega_G}$ and $\mathfrak R_{\Omega_M}$ commute with $1+\mathfrak L_\Delta$, each $E_\Lambda(1+\mathfrak L_\Delta)$ is invariant
under both $\mathfrak R_{\Omega_G}$ and $\mathfrak R_{\Omega_M}$. Since $E_\Lambda(1+\mathfrak L_\Delta)$ is a finite-dimensional
Hermitian space and the restrictions of $\mathfrak R_{\Omega_G}$ and $\mathfrak R_{\Omega_M}$ to $E_\Lambda(1+\mathfrak L_\Delta)$
are commuting self-adjoint endomorphisms, they are simultaneously diagonalizable.
Hence for each $\Lambda$ one has an orthogonal decomposition
\begin{equation*}
E_\Lambda(1+\mathfrak L_\Delta)
=
\bigoplus_{(\mu,\varpi)\in\mathcal J(\Lambda)} E_{\Lambda,\mu,\varpi},
\end{equation*}
where $\mathcal J(\Lambda)$ is the finite set of joint eigenvalue pairs
$(\mu,\varpi)$ occurring in $E_\Lambda(1+\mathfrak L_\Delta)$, and where on each block
$E_{\Lambda,\mu,\varpi}$ one has
\begin{equation*}
1+\mathfrak L_\Delta=\Lambda\,\mathrm{Id},
\qquad
\mathfrak R_{\Omega_G}=\mu\,\mathrm{Id},
\qquad
\mathfrak R_{\Omega_M}=\varpi\,\mathrm{Id}.
\end{equation*}

By elliptic regularity, every eigenspace $E_\Lambda(1+\mathfrak L_\Delta)$ consists of smooth vectors,
hence
\begin{equation*}
E_{\Lambda,\mu,\varpi}\subset E_\Lambda(1+\mathfrak L_\Delta)\subset C^\infty(X).
\end{equation*}

To establish the convergence of $I(t,v,\varphi)$ for an arbitrary vector $v$ from a general unitary representation,
we proceed in several steps.
First, we decompose $v$ into joint eigenvector $v_{\mu, \upvarpi}$:
\begin{equation*}
v = \underset{(\mu, \upvarpi)}{\sum} v_{\mu, \upvarpi},\qquad v_{\mu, \upvarpi}\in E_{\Lambda,\mu,\upvarpi}.
\end{equation*}
Recall the discriminant $\mathcal{D}(\mu, \upvarpi):=\mathcal{D} $ defined in~\eqref{equation:discriminant}.

\subsection{Convergence of the Spectral Summation}
In this section, we combine the Weyl-type upper bounds from Section~\ref{section:weyl-type-upper-bound} with the estimate of $\Psi_{\ell}(t,v_{\nu,\upvarpi},\varphi)$ and $\mathcal{R}_{\ell}(t,v_{\nu,\upvarpi},\varphi)$ established in Section~\ref{section:asymptotic_expansion} to prove the convergence of the spectral sums over $(\mu, \upvarpi)$ of the main term and remainder term in the asymptotic expansion of $I(t,v,\varphi)$.

We fix an integer $k > \frac{n(n+1)}{4}$ and an integer $\ell >\frac{n-1}{2}+\nu(\Gamma)$ for this section.  Associated with the lattice $\Gamma$ is the spectral gap parameter $\nu(\Gamma)$ introduced in Section~\ref{representationtheory}. 
Throughout the paper, the Sobolev norm of $\varphi$ is required only up to order $\ell + q_0 + 1$. For simplicity, we introduce $\beta_1$ to uniformly control the Sobolev norms of both $v$ and $\varphi$.

\begin{defi}
For any fixed $\ell>\frac{n-1}{2}+\nu(\Gamma)$, 
we define the total main contribution by
\begin{equation*}
\mathcal{R}_{\ell}(t,v,\varphi)= \underset{(\mu, \upvarpi)}{\sum}\mathcal{R}_{\ell}(t,v_{\mu, \upvarpi},\varphi),
\end{equation*}
where $\mathcal{R}_{\ell}(t,v_{\mu, \upvarpi},\varphi)$ is introduced in Lemma~\ref{lem:uniform_upper_bound_R}.
\end{defi}

\begin{prop}[Convergence of $\mathcal{R}_{\ell}(t,v,\varphi)$]
\label{prop:convergence_of_tilde_R}
Recall $q_0$ defined in Lemma~\ref{cor:uniform_ODE_bootstrap_sobolev_growth}. Set
\begin{equation*}
d_{Y}:=\dim(\Gamma\backslash G/K),
\qquad
\beta_1:=2\ell+q_0+k+2+\frac{d_{Y}+d_{K}}{2}.
\end{equation*}
Fix an integer $\beta>\beta_1$.
Then for any $v\in W_{\beta,2}(X)$ orthogonal to constants and any smooth
density function $\varphi\in C^\infty(K)$, one has
\begin{equation*}
\|\mathcal{R}_{\ell}(t,v,\varphi)\|_{W_{k,2}}
\ll_{\ell,\Gamma,\beta}
(1+t)\,e^{(1-\ell)t}\,
\|v\|_{W_{\beta,2}}\,
\|\varphi\|_{W_{\beta,1}}.
\end{equation*}
\end{prop}

\begin{proof}
Let $P_{\Lambda,\mu,\varpi}$ be the orthogonal projection onto
$E_{\Lambda,\mu,\varpi}$, and write
\begin{equation*}
v=\sum_{\Lambda}\ \sum_{(\mu,\varpi)\in\mathcal J(\Lambda)} v_{\Lambda,\mu,\varpi},
\qquad
v_{\Lambda,\mu,\varpi}:=P_{\Lambda,\mu,\varpi}v.
\end{equation*}
Moreover,
\begin{equation}
\label{eq:sobolev_decomp_prop_final_corrected}
\|v\|_{W_{\beta,2}}^2
=
\sum_{\Lambda}\ \sum_{(\mu,\varpi)\in\mathcal J(\Lambda)}
\Lambda^\beta \|v_{\Lambda,\mu,\varpi}\|_{L^2}^2,
\end{equation}
and
\begin{equation}
\label{eq:block_weight_prop_final_corrected}
\|v_{\Lambda,\mu,\varpi}\|_{W_{2\ell+q_0+k+2,2}}
=
\Lambda^{\ell+\frac{k+q_0}{2}+1}\|v_{\Lambda,\mu,\varpi}\|_{L^2}.
\end{equation}

For $N\ge1$, define the finite truncation
\begin{equation*}
v^{(N)}
:=
\sum_{\Lambda\le N}\ \sum_{(\mu,\varpi)\in\mathcal J(\Lambda)}
v_{\Lambda,\mu,\varpi}.
\end{equation*}
Since $1+\mathfrak L_\Delta$ has discrete spectrum and each $E_\Lambda(1+\mathfrak L_\Delta)$ is finite-dimensional,
this is a finite sum.

By Lemma~\ref{lem:uniform_upper_bound_R} and
\eqref{eq:block_weight_prop_final_corrected},
\begin{align}
\|\mathcal{R}_\ell(t,v^{(N)},\varphi)\|_{W_{k,2}}
&\le
\sum_{\Lambda\le N}\ \sum_{(\mu,\varpi)\in\mathcal J(\Lambda)}
\|\mathcal{R}_\ell(t,v_{\Lambda,\mu,\varpi},\varphi)\|_{W_{k,2}}
\nonumber\\
&\ll_{\ell}
(1+t)\,e^{(1-\ell)t}\,
\|\varphi\|_{W_{\beta,1}}
\sum_{\Lambda\le N}\ \sum_{(\mu,\varpi)\in\mathcal J(\Lambda)}
\Lambda^{\ell+\frac{k+q_0}{2}+1}\|v_{\Lambda,\mu,\varpi}\|_{L^2}.
\label{eq:truncation_bound_prop_final_corrected}
\end{align}

Apply Cauchy-Schwarz inequality:
\begin{align}
&\sum_{\Lambda\le N}\ \sum_{(\mu,\varpi)\in\mathcal J(\Lambda)}
\Lambda^{\ell+\frac{k+q_0}{2}+1}\|v_{\Lambda,\mu,\varpi}\|_{L^2}
\nonumber=
\sum_{\Lambda\le N}\ \sum_{(\mu,\varpi)\in\mathcal J(\Lambda)}
\Lambda^{(2\ell+q_0+k+2-\beta)/2}\Lambda^{\beta/2}\|v_{\Lambda,\mu,\varpi}\|_{L^2}
\nonumber\\
&\qquad\le
\Biggl(
\sum_{\Lambda\le N}\ \sum_{(\mu,\varpi)\in\mathcal J(\Lambda)}
\Lambda^\beta\|v_{\Lambda,\mu,\varpi}\|_{L^2}^2
\Biggr)^{1/2}
\Biggl(
\sum_{\Lambda\le N}\ \sum_{(\mu,\varpi)\in\mathcal J(\Lambda)}
\Lambda^{2\ell+q_0+k+2-\beta}
\Biggr)^{1/2}.
\label{eq:CS_prop_final_corrected}
\end{align}
By \eqref{eq:sobolev_decomp_prop_final_corrected}, the first factor is bounded by $\|v\|_{W_{\beta,2}}.$

Since $\beta>\beta_1$, 
we have $2\ell+q_0+k+2-\beta< -\frac{d_{Y}+d_{K}}{2}.$
Therefore Corollary~\ref{cor:block_summability_for_A} yields
\begin{equation}
\label{eq:R_block_sum_simplified_corrected}
\sum_{\Lambda}\ \sum_{(\mu,\varpi)\in\mathcal J(\Lambda)}
\Lambda^{2\ell+q_0+k+2-\beta}<\infty.
\end{equation}
Substituting \eqref{eq:R_block_sum_simplified_corrected} into
\eqref{eq:CS_prop_final_corrected}, and then into
\eqref{eq:truncation_bound_prop_final_corrected}, we obtain the uniform estimate
\begin{equation}
\label{eq:R_uniform_trunc_simplified_corrected}
\|\mathcal{R}_\ell(t,v^{(N)},\varphi)\|_{W_{k,2}}
\ll_{\ell,\Gamma,\beta}
(1+t)\,e^{(1-\ell)t}\,
\|v\|_{W_{\beta,2}}\,
\|\varphi\|_{W_{\beta,1}}.
\end{equation}

Now let $M>N$. Applying the same argument to
\begin{equation*}
v^{(M)}-v^{(N)}
=
\sum_{N<\Lambda\le M}\ \sum_{(\mu,\varpi)\in\mathcal J(\Lambda)}
v_{\Lambda,\mu,\varpi},
\end{equation*}
we get
\begin{align*}
\|\mathcal{R}_\ell(t,v^{(M)}-v^{(N)},\ &\varphi)\|_{W_{k,2}}
\ll
(1+t)\,e^{(1-\ell)t}\,
\|\varphi\|_{W_{\beta,1}} \\
&\cdot\Biggl(
\sum_{N<\Lambda\le M}\ \sum_{(\mu,\varpi)\in\mathcal J(\Lambda)}
\Lambda^\beta\|v_{\Lambda,\mu,\varpi}\|_{L^2}^2
\Biggr)^{1/2}
\Biggl(
\sum_{N<\Lambda\le M}\ \sum_{(\mu,\varpi)\in\mathcal J(\Lambda)}
\Lambda^{2\ell+q_0+k+2-\beta}
\Biggr)^{1/2}.
\end{align*}
By \eqref{eq:sobolev_decomp_prop_final_corrected}, the first factor tends to $0$
as $M,N\to\infty$, and by \eqref{eq:R_block_sum_simplified_corrected}, the second
factor is bounded. Hence $\mathcal{R}_\ell(t,v^{(N)},\varphi)$
is a Cauchy sequence in $W_{k,2}$.
Since $W_{k,2}$ is complete, it converges to some limit, which we denote by
$\mathcal{R}_\ell(t,v,\varphi)$.
Passing to the limit in \eqref{eq:R_uniform_trunc_simplified_corrected} yields
\begin{equation*}
\|\mathcal{R}_\ell(t,v,\varphi)\|_{W_{k,2}}
\ll_{\ell,\Gamma,\beta}
(1+t)\,e^{(1-\ell)t}\,
\|v\|_{W_{\beta,2}}\,
\|\varphi\|_{W_{\beta,1}}.
\end{equation*}
\end{proof}

\begin{defi}
For any fixed $\ell>\frac{n-1}{2}+\nu(\Gamma)$, 
we define the total main contribution by
\begin{equation*}
\Psi_{\ell}(t,v,\varphi)
=
\sum_{(\mu,\upvarpi)}
\Psi_{\ell}(t,v_{\mu,\upvarpi},\varphi),
\end{equation*}
where $\Psi_{\ell}(t,v_{\mu,\upvarpi},\varphi)$ is defined in Definition~\ref{defi:Psi_joint_engenvector}.
\end{defi}

Then by analogy with Proposition~\ref{prop:convergence_of_tilde_R},  the sum $\Psi_{\ell}(t,v,\varphi)$ is convergent in $W_{k, 2}(X)$:
\begin{prop}\label{coro:convergence_of_Psi}
Let $\ell>\frac{n-1}{2}+\nu(\Gamma)$.
Assume the same conditions as in Proposition~\ref{prop:convergence_of_tilde_R}. Then one has
\begin{equation*}
\|\Psi_{\ell}(t,v,\varphi)\|_{W_{k,2}}
\ll_{\ell,\Gamma,\beta} (1+t)\,e^{(\frac{1-n}{2}+\nu(\Gamma))t}\,
\|v\|_{W_{\beta,2}}\,
\|\varphi\|_{W_{\beta,1}}.
\end{equation*}
\end{prop}
The proof is identical to that of Proposition~\ref{prop:convergence_of_tilde_R},
using Lemma~\ref{lem:uniform_upper_bound_of_psi} instead of
Lemma~\ref{lem:uniform_upper_bound_R}.

Combining Proposition~\ref{prop:convergence_of_tilde_R}
and Proposition ~\ref{coro:convergence_of_Psi}, one has
\begin{equation*}
I(t,v,\varphi)=\Psi_{\ell}(t,v,\varphi)+\mathcal{R}_{\ell}(t,v,\varphi),
\end{equation*}
which converges in  $W_{k,2}(X).$
\subsection{Main Theorem}
As discussed in Section~\ref{section:asymptotic_expansion}, we derive the asymptotic expansion~\eqref{equation:Expansion_of_I_sec5} for $I(t,v,\varphi)$ and provide the uniform estimate of the coefficients.

Assembling the preceding results, we obtain the following main theorem.

\begin{thm}\label{theorem:main_vector}
Recall the constants $q_0$, $k$, $\ell$, $d_{Y}$, and $\beta$ appearing in Proposition~\ref{prop:convergence_of_tilde_R}. Recall that $X = \Gamma \backslash G$ is compact.
Then for any vector $v \in W_{k, 2}(X)$ and $\varphi\in C^{\infty}(K)$, there exists $C_m^{ui}( v_{\mu, \upvarpi},\varphi) \in W_{k,2}(X), ui \in \{+,-,P\}$, such that for all $t > 0$, one has
\begin{align*}
I(t,v,\varphi)
=&\;
\underset{(\mu, \upvarpi)}{\sum}
\Biggl(
\sum_{m=0}^{2\ell}
C_m^+( v_{\mu, \upvarpi},\varphi) \,e^{(\lambda_+-m) t}
+
\sum_{m=0}^{2\ell}
\Bigl(
C_m^-( v_{\mu, \upvarpi},\varphi)
+
C_m^P( v_{\mu, \upvarpi},\varphi)\,t
\Bigr)e^{(\lambda_- -m)t}
\Biggr)
\\
&
+
\mathcal{R}_\ell(t,v,\varphi),
\end{align*}
where the series converges in $W_{k, 2}(X)$, and 
\begin{equation*}
\|\mathcal{R}_\ell(t,v,\varphi)\|_{W_{k,2}}
\ll_{\ell,\Gamma,\beta}
(1+t)\,e^{(1-\ell)t}\,
\|v\|_{W_{\beta,2}}\,
\|\varphi\|_{W_{\beta,1}}.
\end{equation*}
Moreover,  $C_m^+( v_{\mu, \upvarpi},\varphi)=0$ whenever
$\sqrt{\mathcal D(\mu, \upvarpi)}> \nu(\Gamma)$.
\end{thm}

\begin{proof}
 Since the averaging
operator $I(t,\cdot,\varphi)$ is linear and continuous in $L^{2}(X)$, we obtain
\begin{equation*}
\begin{aligned}
    I(t,v,\varphi)=\underset{(\mu, \upvarpi)}{\sum} I(t, v_{\mu, \upvarpi},\varphi)
    =\underset{(\mu, \upvarpi)}{\sum} \Psi(t, v_{\mu, \upvarpi},\varphi)+ \underset{(\mu, \upvarpi)}{\sum}R_{\ell}(t, v_{\mu, \upvarpi},\varphi).
\end{aligned}
\end{equation*}
By Proposition~\ref{prop:convergence_of_tilde_R} and~\ref{coro:convergence_of_Psi}, both sum converges in $W_{k, 2}(X)$.
For each fixed $\mathcal{D}(\mu, \upvarpi)$, recall the asymptotic expansion~\eqref{equation:Expansion_of_I_sec5} for $I(t,v,\varphi)$:
\begin{align*}
I(t, v_{\mu, \upvarpi},\varphi)
=
\sum_{m=0}^{2\ell}
C_m^{+}( v_{\mu, \upvarpi},\varphi)e^{(\lambda_+-m)t}
+
\sum_{m=0}^{2\ell}
\bigl(
C_m^{-}( v_{\mu, \upvarpi},\varphi)
+
C_m^{P}( v_{\mu, \upvarpi},\varphi)\,t
\bigr)e^{(\lambda_-m)t}
+
\mathcal{R}_\ell(t, v_{\mu, \upvarpi},\varphi).
\end{align*}
And as details in Corollary~\ref{coro:Sobolev_regularity_of_polynomial_exponential_coefficients}, we know that the coefficients $C_m^{ui}( v_{\mu, \upvarpi},\varphi) \in W_{k,2}(X), ui \in \{+,-,P\}$.

Proposition~\ref{prop:convergence_of_tilde_R}
and Proposition ~\ref{coro:convergence_of_Psi} yields the following expansion:
\begin{align*}
I(t,v,\varphi)
= &\;
\underset{(\mu, \upvarpi)}{\sum}
\Biggl(
\sum_{m=0}^{2\ell} C_m^{+}( v_{\mu, \upvarpi},\varphi)e^{(\lambda_+-m)t}
+
\sum_{m=0}^{2\ell}
\bigl(
C_m^{-}( v_{\mu, \upvarpi},\varphi)
+
C_m^{P}( v_{\mu, \upvarpi},\varphi)\,t
\bigr)e^{(\lambda_-m)t}
\Biggr)
\\
&\;+
\mathcal{R}_\ell(t,v,\varphi),
\end{align*}
where 
\begin{equation*}
\|\mathcal{R}_\ell(t,v,\varphi)\|_{W_{k,2}}
\ll_{\ell,\Gamma,\beta}
(1+t)\,e^{(1-\ell)t}\,
\|v\|_{W_{\beta,2}}\,
\|\varphi\|_{W_{\beta,1}}.
\end{equation*}

Recall the discussion about matrix coefficient filter in the beginning of Section~\ref{subsection:Third_real_part}, for $\lambda_+ >\frac{1-n}{2}+\nu(\Gamma)$, the corresponding coefficients vanishes. 
Finally, the condition $\ell>n+1$ guarantees whenever one expects a meaningful expansion.
\end{proof}

\newpage
\bibliographystyle{plain} 
\bibliography{dynamics} 

@article {ernestgenerallorentzgroup,
    AUTHOR = {Thieleker, Ernest A.},
     TITLE = {The unitary representations of the generalized {L}orentz
              groups},
   JOURNAL = {Trans. Amer. Math. Soc.},
  FJOURNAL = {Transactions of the American Mathematical Society},
    VOLUME = {199},
      YEAR = {1974},
     PAGES = {327--367},
      ISSN = {0002-9947,1088-6850},
   MRCLASS = {22E43},
  MRNUMBER = {379754},
MRREVIEWER = {Kenneth\ Johnson},
       DOI = {10.2307/1996891},
       URL = {https://doi.org/10.2307/1996891},
}

@article {ernestquasisimpleirredrepnt,
    AUTHOR = {Thieleker, Ernest},
     TITLE = {On the quasi-simple irreducible representations of the
              {L}orentz groups},
   JOURNAL = {Trans. Amer. Math. Soc.},
  FJOURNAL = {Transactions of the American Mathematical Society},
    VOLUME = {179},
      YEAR = {1973},
     PAGES = {465--505},
      ISSN = {0002-9947,1088-6850},
   MRCLASS = {22E43},
  MRNUMBER = {325856},
MRREVIEWER = {H.\ Leptin},
       DOI = {10.2307/1996515},
       URL = {https://doi.org/10.2307/1996515},
}

@article{Corso2022LargeHC,
  title={Large hyperbolic circles},
  author={Emilio Corso and Davide Ravotti},
  journal={Bollettino dell'Unione Matematica Italiana},
  year={2022},
  url={https://api.semanticscholar.org/CorpusID:251594629}
}

@book {knapprepntofsemisimple,
    AUTHOR = {Knapp, Anthony W.},
     TITLE = {Representation theory of semisimple groups},
    SERIES = {Princeton Landmarks in Mathematics},
      NOTE = {An overview based on examples,
              Reprint of the 1986 original},
 PUBLISHER = {Princeton University Press, Princeton, NJ},
      YEAR = {2001},
     PAGES = {xx+773},
      ISBN = {0-691-09089-0},
   MRCLASS = {22E46 (22-01 22E30)},
  MRNUMBER = {1880691},
}

@Inbook{Margulis2004,
author="Margulis, G. A.",
title="On Some Aspects of the Theory of Anosov Systems",
bookTitle="On Some Aspects of the Theory of Anosov Systems: With a Survey by Richard Sharp: Periodic Orbits of Hyperbolic Flows",
year="2004",
publisher="Springer Berlin Heidelberg",
address="Berlin, Heidelberg",
pages="1--71",
isbn="978-3-662-09070-1",
doi="10.1007/978-3-662-09070-1_1",
url="https://doi.org/10.1007/978-3-662-09070-1_1"
}

@article {Burgerhorocyclegeofinite,
    AUTHOR = {Burger, Marc},
     TITLE = {Horocycle flow on geometrically finite surfaces},
   JOURNAL = {Duke Math. J.},
  FJOURNAL = {Duke Mathematical Journal},
    VOLUME = {61},
      YEAR = {1990},
    NUMBER = {3},
     PAGES = {779--803},
      ISSN = {0012-7094,1547-7398},
   MRCLASS = {58F17},
  MRNUMBER = {1084459},
MRREVIEWER = {Robert\ Brooks},
       DOI = {10.1215/S0012-7094-90-06129-0},
       URL = {https://doi.org/10.1215/S0012-7094-90-06129-0},
}

@article {edwardsrateofexpandinghorospheres,
    AUTHOR = {Edwards, Samuel C.},
     TITLE = {On the rate of equidistribution of expanding translates of
              horospheres in {$\Gamma\backslash G$}},
   JOURNAL = {Comment. Math. Helv.},
  FJOURNAL = {Commentarii Mathematici Helvetici. A Journal of the Swiss
              Mathematical Society},
    VOLUME = {96},
      YEAR = {2021},
    NUMBER = {2},
     PAGES = {275--337},
      ISSN = {0010-2571,1420-8946},
   MRCLASS = {37A17 (22E40)},
  MRNUMBER = {4277274},
MRREVIEWER = {Thomas\ Ward},
       DOI = {10.4171/cmh/513},
       URL = {https://doi.org/10.4171/cmh/513},
}

@article {EskinMcMullenmixingcouting,
    AUTHOR = {Eskin, Alex and McMullen, Curt},
     TITLE = {Mixing, counting, and equidistribution in {L}ie groups},
   JOURNAL = {Duke Math. J.},
  FJOURNAL = {Duke Mathematical Journal},
    VOLUME = {71},
      YEAR = {1993},
    NUMBER = {1},
     PAGES = {181--209},
      ISSN = {0012-7094,1547-7398},
   MRCLASS = {22E40 (57S30 58F17)},
  MRNUMBER = {1230290},
MRREVIEWER = {Nimish\ A.\ Shah},
       DOI = {10.1215/S0012-7094-93-07108-6},
       URL = {https://doi.org/10.1215/S0012-7094-93-07108-6},
}

@unpublished{Lutsko2022SpectralHorospherical,
  author  = {Christopher Lutsko},
  title   = {An Abstract Spectral Approach to Horospherical Equidistribution},
  year    = {2022},
  note    = {Preprint, available at \url{https://chrislutsko.com/files/Effective_Horocycles.pdf}}
}

@phdthesis{edwards2018symmetricsubgroup,
  author       = {Edwards, Samuel},
  title        = {On the equidistribution of translates of orbits of symmetric subgroups in $\Gamma \backslash G$},
  school       = {Uppsala University},
  year         = {2018},
  type         = {Ph.D. Thesis},
  note         = {Chapter of Ph.D. Thesis},
  url          = {https://drive.google.com/file/d/1IZ0XLMk5pHSR6WgCfROfJj72izyxU7Rk/view}
}

@article {Davidehorocycle2020,
    AUTHOR = {Ravotti, Davide},
     TITLE = {Quantitative equidistribution of horocycle push-forwards of
              transverse arcs},
   JOURNAL = {Enseign. Math.},
  FJOURNAL = {L'Enseignement Math\'ematique},
    VOLUME = {66},
      YEAR = {2020},
    NUMBER = {1-2},
     PAGES = {135--150},
      ISSN = {0013-8584,2309-4672},
   MRCLASS = {37A17 (37A25 37D40)},
  MRNUMBER = {4162285},
       DOI = {10.4171/LEM/66-1/2-7},
       URL = {https://doi.org/10.4171/LEM/66-1/2-7},
}

@article{SamuelC.Edwards2017JournalofModernDynamics,
title = {On the rate of equidistribution of expanding horospheres in finite-volume quotients of {$\mathrm{SL}(2,\mathbb C)$}},
journal = {Journal of Modern Dynamics},
volume = {11},
number = {0},
pages = {155-188},
year = {2017},
issn = {1930-5311},
doi = {10.3934/jmd.2017008},
url = {https://www.aimsciences.org/article/id/fc8d1219-53ca-438c-8fdb-2a683ab4860a},
author = {Samuel C. Edwards}
}

@article{Strombergsson2013DeviationHorocycle,
  author  = {Str{\"o}mbergsson, Andreas},
  title   = {On the deviation of ergodic averages for horocycle flows},
  journal = {Journal of Modern Dynamics},
  volume  = {7},
  number  = {2},
  pages   = {291--328},
  year    = {2013},
  doi     = {10.3934/jmd.2013.7.291},
  url     = {https://www.aimsciences.org/article/id/d1b864ba-dba0-41f7-9107-187720d8e6d9}
}

@article{TROMBI197883,
title = {Asymptotic expansions of matrix coefficients: The real rank one case},
journal = {Journal of Functional Analysis},
volume = {30},
number = {1},
pages = {83-105},
year = {1978},
issn = {0022-1236},
doi = {https://doi.org/10.1016/0022-1236(78)90057-5},
url = {https://www.sciencedirect.com/science/article/pii/0022123678900575},
author = {P.C Trombi}
}

@article{Miatelloweyllaw,
  issn = {00029947},
  url = {http://www.jstor.org/stable/1999874},
  author = {Roberto J. Miatello},
  journal = {Transactions of the American Mathematical Society},
  number = {1},
  pages = {1--33},
  publisher = {American Mathematical Society},
  title = {The Minakshisundaram-Pleijel Coefficients for the Vector Valued Heat Kernel on Compact Locally Symmetric Spaces of Negative Curvature},
  urldate = {2026-06-01},
  volume = {260},
  year = {1980},
}

@article{Flaminioforni,
author = {Livio Flaminio and Giovanni Forni},
title = {{Invariant distributions and time averages for horocycle flows}},
volume = {119},
journal = {Duke Mathematical Journal},
number = {3},
publisher = {Duke University Press},
pages = {465 -- 526},
year = {2003},
doi = {10.1215/S0012-7094-03-11932-8},
URL = {https://doi.org/10.1215/S0012-7094-03-11932-8}
}
\end{document}